\documentclass[a4paper,11pt]{article}

\usepackage[
    a4paper,
    top=1.05in,
    bottom=1.05in,
    left=1.00in,
    right=1.00in
]{geometry}
\usepackage[T1]{fontenc}

\usepackage{amsmath,amsthm,mathtools}
\usepackage{amssymb}
\usepackage{newtxtext,newtxmath}
\usepackage{microtype}
\microtypesetup{protrusion=true,expansion=true}
\usepackage[new]{old-arrows}

\usepackage{mathrsfs}
\usepackage{slashed}
\usepackage{upgreek}
\allowdisplaybreaks[4]
\usepackage{enumitem}
\setlist{itemsep=0.25em, topsep=0.45em, parsep=0pt}
\usepackage{scalerel}
\usepackage[usestackEOL]{stackengine}
\usepackage[thinc]{esdiff}

\usepackage[all]{xy}
\usepackage{graphicx,psfrag}
\usepackage{pstool}
\usepackage{tikz-cd}
\usepackage[font=small,labelfont=bf,labelsep=period]{caption}

\usepackage{titlesec}
\titleformat{\section}{\normalfont\Large\bfseries}{\thesection}{0.75em}{}
\titleformat{\subsection}{\normalfont\large\bfseries}{\thesubsection}{0.75em}{}
\titlespacing*{\section}{0pt}{2.6ex plus 0.8ex minus 0.2ex}{1.2ex plus 0.2ex}
\titlespacing*{\subsection}{0pt}{2.0ex plus 0.6ex minus 0.2ex}{0.9ex plus 0.2ex}

\usepackage{fancyhdr}
\fancypagestyle{plain}{%
  \fancyhf{}%
  \fancyfoot[C]{\thepage}%
}

\usepackage[dvipsnames]{xcolor}
\definecolor{linkblue}{RGB}{20,67,128}
\definecolor{citegreen}{RGB}{25,105,80}
\definecolor{urlpurple}{RGB}{100,55,125}
\usepackage[colorlinks=true,
            linkcolor=linkblue,
            citecolor=citegreen,
            urlcolor=urlpurple,
            plainpages=false]{hyperref}
\usepackage{bookmark}

\usepackage{array,tabularx,booktabs}

\usepackage{setspace}

\usepackage[titletoc,title]{appendix}
\usepackage{titletoc}

\usepackage[capitalize,nameinlink,noabbrev]{cleveref}

\newcommand{\GHconvtext}{\mathrm{Gromov\text{--}Hausdorff}}
\newcommand{\pGHconvtext}{\mathrm{pointed\text{--}Gromov\text{--}Hausdorff}}

 \newcommand{\R}{\ensuremath{\mathbb{R}}}
 
 \def\d{\mathrm{d}}

 \def\RR{\mathcal{R}}

 \def\scal{\mathrm{R}}
 
 \def\loc{\mathrm{loc}}
  \def\diam{\mathrm{diam}}

 \newcommand{\RP}{\ensuremath{\mathbb{RP}}}

 \newcommand{\ba}{\begin{align*}}
 \newcommand{\ea}{\end{align*}}
 \newcommand{\na}{\nabla}

\newcommand{\lc}{\left(}
\newcommand{\rc}{\right)}
 
\newcommand{\ep}{\epsilon}

\def\vol{\mathrm{Vol}}

\newcommand{\tr}{\text{tr}}

\newcommand{\Rm}{\ensuremath{\mathrm{Rm}}}
\newcommand{\Ric}{\ensuremath{\mathrm{Ric}}}

\renewcommand{\t}{\mathfrak{t}}

\def\MS{\mathcal{S}}

\DeclareMathOperator{\GH}{GH}

 \makeatletter
 \def\ExtendSymbol#1#2#3#4#5{\ext@arrow 0099{\arrowfill@#1#2#3}{#4}{#5}}

 \makeatother

\def\aint{\,\ThisStyle{\ensurestackMath{%
  \stackinset{c}{.2\LMpt}{c}{.5\LMpt}{\SavedStyle-}{\SavedStyle\phantom{\int}}}%
  \setbox0=\hbox{$\SavedStyle\int\,$}\kern-\wd0}\int}

\DeclarePairedDelimiter\abs{\lvert}{\rvert}%
\makeatletter
\let\oldabs\abs
\def\abs{\@ifstar{\oldabs}{\oldabs*}}
\makeatother

\numberwithin{equation}{section}

\theoremstyle{plain}
\newtheorem{thm}{Theorem}[section]

\theoremstyle{definition}

\theoremstyle{remark}

\titlecontents{part}[0pt]  
{\addvspace{10pt}\bfseries} 
{\makebox[3em][l]{PART-\thecontentslabel}}      
{}                          
{}                          

\providecommand{\C}{\ensuremath{\mathbb{C}}}

\newcommand{\ii}{\ensuremath{\sqrt{-1}}}
\newcommand{\pp}{\bar\partial}
\DeclareMathOperator{\nul}{Null}

\DeclareMathOperator{\sing}{sing}
\DeclareMathOperator{\reg}{reg}
\DeclareMathOperator{\lie}{Lie}

\theoremstyle{plain}
\newtheorem{theorem}[thm]{Theorem}
\newtheorem{corollary}[thm]{Corollary}
\newtheorem{proposition}[thm]{Proposition}
\newtheorem{lemma}[thm]{Lemma}
\newtheorem{conjecture}[thm]{Conjecture}

\theoremstyle{definition}
\newtheorem{definition}[thm]{Definition}

\theoremstyle{remark}
\newtheorem{remark}[thm]{Remark}

\title{Gromov--Hausdorff Limits of Noncollapsed K\"ahler--Ricci Flows and the Geometry of Ricci Shrinkers}
\author{Yu Li \quad and \quad Junsheng Zhang}
\date{}

\begin{document}
\maketitle
\begin{abstract}
We study the metric geometry of finite-time singularities of volume-noncollapsed K\"ahler--Ricci flows and the geometry at infinity of Ricci shrinkers. For a compact volume-noncollapsed K\"ahler--Ricci flow approaching its first singular time, we prove that the time slices converge to a unique Gromov--Hausdorff limit, canonically identified with the time-zero slice of the Ricci-flow spacetime completion. Its regular part is identified with the complement of the null locus, and we show that the singular set has Hausdorff codimension at least four.

For K\"ahler--Ricci shrinkers, we relate the geometry at infinity to the associated Fano fibration. We prove local smooth convergence to a K\"ahler cone metric on the locus where the fibration is biholomorphic, and characterize asymptotic conicality by biholomorphicity outside a compact set.

For Ricci shrinkers with bounded scalar curvature, we show that local compactness of the time-zero slice implies pointed Gromov--Hausdorff convergence of the full self-similar flow to that slice. As applications, we prove uniqueness of the tangent space at infinity for K\"ahler--Ricci shrinkers with bounded scalar curvature and Euclidean volume growth whose soliton vector field generates an $S^1$-action, and for all four-dimensional Ricci shrinkers with bounded scalar curvature.
\end{abstract}

\tableofcontents

\section{Introduction}

Ricci flow provides a canonical way to deform a Riemannian metric, and the
formation of singularities makes it necessary to understand geometric limits
in the absence of uniform curvature bounds. There are two complementary
aspects of this problem. One may rescale around regions of high curvature and
study tangent flows, which describe the local models of singularity formation;
or one may keep the original scale and ask whether the time slices themselves
converge to a canonical metric space at the singular time. The second question
is global and is generally more delicate: without a uniform lower Ricci
curvature bound, distances at different times need not be directly comparable,
and even the uniqueness of a sequential Gromov--Hausdorff limit is not
automatic.

This paper studies this global metric problem for noncollapsed
K\"ahler--Ricci flows and, in parallel, the corresponding problem at infinity
for Ricci shrinkers, that is, complete shrinking gradient Ricci solitons. The two principal tools are the time-zero slice
of a Ricci-flow spacetime completion \cite{FL1} and the polarized Fano fibration structure for K\"ahler--Ricci shrinkers \cite{SunZhang}.

\medskip
\noindent\textbf{The singular-time limit of a noncollapsed
K\"ahler--Ricci flow.}

Let $(X^n,(\omega_t)_{t\in[-1,0)})$ be a K\"ahler--Ricci flow on a compact
complex manifold of complex dimension $n$, satisfying
\[
    \partial_t\omega_t=-\Ric(\omega_t),
\]
and developing its first singularity at $t=0$. We assume that the flow is
volume-noncollapsed in the sense that
\begin{equation}\label{eq--global-volume-non-collapsing}
    \lim_{t\nearrow0}|X|_{\omega_t}>0.
\end{equation}
This condition has a precise algebro-geometric interpretation: it is
equivalent to
\begin{equation}\label{eq:numerical-noncollapsing}
    \text{the class $\alpha=[\omega_{-1}]+K_X$ being big and nef.}
\end{equation}
The complex geometry already determines a large regular region of the
singular-time limit. By Collins--Tosatti \cite{CT},
\begin{equation}\label{eq:smooth-convergence}
    \omega_t
    \xrightarrow[t\nearrow0]{C^\infty_{\loc}(X\setminus\nul(\alpha))}
    \omega_0,
\end{equation}
where $\omega_0$ is a smooth K\"ahler metric on
$X\setminus\nul(\alpha)$ and extends to a positive $(1,1)$-current on $X$.
Here $\nul(\alpha)$ is the null locus of $\alpha$, namely the union of the
positive-dimensional irreducible analytic subvarieties $V\subset X$ satisfying
\[
    \int_V\alpha^{\dim V}=0.
\]
We write
\[
    X^\circ:=X\setminus\nul(\alpha).
\]

Recent diameter and noncollapsing estimates
\cite{GPSS1,GPSS-ii,GPSS2,guedj-to,vu2024} imply that
\begin{equation}\label{eq:diameter}
    \diam(X,\omega_t)\le C,
    \qquad t<0,
\end{equation}
and that, for every $\ep>0$, $r\in(0,1]$, and $x\in X$,
\begin{equation}\label{eq:volume-noncollapsing}
    |B_{\omega_t}(x,r)|_{\omega_t}
    \ge C_\ep^{-1}r^{2n+\ep}.
\end{equation}
In particular, the family of time slices is Gromov--Hausdorff precompact. If
$t_j\nearrow0$, after passing to a subsequence one obtains
\[
    (X,d_{\omega_{t_j}})
    \xrightarrow[j\to\infty]{\GHconvtext}
    (X_{\GH},d_{\GH}).
\]
A priori, however, the limit may depend on the sequence $\{t_j\}$. The local
smooth convergence in \eqref{eq:smooth-convergence} identifies an open part of
any such limit, but it does not by itself determine the metric completion or
rule out additional points arising from the degenerating region.

A canonical candidate for the missing singular-time space is supplied by the
spacetime completion introduced in \cite{FL1}. More precisely, one equips
$X\times[-0.99,0)$ with a spacetime distance $d^*$ defined in terms of
conjugate heat kernel measures and the $W_1$-Wasserstein distance. Its metric
completion is a parabolic metric space
\[
    (Z,d_Z,\mathfrak t),
\]
and the completion adds points only at time zero. The time-zero slice
\[
    Z_0:=\mathfrak t^{-1}(0)
\]
carries a canonical intrinsic distance $d^Z_0$, again defined using conjugate
heat kernel measures. The resulting space $(Z_0,d^Z_0)$ is isometric to
Bamler's asymptotic boundary \cite[Section~2.6]{bamler3}; see \cite[Lemma~9.1]{FL1}.

Our first theorem shows that the time-zero slice $(Z_0,d^Z_0)$ is the unique
Gromov--Hausdorff limit of $(X,d_{\omega_t})$ as $t\nearrow0$ and
identifies the smooth locus of the limit. A related convergence result for three-dimensional singular Ricci flows was proved in \cite{Li26}.

\begin{theorem}[Canonical singular-time limit]\label{thm--all}
There exist maps
\begin{equation}\label{lem--preliminary}
    X^\circ
    \overset{\Phi}{\longhookrightarrow}
    Z_0
    \overset{h}{\longrightarrow}
    X_{\GH},
\end{equation}
such that $\Phi: (X^{\circ},d_{\omega_0}) \longrightarrow (Z_0^{\reg}, d^Z_0)$ is bijective and $1$-Lipschitz, and
\(
    h:(Z_0,d^Z_0)\longrightarrow(X_{\GH},d_{\GH}) \text{ is an isometry.}
\)
Consequently,
\[
    (X,d_{\omega_t})
    \xrightarrow[t\nearrow0]{\GHconvtext}
    (Z_0,d^Z_0),
\]
and the singular-time Gromov--Hausdorff limit is unique.
\end{theorem}

Let us describe the main points of the proof. For $x\in X^\circ$, the local
curvature bound following from \eqref{eq:smooth-convergence} implies that the
worldline $(x,t)$ has a limit in the spacetime completion. This defines the
map $\Phi:X^\circ\to Z_0$, and spacetime regularity shows that its image is
contained in $Z_0^{\reg}$ and $\Phi$ is injective. Conversely, if $z\in Z_0^{\reg}$, the regular
spacetime flow through $z$ determines a fixed point $x\in X$ at all nearby
negative times. A local two-sided metric bound near $x$, together with the
cohomological characterization of the null locus, rules out
$x\in\nul(\alpha)$. This proves that $\Phi$ is a bijection from $X^{\circ}$ onto $Z_0^{\reg}$. The $1$-Lipschitz property of $\Phi$ is standard.

To compare $Z_0$ with an arbitrary sequential Gromov--Hausdorff limit, we
choose a countable dense subset of $Z_0$ and represent each of its points by
$H$-centers on the slices $(X,\omega_{t_j})$. Passing to limits defines the
map $h$. The definition of $d^Z_0$ and the monotonicity of the Wasserstein
distance imply that $h$ preserves distances. The essential issue is
surjectivity. If a point of $X_{\GH}$ were separated from $h(Z_0)$, the
volume lower bound \eqref{eq:volume-noncollapsing}, combined with the
quantitative curvature-radius estimate from \cite{FL1}, would produce a
regular point in a definite ball around it. The worldline of this regular
point determines a point of $Z_0$ whose image lies in the same ball, yielding
a contradiction. This argument applies to every time sequence and is the
mechanism behind the uniqueness assertion.

Once the metric limit has been identified, a second problem is to understand
its singular set. The structure theory of noncollapsed Ricci flow limit spaces
provides a quantitative stratification by the Euclidean splitting of tangent
flows \cite{FL1,fang-li2}. In the present K\"ahler setting, the potentially
largest singular stratum would arise from the non-static cylindrical tangent
flow
\(
\overline{\mathcal C}^{2n-2},
\)
that is, the standard shrinking product Ricci flow on $S^2\times\mathbb R^{2n-2}$.
The decisive K\"ahler input of this paper is that such a tangent flow cannot occur.

The reason is both complex-analytic and cohomological. Smooth convergence to
$S^2\times\mathbb R^{2n-2}$ produces a limiting parallel complex structure,
and the K\"ahler splitting forces the model to be biholomorphic to
$\mathbb P^1\times\mathbb C^{n-1}$. A deformation argument for holomorphic
$\mathbb P^1$'s then gives, on a nearby time slice of the original flow, a
local holomorphic fibration by rational curves with trivial normal bundle.
All fibers represent the same homology class. Comparing their areas with the
evolving K\"ahler class shows that the limiting class $\alpha$ vanishes on
this fiber class. This contradicts the fact that the null locus of $\alpha$ cannot contain an open subset. Thus, the
$(2n-2)$-splitting cylindrical model is excluded.

Combining this exclusion with the quantitative stratification and covering
estimates of \cite{FL1,fang-li2} yields the following codimension-four
statement.

\begin{theorem}[Codimension-four estimate]\label{thm:codimension-four}
The Hausdorff dimension of the singular set $Z_0^{\sing} \subset Z_0$ with respect to $d^Z_0$ satisfies
    \[
        \dim_{\mathscr H}(Z_0^{\sing})\le 2n-4.
    \]
\end{theorem}

\medskip
\noindent\textbf{K\"ahler--Ricci shrinkers and Fano fibrations.}

Ricci shrinkers arise as canonical blow-up models for finite-time
singularities of Ricci flow
\cite{EndersMullerTopping,Naber2010,bamler3,FL1}. At the same time, a
noncompact shrinker has its own global asymptotic problem: one would like to
understand the geometry of its blow-downs and, in particular, whether its
tangent space at infinity is unique. A shrinker generates a self-similar
Ricci flow on $(-\infty,0)$, so its geometry at infinity is encoded by the
behavior of this associated flow as $t\nearrow0$. This makes the time-zero
completion used above a natural tool in the noncompact setting as well.

For K\"ahler--Ricci shrinkers, there is an additional complex-algebraic
structure. Sun--Zhang \cite{SunZhang} associate a polarized Fano fibration with every K\"ahler--Ricci shrinker
\[
    \pi:X\longrightarrow Y,
\]
where $Y$ is an affine variety and $\pi$ is constructed from homogeneous
holomorphic functions for the soliton vector field. Using this polarized Fano fibration structure and extending techniques from compact K\"ahler--Ricci flows to the noncompact setting, we establish several structural results for K\"ahler--Ricci shrinkers without imposing curvature assumptions.

Let $X^\circ\subset X$
be the largest open subset on which $\pi$ is biholomorphic onto an open subset
of $Y$, and let
\[
    E:=X\setminus X^\circ
\]
be the exceptional set. Thus $X^\circ\ne\varnothing$ precisely when $\pi$ is
birational. We show in Proposition~\ref{prop:maximal-volume-growth-characterization}
that this is equivalent to maximal volume growth.
Let $\psi^t$, $t\in[-1,0)$, be the self-similar biholomorphisms generated by
$(2|t|)^{-1}\nabla f$, normalized by $\psi^{-1}=\mathrm{id}$, and set
\[
    \omega_t:=|t|(\psi^t)^*\omega.
\]
Then $\partial_t\omega_t=-\Ric(\omega_t)$. The first question is whether the
self-similar metrics develop a genuine cone structure on the part of $X$
seen by the Fano fibration.

\begin{theorem}[Cone limit on the isomorphism locus]
\label{thm:local-smooth-convergence}
The metrics $\omega_t$ converge to a K\"ahler cone metric $\omega_0$ in
$C^\infty_{\loc}(X^\circ)$ as $t\nearrow0$.
\end{theorem}

The key analytic issue is
to control the homogeneous holomorphic functions defining $\pi$ on a
noncompact shrinker. If $h$ has weight $\lambda$, then its real and imaginary
parts are eigenfunctions of the weighted Laplacian. We first prove weighted
$L^2$ integrability, then combine a local weighted mean-value estimate with a
maximum-principle barrier to obtain the sharp polynomial growth bound
\[
    |h(x)|\le C(1+\rho(x))^\lambda,
    \qquad
    \rho=\sqrt{f+n}.
\]
These growth estimates allow us to apply the maximum principle and obtain a global Schwarz-type lower bound for $\omega_t$:
\[
    \omega_t\ge C_K^{-1}\pi^*\omega_Y
    \quad\text{on }\pi^{-1}(K)
\]
for every compact set $K\Subset Y$; see
Proposition~\ref{prop:schwarz-lower-bound-general}. On $X^\circ$ the form
$\pi^*\omega_Y$ is positive definite. The scalar-curvature nonnegativity gives
an upper bound for the volume form $\omega_t^n$, and together these estimates
yield a local two-sided metric bound. Interior estimates for the
K\"ahler--Ricci flow then give smooth convergence. Finally, the limiting
rescaled potential satisfies
\[
    \nabla^2_{g_0}F_0=g_0,
    \qquad
    |\nabla F_0|_{g_0}^2=2F_0,
\]
which identifies $(g_0,\omega_0)$ as a K\"ahler cone metric.


The same estimate leads to a complex-analytic characterization of
asymptotic conicality.

\begin{theorem}[Asymptotically conical criterion]\label{thm:ac-criterion}
A noncompact K\"ahler--Ricci shrinker $(X,g,f)$ is asymptotically
conical if and only if its associated Fano fibration
$\pi:X\to Y$ is biholomorphic outside a compact set.
\end{theorem}

The forward implication was established in \cite{CSunD}. For the converse,
if $\pi$ is biholomorphic outside a compact set, choose a large regular level
set of the soliton potential which lies in $X^\circ$. Theorem~\ref{thm:local-smooth-convergence} gives a uniform curvature bound for the
self-similar flow on this level set up to time zero. Transporting this bound
outward by the soliton diffeomorphisms and using the quadratic growth of the
potential yields
\[
    |\Rm_g|(x)\le C\,d_g(p,x)^{-2}.
\]
The conical structure theorem in \cite{KW} then implies that
the shrinker is asymptotically conical.

Theorem~\ref{thm:ac-criterion} has several consequences. It gives, for
example, another proof of the asymptotic conicality of K\"ahler--Ricci
shrinker surfaces with maximal volume growth, a result previously established in
\cite{munteanu2015geometry,CCD,LW26}. More importantly, together with the
uniqueness theorem of Esparza \cite{esparza2025}, we prove in Corollary~\ref{cor:uniqueness-shrinker-structures} that any two K\"ahler--Ricci
shrinkers on the same complex manifold are equivalent by a biholomorphism as
soon as one of them is asymptotically conical.

We say that a complex manifold $X$ has no compact positive-dimensional
analytic subvarieties at infinity if there exists a compact subset
$K\Subset X$ such that $X\setminus K$ contains no compact
positive-dimensional analytic subset. This condition is automatic for the
underlying complex manifold of every asymptotically conical
K\"ahler--Ricci shrinker; see \cite{CSunD}. Since the fibers of a Fano fibration are
compact analytic subvarieties, this condition forces the fibration to be
finite outside a compact set; connectedness of the fibers and normality of
the base then imply that it is biholomorphic there. We consequently obtain
the following uniqueness theorem.

\begin{theorem}
\label{thm:special-uniqueness}
If a complex manifold $X$ contains no compact positive-dimensional analytic
subvarieties at infinity, then, up to biholomorphism, there is at most one
K\"ahler--Ricci shrinker on $X$. In particular, for each complex manifold
$X$ appearing in Table~\ref{tab:known-ac-shrinkers}, the listed
asymptotically conical shrinker is, up to biholomorphism, the unique
K\"ahler--Ricci shrinker on $X$.
\end{theorem}

For reference, Table~\ref{tab:known-ac-shrinkers} summarizes, to the best of
our knowledge, the currently known explicit families of smooth asymptotically conical K\"ahler--Ricci shrinkers. The rows are not
disjoint: several classical examples occur as special cases of the more
general constructions. In the table, $B$ is a compact K\"ahler--Einstein Fano manifold of
Fano index $i_B$, and $L\to B$ is a line bundle satisfying
$L^{\otimes i_B}\simeq K_B$, unless otherwise specified.

\begin{center}
\begin{minipage}{\textwidth}
\captionof{table}{Currently known explicit constructions of asymptotically conical K\"ahler--Ricci shrinkers.}
\label{tab:known-ac-shrinkers}
\scriptsize
\renewcommand{\arraystretch}{1.2}
\begin{tabularx}{\linewidth}
{@{}>{\raggedright\arraybackslash}p{0.20\linewidth}
>{\raggedright\arraybackslash}X
>{\raggedright\arraybackslash}p{0.15\linewidth}@{}}
\toprule
Construction & Underlying complex manifold  & Reference \\
\midrule
Gaussian shrinker
&
$X=\mathbb C^n$.
&
Classical
\\[0.3em]

Calabi ansatz on line bundles over Fano K\"ahler--Einstein manifolds
&
$X=\operatorname{Tot}(L^{\otimes m})$ with $1\leq m<i_B$.
For $B=\mathbb P^{n-1}$ and
$L=\mathcal O_{\mathbb P^{n-1}}(-1)$, this includes the
Feldman--Ilmanen--Knopf shrinkers on
$\operatorname{Tot}(\mathcal O_{\mathbb P^{n-1}}(-m))$,
$1\leq m<n$.
&
\cite{koiso1990,Cao96,FIK03,Yang12}
\\[0.3em]

Equal direct sums
&
$X=\operatorname{Tot}((L^{\otimes m})^{\oplus r})$ with $r\geq2$ and
$rm<i_B$.
&
\cite{Li10}
\\[0.3em]

Line bundle over products
&
$X=\operatorname{Tot}(\bigotimes_i\pi_i^*(L_i^{\otimes q_i}))$, with the $q_i$ satisfying suitable numerical constraints.
&
\cite{DancerWang11}
\\[0.3em]

Line bundle over toric Fano
&
$X=\operatorname{Tot}(L^{\otimes m})$, with $1\leq m<i_B$, where $B$ is a toric Fano manifold that need not admit a K\"ahler--Einstein metric.
&
\cite{FutakiWang11,Futaki21}
\\[0.3em]

Unequal direct sums
&
$X=\operatorname{Tot}(L^{\otimes m_1}\oplus L^{\otimes m_2})$ with
$m_1>m_2>0$ and $m_1+m_2<i_B$.
&
\cite{Cifarelli25}
\\
\bottomrule
\end{tabularx}
\end{minipage}
\end{center}

\medskip
\noindent\textbf{The time-zero slice of a Ricci shrinker.}

We next return to a general, not necessarily K\"ahler, Ricci shrinker
$(M^n,g,f)$ with bounded scalar curvature. We normalize it by
\[
    \Ric(g)+\nabla^2f=\frac12g,
    \qquad
    \scal+|\nabla f|^2=f,
\]
and consider the associated self-similar Ricci flow
$(M,(g_t)_{t<0})$. The heat kernel estimates of
\cite{liwang2020heat1,liwang2024heat} allow the spacetime-completion theory of
\cite{FL1} to be applied to this noncompact flow. We obtain a completion
$(Z,d_Z,\mathfrak t)$ and a pointed time-zero slice
\[
    (Z_0,d^Z_0,\bar p).
\]
The rescaled potential extends continuously to $Z_0$ and is given by the
canonical radial function
\[
    F(z)=\frac14\bigl(d^Z_0(z,\bar p)\bigr)^2.
\]
On the regular part $Z_0^{\reg}$ of $Z_0$, the soliton equation passes to the limit and
gives
\[
    \nabla^2_{g_0^Z}F=\frac12g_0^Z.
\]
Thus $(Z_0^{\reg},g_0^Z)$ is a smooth metric cone.
Self-similarity induces a one-parameter family of homotheties of $Z_0$, so
$Z_0$ is topologically a cone. Moreover, every worldline has a time-zero
limit, defining a continuous proper surjection
\[
    \Phi:M\longrightarrow Z_0.
\]
For fixed $x,y\in M$,
\[
    d^Z_0(\Phi(x),\Phi(y))
    =\lim_{t\nearrow0}d_{g_t}(x,y).
\]

A useful point is that local compactness of the canonical time-zero slice is
precisely enough to turn these pointwise distance limits into full pointed
Gromov--Hausdorff convergence.
\begin{theorem}\label{thm--locally compact criterion}
Let $(M^n,g,f)$ be a Ricci shrinker with bounded scalar curvature.  If $(Z_0,d^Z_0)$ is locally
compact, then
\[
    (M,d_{g_t},p)
    \xrightarrow[t\nearrow0]{\pGHconvtext}
    (Z_0,d^Z_0,\bar p).
\]
Consequently, if $(Z_0,d^Z_0)$ is locally compact, then it is the unique tangent space at infinity of $(M,g)$.
\end{theorem}
The proof uses the proper map $\Phi$. Closed $d^Z_0$-balls have compact
preimages in $(M,g)$; on each such preimage, the pointwise convergence of
distances is upgraded to uniform convergence by a compactness argument. The
potential estimate supplies uniform radial control and shows that the image
of a $g_t$-ball under $\Phi$ contains, and is contained in, a ball of almost
the same radius in $Z_0$. Thus $\Phi$ itself becomes a pointed
Gromov--Hausdorff approximation. By self-similarity, local compactness of
$Z_0$ also implies that the shrinker has a unique tangent space at infinity.
This reduces the metric uniqueness problem, under the bounded scalar
curvature assumption, to establishing local compactness of the canonical
slice.

For a K\"ahler--Ricci shrinker $X$, we compare the associated Fano fibration $\pi:X\to Y$ with the map $\Phi$ constructed above. If the scalar curvature is bounded,
then the restriction of $\Phi$ is biholomorphic from $X^\circ$ onto
$Z_0^{\reg}$. If the full Riemannian curvature has subquadratic growth, then the diameter of compact
fibers of the Fano fibration goes to zero in the metrics
$\omega_t$, allowing us to identify $Y$ with $Z_0$. This result gives a
concrete link between the affine variety $Y$ produced by algebraic geometry
and the canonical metric boundary produced by the Ricci flow.
\begin{theorem}\label{bounded curvature-homome}
   Let $(X,g)$ be a K\"ahler--Ricci shrinker with bounded scalar curvature. Then the following hold.
   \begin{enumerate}[label=\textnormal{(\arabic*)}]
       \item  The restriction $\Phi|_{X^\circ}$ is a biholomorphism onto $Z_0^{\reg}$.
       \item The Fano fibration $\pi$ factors through $\Phi$; that is, there is a continuous map $\Phi_0:Z_0\to Y$ such that $\pi=\Phi_0\circ \Phi$.
\item If, in addition, $(X,g)$ has subquadratic Riemannian curvature
growth, namely,
\[
    \lim_{x\to\infty} f(x)^{-1}|\Rm|(x)=0,
\]
then $\Phi_0:Y\to Z_0$ is a homeomorphism. Here \(Y\) is endowed with
its analytic topology, while \(Z_0\) is endowed with the topology
induced by the restriction of the ambient spacetime distance \(d_Z\).
   \end{enumerate}
\end{theorem}

\medskip
\noindent\textbf{Unique tangent spaces at infinity in the K\"ahler case.}

It is conjectured in \cite[Conjecture~6.2]{SunZhang} that every K\"ahler--Ricci shrinker has a unique tangent space at infinity and that the
limit is a metric cone. To prove the uniqueness, the discussion above shows that it is enough to prove
pointed precompactness, or equivalently local compactness of the time-zero
slice. We establish the required noncollapsing estimate under a natural set of
additional hypotheses. The main ingredient is the K\"ahler reduction technique introduced in \cite{SunZhang}.

The first step is a complex-analytic characterization of maximal volume
growth. K\"ahler reduction by a rational Hamiltonian vector field identifies
all sufficiently high level-set quotients with a fixed compact K\"ahler
orbifold $X_{z_0}$. The reduced K\"ahler classes have the form
\[
    -(K_{X_{z_0}}+D)+zL,
\]
where $D$ is an effective $\mathbb Q$-divisor and $L$ is a semiample
$\mathbb Q$-line bundle. The coarea formula shows that the volume-growth
exponent of the shrinker is determined by the numerical dimension of $L$. In
particular, Proposition~\ref{prop:maximal-volume-growth-characterization} proves the equivalence
\[
    \operatorname{AVR}(X,g)>0
    \quad\Longleftrightarrow\quad
    \dim Y=\dim X
    \quad\Longleftrightarrow\quad
    L\ \text{is big}.
\]
This is the noncompact counterpart of the equivalence between volume
noncollapse and the bigness of the limiting K\"ahler class in
\eqref{eq:numerical-noncollapsing}.

The reduced metrics,
after a change of variables, satisfy a twisted K\"ahler--Ricci flow on the
fixed compact orbifold. The assumption that the soliton vector field generates an $S^1$-action, together with scalar curvature bounds, controls the density
appearing in the corresponding Monge--Amp\`ere equation and gives a uniform
upper bound for the time derivative of the K\"ahler potential. This yields a uniform
$L^p$ bound for
the quotient volume forms. The compact K\"ahler noncollapsing theory
\cite{GPSS2,GPSS-ii} then gives lower volume bounds for balls in the quotient.
Lifting horizontally through the $S^1$-fibration and applying the coarea
formula in the potential direction produces the scale-invariant estimate
\[
    |B_g(x,\sigma r)|_g
    \ge C_\ep\sigma^{2n+\ep}r^{2n},
    \qquad r=d_g(p,x)\ge1,
\]
proved in Theorem~\ref{thm-shrinker-volume-noncollapsing}. This estimate gives
uniform ball-packing on bounded regions of the rescaled metrics and hence
pointed sequential compactness. As in the proof of
Theorem~\ref{thm--all}, the canonical slice $(Z_0,d^Z_0)$ embeds
isometrically into every sequential limit. The latter is locally compact, so
$Z_0$ is locally compact as well; Proposition~\ref{prop:local-compactness-pgh} then upgrades the subsequential compactness
to convergence of the full family and identifies the limit with
$(Z_0,d^Z_0,\bar p)$.

\begin{theorem}[Unique tangent space at infinity]\label{thm:tangentspace}
Let $(X,g,f)$ be a K\"ahler--Ricci shrinker with maximal volume
growth and bounded scalar curvature. Suppose that the soliton vector field
$J\nabla f$ generates an $S^1$-action. Then $X$ admits a unique tangent space
at infinity in the Gromov--Hausdorff sense. More precisely,
\[
    (X,\lambda^{-2}g,p)
    \xrightarrow[\lambda\to\infty]{\pGHconvtext}
    (Z_0,d^Z_0,\bar p).
\]
\end{theorem}

The method used to prove Theorem~\ref{thm:codimension-four} also yields a codimension-four Hausdorff estimate for the
singular set of this tangent space.

\medskip
\noindent\textbf{Four-dimensional Ricci shrinkers.}

Our final results concern general four-dimensional Ricci shrinkers and do not
use a K\"ahler structure. Let $(M^4,g,f)$ have bounded scalar curvature. In
this dimension, the scalar-curvature bound implies bounded Riemannian
curvature \cite{munteanu2015geometry}, but it does not directly give the
uniform volume estimates needed for pointed Gromov--Hausdorff precompactness
of the blow-downs.

The starting point is again the time-zero slice $Z_0$. If
$z\in Z_0\setminus\{\bar p\}$ and one blows up the spacetime at $z$, a suitably
normalized version of $F$ converges on the regular part of the tangent flow to
a function with unit parallel gradient. Hence, every tangent flow at $z$
splits off a line. In real dimension four, the remaining factor is a
three-dimensional shrinking soliton. The classification of such solitons implies that the tangent flow at each
$z\ne\bar p$ is unique; see Proposition~\ref{prop:1splitting}.

This splitting result organizes the geometry of the ends of $M$. Each end is
of one of three types. A conical end corresponds to a smooth region of the
time-zero slice and has quadratic curvature decay. A cylindrical end is
characterized by the occurrence of a quotient round cylinder
$S^3/\Gamma\times\mathbb R$ as a tangent flow. The remaining ends are called
generic; their singular tangent models are modeled on
$S^2\times\mathbb R^2$, $\mathbb{RP}^2\times\mathbb R^2$, or the twisted
quotient model. The rigidity theory of round cylinders
\cite{liwang2024rigidity} prevents a generic end from developing a sufficiently
accurate spherical cylindrical region.

The main local input is a canonical neighborhood theorem. At a point in a
fixed annular region of the self-similar spacetime where the scalar curvature
is sufficiently large, the rescaled flow is close to
\[
    N^3\times\mathbb R,
\]
where $N^3$ is a three-dimensional $\kappa$-solution; see
Theorem~\ref{thm:canonical}. The splitting direction is again generated by
the normalized potential. We use this theorem to prove uniform covering
estimates for the rescaled time slices. Conical and cylindrical ends are
handled directly by their model geometries. On a generic end, the canonical
neighborhood theorem reduces the high-curvature region to neck and
Bryant-type models, while a uniform local $L^1$ scalar-curvature estimate
controls the number of disjoint regions of definite size. This gives
sequential pointed Gromov--Hausdorff compactness.

Finally, the spacetime completion identifies the limit. Given any sequential
pointed Gromov--Hausdorff limit, the $H$-center construction used in
Theorem~\ref{thm--all} gives an isometric embedding of
$(Z_0,d^Z_0)$ into that limit. Since the latter is locally compact, the
canonical slice is locally compact. Proposition~\ref{prop:local-compactness-pgh} then applies directly and shows that the
entire family $(M,d_{g_t},p)$ converges to $(Z_0,d^Z_0,\bar p)$. We obtain:

\begin{theorem}[Four-dimensional tangent space at infinity]\label{thm:4d}
Every four-dimensional Ricci shrinker $(M^4,g,f)$ with bounded scalar
curvature has a unique tangent space at infinity in the Gromov--Hausdorff
sense. More precisely,
\[
    (M,\lambda^{-2}g,p)
    \xrightarrow[\lambda\to\infty]{\pGHconvtext}
    (Z_0,d^Z_0,\bar p).
\]
\end{theorem}

Thus, Theorems~\ref{thm--all},~\ref{thm:tangentspace}, and~\ref{thm:4d} fit
into a common picture: the time-zero slice of the spacetime completion is a
canonical metric object, and the principal task is to prove enough compactness
to identify it with the metric limit of the time slices. In the compact
K\"ahler setting, compactness follows from global diameter and volume
estimates. For K\"ahler--Ricci shrinkers, it is obtained from the Fano fibration and
K\"ahler reduction. In dimension four, it follows from the structure of the
ends and the canonical neighborhood theorem.

\medskip
\noindent\textbf{Organization of the paper.}

Section~\ref{sec--2} constructs the comparison map between the time-zero slice and an
arbitrary sequential singular-time limit and proves
Theorem~\ref{thm--all}. Section~\ref{sec--3} excludes the
$S^2\times\mathbb R^{2n-2}$ tangent flow by means of holomorphic
$\mathbb P^1$-perturbations and a coarea argument, and then proves
Theorem~\ref{thm:codimension-four}. Section~\ref{sec:schwarz-generic}
develops the analytic theory of homogeneous holomorphic functions on a
K\"ahler--Ricci shrinker, establishes the Schwarz estimate for the Fano
fibration, and proves Theorems~\ref{thm:local-smooth-convergence},~\ref{thm:ac-criterion}, and~\ref{thm:special-uniqueness}. Section~\ref{sec:generalshrinker} constructs
the spacetime completion of a shrinker with bounded scalar curvature, studies
the canonical map $\Phi:M\to Z_0$, and proves the local-compactness convergence
criterion, Theorem~\ref{thm--locally compact criterion}. In Section~\ref{sec--6}, we compare $Z_0$ with the Fano fibration base and prove Theorem~\ref{bounded curvature-homome}. Section~\ref{sec--7} uses
K\"ahler reduction to prove the required volume estimate and
Theorem~\ref{thm:tangentspace}. Section~\ref{sec--8} studies the ends of
four-dimensional shrinkers, proves the canonical neighborhood theorem and
sequential compactness, and concludes with Theorem~\ref{thm:4d}.

\subsection*{Acknowledgements}
Both authors thank Song Sun for many helpful discussions. Junsheng Zhang also
thanks the IASM at Zhejiang University for the hospitality during his visit in
the summer of 2026, where part of this work was done. Yu Li is supported by
the National Key R\&D Program of China (Grant No.~2025YFA1018200),
NSFC-12522105, and research funds from the University of Science and
Technology of China and the Chinese Academy of Sciences.

\section{Uniqueness of the Singular-Time Gromov--Hausdorff Limit}\label{sec--2}

\subsection{The time-zero slice \texorpdfstring{$(Z_0,d^Z_0)$}{(Z0,d0)}}\label{sec:time-zero-slice}

In this subsection, we recall the definition of the metric space $(Z_0,d^Z_0)$ from \cite{FL1}, namely the time-zero slice of the Ricci-flow completion with respect to the spacetime distance.

For the given K\"ahler--Ricci flow $(X^n,(\omega_t)_{t\in[-1,0)})$, let $g_t$ denote the associated K\"ahler metric. We assume that the entropy is bounded below by $-Y$, that is,
\begin{align*}
\inf_{\tau \in (0, 1]}\boldsymbol{\mu}(g_{-1}, \tau) \ge -Y,
\end{align*}
where $\boldsymbol{\mu}$ is Perelman's functional. Since $X$ is compact, such a lower bound always exists. We write $d_{\omega_t}$ for the distance on $X$ induced by $\omega_t$.

We consider the $d^*$-distance on $X \times [-0.99,0)$, defined as in \cite[Definition 3.2]{FL1}. More precisely, for any $x^*=(x,t)$ and $y^*=(y,s)$ in $X\times[-0.99,0)$ with $s\le t$, we define
	\begin{align}\label{eq:spacetime-distance-compact}
		d^*(x^*,y^*):=\inf_{-0.99\le \tau\le s}
        \max\left\{\sqrt{t-\tau},
        d_{W_1}^{\tau}(\nu_{x^*;\tau},\nu_{y^*;\tau})\right\}.\index{$\ensuremath{d^*}$}
	\end{align}
Here, $\nu$ denotes the conjugate heat kernel measure of the Ricci flow (see \cite[Definition 2.6]{FL1}), and $d^t_{W_1}$ is the $W_1$-Wasserstein distance with respect to $(X,\omega_t)$ (see \cite[Definition 2.1]{FL1}). Taking $\tau=-0.99$ in \eqref{eq:spacetime-distance-compact} shows that $X\times [-0.99,0)$ has bounded diameter with respect to $d^*$ (see \cite[(9.1)]{FL1}). By \cite[Proposition 3.9(2)]{FL1}, the time function $\mathfrak{t}$ is $2$-H\"older continuous with respect to $d^*$; that is,
\begin{equation*}
\left|\mathfrak{t}(x^*)-\mathfrak{t}(y^*)\right| \leq d^*(x^*, y^*)^2.
\end{equation*}

We then define $(Z,d_Z,\t)$ to be the metric completion of $X \times [-0.98,0)$ with respect to $d^*$. By construction, the completion adds only the points in $Z_0=\t^{-1}(0) \subset Z$. Moreover, it can be proved that $(Z,d_Z,\t)$ is a noncollapsed K\"ahler--Ricci flow limit space over $[-0.98,0]$ in the sense of \cite[Definition 3.21]{FL1}; see \cite[Section 9]{FL1}.

As in \cite[Definition 6.1]{FL1}, for each $t \in (-0.98, 0]$, we can define on $Z_t:=\t^{-1}(t) \subset Z$ a distance function:
\begin{align}\label{eq:time-slice-distance}
d^{Z}_t(x,y):=\lim_{s \nearrow t} d_{W_1}^{s}(\nu_{x;s},\nu_{y;s}) \in [0,+\infty],\index{$d^Z_t$}
\end{align}
for $x,y \in Z_t$, where the limit exists by \cite[Lemma 5.29]{FL1}. For $t \in (-0.98,0)$, the space $(Z_t,d^Z_t)$ is precisely $(X,d_{\omega_t})$.
As shown in \cite{FL1}, we have the following regular--singular decomposition in $Z_{(-0.98,0]}$:
	\begin{align*}
Z\supset\t^{-1}((-0.98,0])=:Z_{(-0.98,0]}=\RR \sqcup \MS.
	\end{align*}
Here, $\RR$ carries the structure of a Ricci flow spacetime $(\RR, \t, \partial_\t, g^Z)$. In our case, $\RR_{(-0.98,0)}=X \times (-0.98,0)$, and $g^Z_t=\omega_t$ for every $t\in(-0.98,0)$. Moreover, the singular set $\MS$ has codimension at least four with respect to $d_Z$ \cite[Theorem 1.10]{FL1}, and is contained in $Z_0$. In general, $(Z_0, d^Z_0)$ is a complete extended metric space; see \cite[Proposition 6.6]{FL1}. For simplicity, we define
	\begin{equation}\label{eq:regular-singular-decomposition}
Z_0^{\reg}:=Z_0\cap\RR,
\qquad
Z_0^{\sing}:=Z_0\cap\MS.
\end{equation}

\subsection{Identification of the regular part of \texorpdfstring{$Z_0$}{Z0}}
We define a natural embedding \(\Phi:X^\circ \to Z_0\). The following lemma establishes that it is well defined, and the subsequent proposition identifies \(X^\circ\) with \(Z_0^{\reg}\).
\begin{definition}[The embedding $\Phi$] \label{def:embed}
For any $x \in X^\circ$, we define
	\begin{align*}
\Phi(x):=\lim_{t \nearrow 0} (x, t)\in Z_0,
	\end{align*}
where the limit is taken with respect to $d_Z$.
\end{definition}

The following lemma shows that $\Phi$ is well defined.

\begin{lemma}\label{lem:contain1}
For any $x \in X^\circ$, the limit of $(x, t)$ as $t \nearrow 0$ exists and is contained in $Z_0^{\reg}$. Moreover, $\Phi$ is injective.
\end{lemma}

\begin{proof}
For any $x\in X^\circ$, \cite[Theorem~4.1]{CT} gives a small neighborhood $U_x$ of $x$ and a constant $C_x>0$ such that
\begin{equation} \label{eq:curvaturebound}
    |\Rm|_{g_t}\leq C_x \text{ on } U_x\times [-1,0).
\end{equation}
For $-0.98<s<t<0$, it follows from \cite[Proposition 2.21(i)]{FL1} that $(x, s)$ is an $H$-center of $(x, t)$, where $H=H(C_x, n)>0$ (see \cite[Definition 2.13]{FL1} for the definition of $H$-centers). It then follows from \cite[Lemma 3.13]{FL1} that
	\begin{align*}
d_Z \lc (x, t), (x, s) \rc \le \max\{1,\sqrt{H}\}\sqrt{t-s}.
	\end{align*}
Since $(Z, d_Z)$ is complete, this implies that the limit, denoted by $z$, of $(x, t)$ as $t \nearrow 0$ exists and is contained in $Z_0$. Furthermore, $z$ is regular because, by \eqref{eq:curvaturebound}, the curvature radius $r_{\Rm}(x, t) \ge c_x>0$ for a constant $c_x$ depending on $U_x$ and $C_x$ (for the definition of the curvature radius, see \cite[Definition~2.10]{FL1}). Thus, $z \in \RR$ by \cite[Theorem 7.15]{FL1}.

For any $y \in X^{\circ}$ with $y \ne x$, suppose $\Phi(y)=\Phi(x)$. Then, by Definition \ref{def:embed}, we have $\lim_{t \nearrow 0}d_Z((x,t),(y,t))=0$. Since $\Phi(x)$ is a regular point in $Z$, this implies $y \in U_x$. However, the distance comparison on $U_x$ would imply $x=y$, a contradiction.
\end{proof}

\begin{proposition}\label{prop:regular-part-identification}
The map $\Phi:(X^{\circ},d_{\omega_0}) \to (Z_0^{\reg}, d^Z_0)$ is bijective and $1$-Lipschitz. Here, $d_{\omega_0}$ is the Riemannian distance on $X^{\circ}$ induced by the limiting metric $\omega_0$ from \eqref{eq:smooth-convergence}.
\end{proposition}
\begin{proof}
By Lemma~\ref{lem:contain1} and \cite[Proposition 6.8]{FL1}, it remains to prove $Z_0^{\reg} \subset \Phi(X^\circ)$.
Let $z \in Z_0^{\reg}$. Since $\RR$ carries the structure of a Ricci flow spacetime $(\RR, \t, \partial_\t, g^Z)$, there exists a sufficiently small $\ep>0$ such that the flow $z_t \in \RR_t$ of $z$ under $\partial_\t$ is defined for $t \in [-\ep, 0]$. Since $\RR_{(-0.98,0)}=X \times (-0.98,0)$, we may assume that $z_t=(x, t) \in X \times \{t\}$ for every $t\in[-\ep,0)$.

The limit of $(x, t)$ as $t \nearrow 0$ is $z$, that is, $\Phi(x)=z$. Moreover, there exists a constant $c>0$ such that $r_{\Rm}(x, t)>c$ for every $t\in[-\ep,0)$. The Ricci flow equation then gives uniform lower and upper bounds for $\omega_t$ on a neighborhood of $x$; that is, there is a neighborhood $U_x$ of $x$ and a constant $C_x>0$ such that, for every $t\in [-1,0)$, we have on $U_x$
\begin{equation}
C_x^{-1}\omega_{-1}\leq \omega_t\leq C_x\omega_{-1}.
\end{equation}
Arguing as in \cite[Proof of Theorem~1.5]{CT} by considering the volume of the null locus, we conclude that $x\in X^\circ$. Indeed, suppose that $x\in\nul(\alpha)$. Then there exists an irreducible $k$-dimensional analytic subvariety $V\subset X$ with $x\in V$ and $\int_V\alpha^k=0$. On the one hand,
\[
\int_V\omega_t^k
\geq \int_{V\cap U_x}\omega_t^k
\geq C_x^{-k}\int_{V\cap U_x}\omega_{-1}^k>0.
\]
On the other hand, $\int_V\omega_t^k\to\int_V\alpha^k=0$, a contradiction.
\end{proof}

\subsection{Construction of the comparison map \texorpdfstring{$h$}{h}}

By \eqref{eq:diameter} and \eqref{eq:volume-noncollapsing}, for every sequence $t_j \nearrow 0$, there exists a subsequence, still denoted by $t_j$, for which $(X, d_{\omega_{t_j}})$ converges to $(X_{\GH}, d_{\GH})$ in the Gromov--Hausdorff sense. Note that the limit $(X_{\GH}, d_{\GH})$ may depend on the time sequence $\{t_j\}$ and hence may not be unique a priori.
In this subsection, we fix a time sequence $\{t_j\}$ with $t_j \nearrow 0$ such that
\begin{align} \label{eq:grsequence}
   (X, d_{\omega_{t_j}}) \xrightarrow[t_j \nearrow 0]{\GHconvtext} (X_{\GH}, d_{\GH}).
\end{align}

\begin{proposition}\label{prop:embh}
There exists an isometric embedding
\begin{align*}
h: (Z_0, d^Z_0) \longrightarrow (X_{\GH}, d_{\GH}).
\end{align*}
\end{proposition}

\begin{proof}
We first prove that $(Z_0, d^Z_0)$ is separable, that is, it contains a countable dense subset. For every $\ep>0$, choose an arbitrary finite $\ep$-separated set $\{z_i\}_{i=1}^N \subset Z_0$.

For $z_1$, we set $x_{1,j}^*$ to be an $H_{2n}$-center of $z_1$ at $X \times \{t_j\}$. After passing to a subsequence of $\{t_j\}$, we may assume that $x_{1,j}^*$ converges to $x_1 \in X_{\GH}$ in the Gromov--Hausdorff sense.

Repeating this process for $z_2,z_3,\ldots$ and taking a diagonal subsequence, still denoted by $\{t_j\}$, we conclude that, for each $z_i$, the $H_{2n}$-center $x_{i,j}^* \in X \times \{t_j\}$ converges to $x_i \in X_{\GH}$ in the Gromov--Hausdorff sense.

Define $h:\bigcup_i\{z_i\}\to\bigcup_i\{x_i\}$ by
\begin{align*}
h(z_i)=x_i
\end{align*}
for every integer $i$ with $1\le i\le N$. By \eqref{eq:time-slice-distance}, $h$ is isometric on this set, in the sense that
\begin{align*}
d_{\GH}\lc h(z_{i_1}), h(z_{i_2}) \rc=d_{\GH}(x_{i_1}, x_{i_2})=d^Z_0(z_{i_1}, z_{i_2}),\quad \forall\,1\le i_1,i_2\le N.
\end{align*}
In particular, $\{x_i\}$ is an $\ep$-separated set of $X_{\GH}$. Since $(X_{\GH},d_{\GH})$ is compact (see \eqref{eq:diameter}), we conclude that $N$ must be finite. Since $(Z_0,d^Z_0)$ is complete, it is compact and, in particular, separable.

We can now choose a countable dense subset $\{z_i\}_{i \ge 1}$ of $Z_0$. By the same argument as above, there exists a map $h:\bigcup_i\{z_i\}\to X_{\GH}$ such that
\begin{align} \label{eq:subiso}
d_{\GH}\lc h(z_{i_1}), h(z_{i_2}) \rc=d^Z_0(z_{i_1}, z_{i_2}),\quad \forall i_1, i_2 \ge 1.
\end{align}
By \eqref{eq:subiso}, we can extend $h$ to an isometric embedding from $Z_0$ to $X_{\GH}$ since $\{z_i\}_{i \ge 1}$ is dense. This completes the proof.
\end{proof}

We next prove the following proposition.
\begin{proposition}\label{prop:surj}
The map $h$ constructed in Proposition~\ref{prop:embh} is surjective.
\end{proposition}

\begin{proof}
We prove the surjectivity of $h$ by contradiction. Since $(Z_0, d^Z_0)$ is compact and $h$ is isometric, $h(Z_0)$ is compact. Thus, if $h(Z_0) \ne X_{\GH}$, there exist $x^\infty\in X_{\GH}$ and $r_0>0$ such that
\begin{align} \label{eq:contr1}
B_{d_{\GH}}(x^\infty,2r_0) \cap h(Z_0) =\emptyset.
\end{align}
By definition, we can find $x_j \in X$ such that $(x_j,t_j)$ converges to $x^\infty$ in the Gromov--Hausdorff sense. By \cite[Theorem 1.12(b)]{FL1}, we have
\begin{align} \label{eq:density}
\left|\{r_{\Rm}<r\}\cap (X\times\{t\})\right|_{\omega_t} \le C(n,Y,\ep) r^{2-\ep}
\end{align}
for every $t\in(-0.01,0)$, $\ep\in(0,1)$, and $r\in(0,0.01]$. Here, the key fact we use to apply \cite[Theorem 1.12(b)]{FL1} is that $(Z, d_Z)$ has bounded diameter with respect to $d_Z$, see \cite[Equation (9.1)]{FL1}, and hence is compact; see \cite[Theorem~3.20]{FL1}. Thus, we can cover $Z_{[-0.01, 0]}$ by finitely many balls with radius $0.1$ and apply \cite[Theorem 1.12(b)]{FL1} for each ball.

On the other hand, it follows from \eqref{eq:volume-noncollapsing} that
\begin{align} \label{eq:density1}
\left|B_{\omega_{t_j}}(x_j,r_0)\right|_{\omega_{t_j}} \ge c r_0^{2n+0.01},
\end{align}
where $c>0$ is a constant depending only on the K\"ahler--Ricci flow $(X,\omega_t)$. Combining \eqref{eq:density} and \eqref{eq:density1}, we conclude that there exists $\delta>0$ such that, for each $j$, we can find $y_j \in B_{\omega_{t_j}}(x_j,r_0)$ such that
\begin{align*}
r_{\Rm}(y_j, t_j) \ge \delta.
\end{align*}
By the definition of the curvature radius, there exist an integer $j_0$ and a constant $\delta_0>0$ such that, for every $j\ge j_0$,
\begin{align*}
|\Rm| \le \delta_0^{-2} \quad \text{on} \quad B_{\omega_{t_j}}(y_j, \delta_0) \times [t_j, 0).
\end{align*}
Thus, by the proof of Lemma~\ref{lem:contain1}, the limit, denoted by $y^j$, of $(y_j, t)$ as $t \nearrow 0$ exists and belongs to $Z_0^{\reg}$. Moreover, $(y_j, t_j)$ is an $H$-center of $y^j$ for a constant $H$ independent of $j$.

From the construction of $h$, there exists a dense countable subset $\{z_i\}_{i \ge 1}$ of $Z_0$ such that $h(z_i)$ is defined to be the Gromov--Hausdorff limit of $x^*_{i,j} \in X \times \{t_j\}$, where $x^*_{i,j}$ is an $H_{2n}$-center of $z_i$. Since $(Z_0, d^Z_0)$ is compact, there exists $z_{i_0}$ such that
\begin{align}\label{eq:density2}
d^Z_0(y^j, z_{i_0}) \le r_0/2
\end{align}
for infinitely many $j$. By taking a subsequence of $\{t_j\}$, we may assume that \eqref{eq:density2} holds for all $j$.

It follows from the monotonicity \cite[Proposition 2.12]{FL1} that
\begin{align*}
d_{\omega_{t_j}}(y_j, x^*_{i_0,j}) \le C \sqrt{|t_j|}+r_0/2.
\end{align*}
From this, we conclude that a Gromov--Hausdorff limit of $(y_j,t_j)$, denoted by $y^\infty$, satisfies
\begin{align*}
d_{\GH}(y^\infty, h(z_{i_0})) \le r_0/2.
\end{align*}
Since $y_j \in B_{\omega_{t_j}}(x_j,r_0)$, we have
\begin{align*}
d_{\GH}(x^\infty, h(z_{i_0})) \le 3r_0/2,
\end{align*}
which contradicts \eqref{eq:contr1} and completes the proof.
\end{proof}
\begin{remark}
   In the proof above, the almost sharp volume lower bound \eqref{eq:density1} is not essential; all that is needed is a lower bound depending only on the radius and independent of time.
\end{remark}
Combining Propositions~\ref{prop:embh} and~\ref{prop:surj}, we obtain the following result, which, together with Proposition~\ref{prop:regular-part-identification}, implies Theorem~\ref{thm--all}.

 \begin{theorem}
The Gromov--Hausdorff limit of $(X, d_{\omega_t})$ as $t \nearrow 0$ exists and is isometric to $(Z_0, d^Z_0)$. In particular, the Gromov--Hausdorff limit is unique.
 \end{theorem}

\section{Structure of the Singular-Time Limit \texorpdfstring{$Z_0=X_{\GH}$}{Z0=XGH}}\label{sec--3}
In this section, we prove Theorem~\ref{thm:codimension-four}.

\subsection{Holomorphic \texorpdfstring{$\mathbb P^1$}{P1}-perturbations and exclusion of \texorpdfstring{$(2n-2)$}{2n-2}-splitting tangent flows}
Let $B_r^k$ denote the ball of radius $r$ centered at the origin in $\C^k$, and set
\begin{equation*}
    V^k_r:=\mathbb P^1 \times B_r^k.
\end{equation*}
Let $(g_0,J_0)$ denote the standard K\"ahler structure on $V^k_r$, and let $\{z_i\}_{i=1}^{k}$ denote the coordinates on $B_r^k$, which can also be naturally viewed as $J_0$-holomorphic functions on $V_r^k$.

\begin{lemma}\label{lem--rigidity}
 There exists a constant $\delta=\delta(k)$ such that the following holds. Let $(g,J)$ be a K\"ahler structure on $V^k_4$ such that
 \begin{equation}\label{eq--closeness}
     \|g-g_0\|_{C^3(V_2^k)}+\|J-J_0\|_{C^3(V_2^k)}\leq \delta.
 \end{equation}
 Then there exists a neighborhood $U$ of $\mathbb P^1 \times \{0\}$ which is $J$-biholomorphic to $V^k_1$.
\end{lemma}

\begin{proof}
   Using \cite{McDuffSalamon2012} and  arguing as in \cite[Corollary~2.3]{CCD}, we find that, for $\delta$ sufficiently small, there exists an embedded $J$-holomorphic sphere $\mathbb P^1_J$ inside $V^k_2$ which is $\Psi(\delta)$-close to $\mathbb P^1 \times \{0\}$, in the sense that there exists a smooth section $\xi_J$ of the tangent bundle restricted on $\mathbb P^1\times \{0\}$ such that 
    \[\mathbb P_J^1=\exp(\xi_J),\quad \|\xi_J\|_{C^3(P^1\times \{0\}}=\Psi(\delta).\] Let $N_{\mathbb P_J^1}$ denote the holomorphic normal bundle of $\mathbb P^1_J$ with respect to the complex structure $J$. Then we claim that $N_{\mathbb P_J^1}\simeq \bigoplus_{j=1}^{k}\mathcal{O}$. Since $c_1(N_{\mathbb P_J^1})$ is a topological invariant, for $\delta$ sufficiently small,
    \begin{equation}\label{eq:normal-c1-zero}
        c_1(N_{\mathbb P_J^1})=0.
    \end{equation}
By considering $\bar\partial_J$-harmonic sections with respect to the metric induced by the K\"ahler metric $g$ and arguing by contradiction, we find that, for $\delta$ sufficiently small,
\begin{equation}\label{eq:normal-upper-semicontinuity}
   \dim H^0(\mathbb P^1_J,  N_{\mathbb P_J^1}\otimes \mathcal{O}(-1))\leq \dim H^0(\mathbb P^1,  N_{\mathbb P^1}\otimes \mathcal{O}(-1))=0.
\end{equation}
More precisely, let $(g_i, J_i)$ be a sequence of K\"ahler structures satisfying \eqref{eq--closeness} for $\delta_i = \Psi(i^{-1})$. In the following, using the exponential map to identify the smooth ambient complex vector bundle, we may assume that we are working with Dolbeault operators vary on a fixed complex vector bundle. By the construction, we have for any $p\geq 1$ 
    \(
\|\bar\partial_{J_i}-\bar\partial_0\|_{W^{1,p}\to L^p}
\longrightarrow 0.
\)
Let $\{s_{i,j}\}_{j=1}^{N_i}$ be an orthonormal basis of the vector space $H^0(\mathbb P^1_{J_i},N_{\mathbb P^1_{J_i}} \otimes \mathcal{O}(-1))$. After passing to a subsequence, we may assume that $s_{i,j}$ converges weakly in $W^{1,p}$ and strongly in $L^p$ to $s_{\infty,j}$ for $j=1,\ldots,N_\infty$, where $N_\infty=\limsup_{i\to\infty}N_i$. The limits are orthonormal and, in particular, linearly independent vectors in the space
\(
H^0\!\left(\mathbb P^1, N_{\mathbb P^1} \otimes \mathcal{O}(-1)\right).
\)
    Combining \eqref{eq:normal-c1-zero} and \eqref{eq:normal-upper-semicontinuity} with the classical Grothendieck splitting theorem for holomorphic vector bundles on $\mathbb P^1$, we obtain
\begin{equation*}
    N_{\mathbb P_J^1}\simeq \bigoplus_{j=1}^{k}\mathcal{O}.
\end{equation*}
We may then apply Grauert's theorem \cite{grauert} and \cite[Theorem~1.7]{hwang} to conclude that a neighborhood of $\mathbb P^1_J$ is biholomorphic to a neighborhood of the zero section in $\mathbb P^1\times\mathbb C^k$. After a simple rescaling, the statement follows.
\end{proof}

Next, we consider our model space:
\begin{align*}
\mathcal C^{2n-2}_{-1}:=(\bar M,\bar g,\bar f)=\left(S^{2} \times \R^{2n-2}, g_{\mathrm{round}} \times g_{\mathrm{flat}}, \frac{|\vec{x}|^2}{4}+\log 2 \right),
\end{align*}
where $g_{\mathrm{round}}$ is the round metric on $S^{2}$ with constant Gaussian curvature $1/2$, and $g_{\mathrm{flat}}$ is the Euclidean metric on $\R^{2n-2}$. The vector $\vec{x}=(x_1,\ldots,x_{2n-2})$ denotes the standard coordinate vector on $\R^{2n-2}$.

We define $\mathcal C^{2n-2}=(\bar M, (\bar g_t)_{t<0}, (\bar f_t)_{t<0})$ to be the associated Ricci flow, where $t=0$ corresponds to the singular time, and the potential function is given by
\begin{align*}
\bar f_t:=\frac{|\vec{x}|^2}{4|t|}+\log 2.
\end{align*}

Let the completion of $(\bar M, (\bar g_t)_{t \in (-\infty, 0)})$ be $\overline{\mathcal C}^{2n-2}$, equipped with the spacetime distance $d_{\mathcal C}^*$ defined as in \eqref{eq:spacetime-distance-compact}. A direct verification shows that the metric completion adds only the singular set $\R^{2n-2} \times \{0\}$ to $ \mathcal C^{2n-2}$. We then define the base point $p^*$ as the limit of $(\bar p, t)$ as $t \nearrow 0$ with respect to $d_{\mathcal C}^*$, where $\bar p \in \bar M$ is a minimum point of $\bar f_{-1}$. The point $p^*$ is independent of the choice of $\bar p$. Moreover, for any $t<0$,
	\begin{align*}
\nu_{p^*;t}=(4\pi |t|)^{-n} e^{-\bar f_t} \,\mathrm{d}V_{\bar g_t}.
	\end{align*}

We next prove the following result. For the precise definitions of tangent flows and isometry, we refer the reader to \cite[Definitions~1.8 and~5.21]{FL1}.

\begin{theorem}\label{thm:exclude}
    For the metric completion $(Z,d_Z,\t)$ of $(X^n,(\omega_t)_{t\in [-1,0)})$, no tangent flow is isometric to $\overline{\mathcal C}^{2n-2}$.
\end{theorem}

\begin{proof}
Suppose otherwise. Then a tangent flow at $z_0 \in Z$ is isometric to $\overline{\mathcal C}^{2n-2}$. Since the singular set is entirely contained in $Z_0$, we must have $z_0 \in Z_0$.

By the smooth convergence (see \cite[Theorem 1.5]{FL1}), there exists a limiting parallel complex structure $J_\infty$ on $S^{2} \times \R^{2n-2}$, obtained as the smooth convergence of the original complex structure $J$ on $X$. Note that $J_\infty$ must preserve the decomposition $S^{2} \times \R^{2n-2}$, since otherwise the tangent flow splits off more than $2n-2$ lines; see \cite[Lemma 11.7]{FL1}. Thus, by the uniformization theorem, $(S^{2} \times \R^{2n-2}, J_\infty)$ is biholomorphic to the standard $\mathbb P^1 \times \C^{n-1}$.

We fix a small constant \(\tau_0>0\), to be determined later. For
\(\tau\in(0,\tau_0]\), let \(q_{-\tau}\) be an \(H_{2n}\)-center of \(p^*\).
If \(\tau_0\) is sufficiently small, then by the smooth convergence and
Lemma~\ref{lem--rigidity}, there exists an open neighborhood
\(U\) of \(q_{-\tau_0}\) which is biholomorphic to the standard model
\(
\mathbb P^1\times B_1^{n-1}.
\)
For each \(z\in B_1^{n-1}\), denote by
\(\mathbb P^1(z)\) the \(\mathbb P^1\)-fiber over \(z\) in
\(U\).

Let \(\phi_t\) be the time-dependent flow generated by
\(-\nabla_{g_t}f_t\), where $f_t$ is the potential function at $z_0$, with initial condition
\(\phi_{-\tau_0}=\mathrm{id}\). We consider
\[
\mathbb P_t^1:=\phi_t\bigl(\mathbb P^1(0)\bigr),
\]
which need not be holomorphic.
In particular,
\[
\mathbb P^1_{-\tau_0}=\mathbb P^1(0).
\]
Since \(\phi_t\) is generated by a vector field, \(\mathbb P_t^1\) is homologous to
\(\mathbb P^1_{-\tau_0}\). By the strong uniqueness theorem
\cite[Theorem 1.2]{fang-li-unique}, the volume of the spheres \(\mathbb P_t^1\) under the metric $|t|^{-1}g_t$ has a uniform upper bound. Therefore, there exists a constant \(C>0\) such
that
\begin{equation}\label{eq:fiber-volume-bound}
  \left|\mathbb P^1_{-\tau_0}\right|_{|t|^{-1}\omega_t}=\langle |t|^{-1}\omega_t, [\mathbb P_t^1]\rangle  \leq \left|\mathbb P^1_t\right|_{|t|^{-1}g_t}
    \leq C,
    \qquad
    \forall\, t\in[-\tau_0,0),
\end{equation}
where we used the Wirtinger inequality.

By the cohomological class relation,
\begin{equation}\label{eq:kahler-class-relation}
    |t|^{-1}[\omega_t]=|t|^{-1}\alpha-K_X.
\end{equation}
The bound
\eqref{eq:fiber-volume-bound} gives
\[
 |t|^{-1}\alpha\cdot [\mathbb P^1_{-\tau_0}]
 -
 K_X\cdot [\mathbb P^1_{-\tau_0}]
 \leq C.
\]
Letting \(t\to 0\), we conclude that
\begin{equation}\label{eq:fiber-alpha-zero}
  \alpha\cdot [\mathbb P^1_{-\tau_0}] =0,
\end{equation}
that is, $\mathbb P^1(0)\subset \nul(\alpha)$.
Since the fibers \(\mathbb P^1(z)\) are all homologous, we have $U\subset\nul(\alpha)$, a contradiction because $\nul(\alpha)$ cannot contain an open subset.
\end{proof}

We next recall from \cite[Sections~8 and~11]{FL1} that the singular set $\MS$ of $Z$ has the following natural stratification:
	\begin{equation} \label{eq:stra}
		\MS^0 \subset \MS^2 \subset \ldots \subset \MS^{2n-2}=\MS.
	\end{equation}
A point $z$ belongs to $\MS^k$ if and only if no tangent flow at $z$ is $(k+1)$-symmetric. Here, a tangent flow $(Z', d_{Z'}, z', \t')$ is said to be $k$-symmetric if one of the following holds:
		\begin{enumerate}[label=\textnormal{(\arabic*)}]
		\item $(Z',d_{Z'},z',\t')$ is $k$-splitting and is not a static cone.

		\item $(Z',d_{Z'},z',\t')$ is a static cone that is $(k-2)$-splitting.
	\end{enumerate}

We next prove the following theorem.

\begin{theorem}\label{thm:exclude1}
    For the metric completion $(Z,d_Z,\t)$ of $(X^n,(\omega_t)_{t\in [-1,0)})$, we have
 	\begin{equation*}
\MS^{2n-4}=\MS.
	\end{equation*}
\end{theorem}

\begin{proof}
By \eqref{eq:stra}, it remains to prove that
 	\begin{equation*}
\MS^{2n-2}\setminus\MS^{2n-4}=\emptyset.
	\end{equation*}
Indeed, for any $z\in\MS^{2n-2}\setminus\MS^{2n-4}$, it follows from the definition that some tangent flow at $z$ is either $(2n-2)$-splitting and non-static or a static cone that is $(2n-4)$-splitting.

If it is the first case, then such a tangent flow must be $\overline{\mathcal C}^{2n-2}$, which is impossible by Theorem~\ref{thm:exclude}. The second case can also be excluded, since a tangent flow at $z$ has finite arrival time; see \cite[Definition 7.17]{FL1}.
\end{proof}

\subsection{Hausdorff dimension estimate for the singular set}

Using Theorem~\ref{thm:exclude1}, we now estimate the Hausdorff measure of $Z_0^{\sing}$. First, we prove the almost Euclidean lower bound for the volume of $d^Z_0$-balls; see \cite[Definition 5.33]{FL1} for the volume definitions.

\begin{lemma}\label{lem:volume}
For every $\ep>0$, there exists $c_\ep>0$ such that
\begin{equation}\label{eq:gh-limit-volume-noncollapsing}
    \left|B_{d^Z_0}(p,r)\right|_{d^Z_0}
    \geq c_\ep r^{2n+\ep}
\end{equation}
for every $p\in Z_0$ and every $r\in(0,1]$. In particular, $Z^{\reg}_0$ is an open dense subset of $Z_0$.
\end{lemma}

\begin{proof}
Fix $p\in Z_0$, $r\in(0,1]$, and a sequence
$t_j\nearrow0$. Let $p_j$ be an $H_{2n}$-center of $\nu_{p;t_j}$. By
\cite[Theorem~1.12(b)]{FL1}, for every sufficiently small
$\theta>0$,
\begin{equation}\label{eq:bad-curvature-volume}
\left|
    \{r_{\Rm}<\theta r\}
    \cap B_{\omega_{t_j}}(p_j,r/2)
\right|_{\omega_{t_j}}
\leq C\theta^{1.99}r^{2n}.
\end{equation}
On the other hand, \eqref{eq:volume-noncollapsing} gives
\[
    \left|B_{\omega_{t_j}}(p_j,r/2)\right|_{\omega_{t_j}}
    \geq C_\ep^{-1}(r/2)^{2n+\ep}.
\]
Choose $\theta=\theta(\ep,r)>0$ so small that the right-hand side of
\eqref{eq:bad-curvature-volume} is at most one-half of this lower
bound. It follows that
\begin{equation}\label{eq:good-curvature-volume}
\left|
    \{r_{\Rm}\geq\theta r\}
    \cap B_{\omega_{t_j}}(p_j,r/2)
\right|_{\omega_{t_j}}
\geq c_\ep r^{2n+\ep},
\end{equation}
where $c_\ep>0$ is independent of $p$, $r$, and $j$.

Since $\theta r>0$ is fixed, the good-curvature region in
\eqref{eq:good-curvature-volume} is contained in $X^\circ$ for all
sufficiently large $j$, and hence $\Phi$ is defined on this region. We
claim that, for all sufficiently large $j$,
\begin{equation}\label{eq:good-region-contained-in-limit-ball}
\Phi\left(
    \{r_{\Rm}\geq\theta r\}
    \cap B_{\omega_{t_j}}(p_j,r/2)
\right)
\subset B_{d^Z_0}(p,r).
\end{equation}
Indeed, otherwise there would exist a subsequence and points
\[
    y_j\in
    \{r_{\Rm}\geq\theta r\}
    \cap B_{\omega_{t_j}}(p_j,r/2)
\]
such that $\Phi(y_j)\notin B_{d^Z_0}(p,r)$. After passing to a further
subsequence, we may assume that
$\Phi(y_j)\to y_\infty\in Z_0$ and hence
$d^Z_0(y_\infty,p)\geq r$. Let $p_j'$ be an $H_{2n}$-center of
$\nu_{y_\infty;t_j}$. By \eqref{eq:time-slice-distance},
\begin{align*}
    d^Z_0(y_\infty,p)
    &=\lim_{j\to\infty}d_{\omega_{t_j}}(p_j',p_j)\\
    &\leq
    \limsup_{j\to\infty}
    \left(
        d_{\omega_{t_j}}(p_j',y_j)
        +d_{\omega_{t_j}}(y_j,p_j)
    \right).
\end{align*}
Moreover, \cite[Proposition~2.21(ii)]{FL1} gives
\[
    d_{\omega_{t_j}}(p_j',y_j)
    \leq
    d^Z_0\bigl(y_\infty,\Phi(y_j)\bigr)
    +C\sqrt{|t_j|}
    \longrightarrow0.
\]
Therefore, $d^Z_0(y_\infty,p)\leq r/2$, a contradiction. This proves
\eqref{eq:good-region-contained-in-limit-ball}. Passing to the limit in
\eqref{eq:good-curvature-volume}, using smooth convergence on the region
where $r_{\Rm}\geq\theta r$, completes the proof.
\end{proof}

\begin{theorem}
The Hausdorff dimension of the singular set $Z_0^{\sing}$ with respect to $d^Z_0$ satisfies
\begin{equation}\label{eq:dimension}
    \dim_{\mathscr H}\bigl(Z_0^{\sing})\leq 2n-4.
\end{equation}
\end{theorem}

\begin{proof}
By \cite[Corollary~4.24(b)]{fang-li2}, for every $z_0\in Z$ and every
$r>0$ satisfying $\t(z_0)-2r^2>-0.98$, one has
\begin{equation}\label{eq:quantitative-stratum-volume-local}
\left|
    B^*_{\delta r}
    \bigl(\MS^{\ep,2n-4}_{\delta r,\ep r}\bigr)
    \cap B_Z^*(z_0,r)\cap Z_t
\right|_{d_t^Z}
\leq
C(n,Y,\ep)\delta^{4-\ep}r^{2n}
\end{equation}
for every $\delta\in(0,\ep)$ and every $t\in[-1,0]$. Here,
$B_s^*(A)$ denotes the $s$-neighborhood of $A$ with respect to $d_Z$,
and $\MS^{\ep,2n-4}_{\delta r,\ep r}$ is the quantitative singular
stratum from \cite[Definition~8.11]{FL1}.

Since $Z_0$ is compact and $Z$ has bounded $d_Z$-diameter, $Z_0$ can be
covered by finitely many spacetime balls of a fixed sufficiently small
radius. After applying
\eqref{eq:quantitative-stratum-volume-local} over this finite cover,
rescaling the parameters by the fixed covering radius, and adjusting the
constant, we obtain
\begin{equation}\label{eq:quantitative-stratum-volume-global}
\left|
    B^*_{\delta}
    \bigl(\MS^{\ep,2n-4}_{\delta,\ep}\bigr)
    \cap Z_0
\right|_{d^Z_0}
\leq C_\ep\delta^{4-\ep}
\end{equation}
for all sufficiently small $\delta>0$.

By \cite[Lemma~6.4]{FL1},
\[
    d_Z(x,y)\leq d^Z_0(x,y),
    \qquad x,y\in Z_0.
\]
Consequently,
\begin{equation}\label{eq:quantitative-stratum-tubular-d0}
\left|
    \left\{
        y\in Z_0:
        d^Z_0\bigl(y,\MS^{\ep,2n-4}_{\delta,\ep}\bigr)<\delta
    \right\}
\right|_{d^Z_0}
\leq C_\ep\delta^{4-\ep}.
\end{equation}
Define
\begin{equation}\label{eq:fixed-quantitative-stratum}
    \MS^{\ep}
    :=
    \bigcap_{0<\delta<\ep}
    \MS^{\ep,2n-4}_{\delta,\ep}.
\end{equation}
Since $\MS^{\ep}$ is contained in every set appearing in the
intersection, \eqref{eq:quantitative-stratum-tubular-d0} implies
\begin{equation}\label{eq:fixed-stratum-tubular-volume}
\left|
    \left\{
        y\in Z_0:
        d^Z_0(y,\MS^{\ep})<\delta
    \right\}
\right|_{d^Z_0}
\leq C_\ep\delta^{4-\ep}
\end{equation}
for all sufficiently small $\delta>0$.

We now estimate the Hausdorff measure of $\MS^{\ep}$. Let
$\{p_i\}_{i=1}^{N_\delta}\subset\MS^{\ep}$ be a maximal
$\delta$-separated set with respect to $d^Z_0$. Then the balls
$B_{d^Z_0}(p_i,\delta/2)$ are pairwise disjoint, while the balls
$B_{d^Z_0}(p_i,\delta)$ cover $\MS^{\ep}$. By
\eqref{eq:fixed-stratum-tubular-volume} and
Lemma~\ref{lem:volume}, for every $\ep>0$,
\[
    N_\delta c_\ep \delta^{2n+\ep}
    \leq C_\ep\delta^{4-\ep}.
\]
Hence
\begin{equation}\label{eq:fixed-stratum-covering-number}
    N_\delta
    \leq C_{\ep}
    \delta^{-(2n-4+2\ep)}.
\end{equation}
Let $\sigma>2\ep$. Denote by
$\mathscr H^s_{\rho,d^Z_0}$ the $s$-dimensional Hausdorff content at
scale $\rho$. From the covering by the balls $B_{d^Z_0}(p_i,\delta)$ and
\eqref{eq:fixed-stratum-covering-number},
\[
\mathscr H^{2n-4+\sigma}_{2\delta,d^Z_0}(\MS^{\ep})
\leq
N_\delta(2\delta)^{2n-4+\sigma}
\leq
C_{\ep}\delta^{\sigma-2\ep}.
\]
Letting $\delta\downarrow0$ gives
\begin{equation}\label{eq:fixed-stratum-hausdorff-zero}
    \mathscr H_{d^Z_0}^{2n-4+\sigma}(\MS^{\ep})=0
    \qquad\text{for every }\sigma>2\ep.
\end{equation}

The definitions imply that
\begin{equation}\label{eq:quantitative-strata-monotonicity}
    \MS^{\ep}\subset\MS^{\ep'}
    \qquad\text{whenever }0<\ep'\leq\ep.
\end{equation}
Moreover,
\begin{equation}\label{eq:qualitative-quantitative-union}
    \MS^{2n-4}
    =\bigcup_{m=1}^{\infty}\MS^{1/m}.
\end{equation}
Fix $\sigma>0$ and $m\geq1$. Choose $j\geq m$ so large that
$2/j<\sigma$. By \eqref{eq:quantitative-strata-monotonicity},
$\MS^{1/m}\subset\MS^{1/j}$, and therefore
\eqref{eq:fixed-stratum-hausdorff-zero} gives
\[
    \mathscr H_{d^Z_0}^{2n-4+\sigma}(\MS^{1/m})=0.
\]
Using Theorem~\ref{thm:exclude1},
\eqref{eq:qualitative-quantitative-union}, and countable subadditivity of
Hausdorff measure, we conclude that
\[
\mathscr H_{d^Z_0}^{2n-4+\sigma}(Z_0^{\sing})
=
\mathscr H_{d^Z_0}^{2n-4+\sigma}(\MS)
\leq
\sum_{m=1}^{\infty}
\mathscr H_{d^Z_0}^{2n-4+\sigma}(\MS^{1/m})
=0.
\]
Since $\sigma>0$ is arbitrary, the Hausdorff dimension estimate follows.
\end{proof}

\section{K\"ahler--Ricci Shrinkers via Fano Fibrations}
\label{sec:schwarz-generic}

In this section, we prove Theorems~\ref{thm:local-smooth-convergence},~\ref{thm:ac-criterion}, and~\ref{thm:special-uniqueness}. The key point is
to establish polynomial growth estimates for homogeneous holomorphic
functions on a K\"ahler--Ricci shrinker and then a Schwarz-type lower bound
for $\omega_t$ using the polarized Fano fibration structure. The main
technique is to apply the maximum principle on a noncompact manifold, using
suitable barrier functions supplied by the soliton potential.

\subsection{Conventions and the Fano fibration}

Throughout this section, let \((X^n,\omega,J,f)\) be a noncompact
K\"ahler--Ricci shrinker, and let \(g\) be the associated K\"ahler metric. We
use the complex convention throughout this section:
\[
    \Delta_\omega=\tr_\omega \ii\partial\bar\partial,
    \qquad
    \scal_\omega=\tr_\omega\Ric(\omega).
\]
Thus \(\Delta_\omega=\frac12\Delta_g\), where \(\Delta_g\) is the Riemannian
Laplacian. Throughout this section, \(\nabla u\) denotes the Riemannian
gradient, and \(|\nabla u|^2\) its Riemannian norm. For a real function \(u\),
we denote by \(\nabla^{1,0}u\) its \((1,0)\)-gradient, so that
\[
    |\nabla^{1,0}u|^2=\frac12|\nabla u|^2.
\]
We normalize the shrinker by
\begin{align}\label{eq:identitiesric}
    \Ric(\omega)+\ii\partial\bar\partial f=\omega,
    \qquad
    \scal_\omega+|\nabla^{1,0}f|^2=f+n.
\end{align}
Set
\[
    \rho:=\sqrt{f+n}.
\]
We may assume that $\rho>0$, since otherwise $(X,g)$ is isometric to the Gaussian soliton. For a function \(u\), we write
\[
    \Delta_f u:=\Delta_\omega u-\frac12(\nabla f)(u).
\]
Thus \(\Delta_f\) is one half of the Riemannian weighted Laplacian
\(u\mapsto\Delta_g u-(\nabla f)(u)\). With this notation, the chosen
normalization gives the simple identity
\[
    \Delta_f f=-f.
\]
By \cite{CaoZhou} and \cite{HaslhoferMueller2011} (see also
\eqref{eq:distanceestimate}), \(\rho\) is proper and
\(\rho\simeq d_g(p,\cdot)\) outside a compact set. We also use the standard
nonnegativity \(\scal_\omega\ge0\) of the scalar curvature of a Ricci shrinker \cite{Chen2009Strong}. Consequently,
\[
    |\nabla^{1,0}f|\le \rho,
    \qquad
    |\nabla^{1,0}\rho|\le\frac12.
\]

Let \(\psi^t\), \(t\in[-1,0)\), be the self-similar biholomorphisms generated
by \((2|t|)^{-1}\nabla f\), normalized by \(\psi^{-1}=\mathrm{id}\). Thus
\[
    \frac{d}{dt}\psi^t=(2|t|)^{-1}\nabla f\circ\psi^t.
\]
Set
\[
    \omega_t:=|t|(\psi^t)^*\omega,
    \qquad
    F(x,t):=|t|\rho^2(\psi^t(x)).
\]
Under this convention,
\(\mathcal L_{\frac12\nabla f}\omega=\ii\partial\bar\partial f\).
Then \(\partial_t\omega_t=-\Ric(\omega_t)\), and
\begin{equation}\label{eq:parabolic-F}
    \Delta_{\omega_t}F=n-|t|\scal_{\omega_t},
    \qquad
    \partial_tF=-|t|\scal_{\omega_t},
    \qquad
    (\partial_t-\Delta_{\omega_t})F=-n.
\end{equation}

We use the Fano fibration associated with a K\"ahler--Ricci shrinker, constructed in
\cite[Section~3]{SunZhang}. Let \(\mathbb T\) be the torus generated by the
soliton vector field \(\xi=J\nabla f\), and let
\[
    \pi:X\longrightarrow Y
\]
be the fibration constructed from homogeneous holomorphic functions. Let
\(X^\circ\subset X\) be the largest open set on which \(\pi\) is
biholomorphic onto an open subset of \(Y\), and let
\[
    E:=X\setminus X^\circ
\]
be the exceptional set of \(\pi\). Choose a \(\mathbb T\)-equivariant affine
embedding
\[
    \iota:Y\hookrightarrow\mathbb C^N
\]
given by nonconstant homogeneous generators of the coordinate ring of \(Y\),
and write
\[
    \iota\circ\pi=(h_1,\ldots,h_N).
\]
The functions \(h_\alpha\) are holomorphic and homogeneous, with positive
weights \(\lambda_\alpha\), that is,
\[
    \xi(h_\alpha)=\sqrt{-1}\lambda_\alpha h_\alpha.
\]
The map \(\pi\) is proper. We let
\[
    \omega_Y:=\iota^*\omega_{\mathbb C^N}.
\]
Then its pullback
\[
    \pi^*\omega_Y
    =\sqrt{-1}\sum_\alpha\partial h_\alpha\wedge
      \bar\partial\overline{h_\alpha}
\]
is a smooth semipositive form on \(X\).

\subsection{Homogeneous functions}

\begin{lemma}[Homogeneity and weighted \(L^2\)]\label{lem:homogeneous-flow-growth}
Let \(h\) be a holomorphic function satisfying
\[
    \xi(h)=\sqrt{-1}\lambda h
\]
for some \(\lambda\ge0\). Then
\[
    \nabla f(|h|^2)=2\lambda |h|^2,
    \qquad
    \Delta_f h=-\frac{\lambda}{2}h.
\]
Moreover, for every \(\tau\in(0,1]\),
\[
    \int_X|h|^2e^{-\tau f}\,\omega^n<\infty.
\]
\end{lemma}

\begin{proof}
Since \(h\) is holomorphic, \(\d h(JV)=\sqrt{-1}\,\d h(V)\) for every real
vector field \(V\). Taking \(V=\nabla f\) and using \(\xi=J\nabla f\), we get
\[
    \sqrt{-1}\,\nabla f(h)=\xi(h)=\sqrt{-1}\lambda h,
\]
hence \(\nabla f(h)=\lambda h\). This gives
\(\nabla f(|h|^2)=2\lambda|h|^2\). Since \(h\) is holomorphic,
\(\Delta_\omega h=0\), and therefore
\[
    \Delta_f h=\Delta_\omega h-\frac12\nabla f(h)
    =-\frac{\lambda}{2}h.
\]

It remains to prove the weighted \(L^2\) statement. For
\(\tau\in(0,1]\), the preceding identities and \eqref{eq:identitiesric} give
\begin{align*}
    &\operatorname{div}_g\bigl(|h|^2e^{-\tau f}\nabla f\bigr)\\
    \qquad=&2|h|^2
    \bigl(\Delta_\omega f-\tau|\na^{1,0}f|^2+\lambda\bigr)e^{-\tau f}\\
    \qquad=&2|h|^2\bigl(\lambda+(1-\tau)n-\tau f
    -(1-\tau)\scal_\omega\bigr)e^{-\tau f}.
\end{align*}

Let \(D_R:=\{f\le R\}\). Since \(f\) is proper, \(D_R\) is compact. For a
regular value \(R\), integration by parts gives
\begin{align*}
    &\int_{\partial D_R}|h|^2e^{-\tau f}
    \langle\nabla f,\nu\rangle_g\,d\sigma_g=2\int_{D_R}|h|^2\bigl(\lambda+(1-\tau)n-\tau f
    -(1-\tau)\scal_\omega\bigr)e^{-\tau f}\,\omega^n.
\end{align*}
The left-hand side is nonnegative. Choose
\[
    B_\tau>\frac{\lambda+(1-\tau)n+1}{\tau}.
\]
Since \(\scal_\omega\ge0\), on \(X\setminus D_{B_\tau}\),
\[
    \tau f+(1-\tau)\scal_\omega-\lambda-(1-\tau)n\ge1.
\]
Therefore, for every regular value \(R>B_\tau\),
\begin{align*}
    \int_{D_R\setminus D_{B_\tau}}|h|^2e^{-\tau f}\,\omega^n
    \le{}&
    \int_{D_{B_\tau}}|h|^2
    \bigl|\lambda+(1-\tau)n-\tau f
    -(1-\tau)\scal_\omega\bigr|e^{-\tau f}\,\omega^n.
\end{align*}
The right-hand side is finite because \(D_{B_\tau}\) is compact. Letting
\(R\to\infty\) through regular values proves the claim.
\end{proof}

\begin{lemma}[Local mean-value estimate]\label{lem:local-mean-value-drift}
Let \(v\) be a real solution of
\[
    (\Delta_f+\mu)v=0
\]
on \(X\), where \(\mu\in\R\). Put
\[
    r_x=(1+\rho(x))^{-1}.
\]
Then there is a constant \(C=C(n,A,\mu)\), where \(-A\) is a lower bound for
the entropy of the Ricci shrinker, such that
\[
    |v(x)|^2+r_x^2|\na v(x)|^2
    \le
    Cr_x^{-2n}\int_{B_g(x,r_x)}v^2\,\omega^n.
\]
\end{lemma}

\begin{proof}
This result is well-known to experts. We provide a sketch of the proof for
the reader's convenience.

Let \(w=v^2\). Then
\[
    \Delta_f w=2|\nabla^{1,0}v|^2-2\mu v^2
    \ge-2|\mu|w.
\]
Since \(|\nabla^{1,0}\rho|\le1/2\), the functions \(1+\rho\) and
\(1+\rho(x)\) are uniformly comparable on \(B_g(x,2r_x)\). Moreover,
\(|\nabla f|\le\sqrt{2}\rho\), so \(f\) has uniformly bounded oscillation on
this ball. Rescale the metric by
\[
    \hat g=r_x^{-2}g.
\]
Then
\[
    \Ric_{\hat g}+\na^2_{\hat g}f=r_x^2\hat g\ge0.
\]
Let
\[
    \mathcal L_f u:=\Delta_g u-(\nabla f)(u)=2\Delta_f u.
\]
With respect to \(\hat g\), the corresponding weighted Laplacian is
\(\widehat{\mathcal L}_f=r_x^2\mathcal L_f\), and
\[
    \widehat{\mathcal L}_f w
    =2|\nabla_{\hat g}v|_{\hat g}^2-4\mu r_x^2v^2
    \ge-4|\mu|w.
\]
We can therefore apply the standard De Giorgi--Nash--Moser iteration, thanks
to the Sobolev inequality of \cite[Lemma~3.2]{munteanu2012analysis} applied
to the rescaled smooth metric measure space
\((X,\hat g,e^{-f}\d V_{\hat g})\), to conclude that
\[
    w(x)
    \le
    \frac{C}{\abs{B_{\hat g}(x,1)}_{\hat g,f}}
    \int_{B_{\hat g}(x,1)}we^{-f}\,\d V_{\hat g}.
\]
Here, \(C=C(n,\mu)\), and
\[
    \abs{B_{\hat g}(x,1)}_{\hat g,f}
    :=\int_{B_{\hat g}(x,1)}e^{-f}\,\d V_{\hat g}.
\]
Since \(f\) has bounded oscillation on this ball, the weighted and unweighted
averages are uniformly comparable. Scaling back,
\[
    |v(x)|^2
    \le
    \frac{C}{\abs{B_g(x,r_x)}_g}
    \int_{B_g(x,r_x)}v^2\,\omega^n.
\]
Finally, \(\scal_\omega\le\rho^2\le Cr_x^{-2}\) on
\(B_g(x,2r_x)\). After replacing \(r_x\) by a fixed multiple, which only
changes the constant in the estimate, the standard \(\kappa\)-noncollapsing
of Ricci shrinkers at this scale \cite[Theorem~22]{liwang2020heat1} gives
\[
    \abs{B_g(x,r_x)}_g\ge\kappa r_x^{2n},
\]
where \(\kappa=\kappa(n,A)>0\). This proves
\[
    |v(x)|^2
    \le
    Cr_x^{-2n}\int_{B_g(x,r_x)}v^2\,\omega^n.
\]

For the gradient estimate, recall that \(\mathcal L_fv=-2\mu v\). The
weighted Bochner formula and the identity
\(\Ric_g+\na^2f=g\) give
\[
    \frac12\mathcal L_f|\na v|^2
    =|\na^2v|^2+(1-2\mu)|\na v|^2
    \ge-C(\mu)|\na v|^2.
\]
The same mean-value argument, applied to \(|\na v|^2\) on a smaller ball,
yields
\begin{align}\label{eq:moser2}
    |\na v(x)|^2
    \le
    Cr_x^{-2n}\int_{B_g(x,r_x/2)}|\na v|^2\,\omega^n.
\end{align}
Since \(\Delta_fv+\mu v=0\), a standard weighted integration by parts,
together with the bounded oscillation of \(f\), gives
\[
    \int_{B_g(x,r_x/2)}|\na v|^2\,\omega^n
    \le
    Cr_x^{-2}\int_{B_g(x,r_x)}v^2\,\omega^n.
\]
Combining this with \eqref{eq:moser2} completes the proof.
\end{proof}

\begin{corollary}[Polynomial growth]\label{cor:homogeneous-polynomial-growth}
Every homogeneous holomorphic function on \(X\) has polynomial growth. More
precisely, if \(h\) is a holomorphic function satisfying, for some
\(\lambda\ge0\),
\[
    \xi(h)=\sqrt{-1}\lambda h,
\]
then there is a constant \(C\) such that
\[
    |h(x)|\le C(1+\rho(x))^\lambda.
\]
\end{corollary}

\begin{proof}
We first prove a subexponential pointwise bound. Fix \(\tau\in(0,1]\). By
Lemma~\ref{lem:homogeneous-flow-growth},
\[
    \int_X|h|^2e^{-\tau f}\,\omega^n<\infty.
\]
Let \(v\) be either the real or imaginary part of \(h\). Then
\((\Delta_f+\lambda/2)v=0\). Put
\[
    r_x=(1+\rho(x))^{-1}.
\]
By Lemma~\ref{lem:local-mean-value-drift},
\[
    |v(x)|^2
    \le
    Cr_x^{-2n}\int_{B_g(x,r_x)}v^2\,\omega^n.
\]
Since \(f\le f(x)+C\) on \(B_g(x,r_x)\), the preceding
\(L^2(e^{-\tau f})\) bound implies
\[
    |v(x)|^2
    \le
    C_\tau(1+\rho(x))^{2n}e^{\tau f(x)}
\]
for every \(\tau\in(0,1]\).

We now turn this subexponential bound into a polynomial one by a barrier
argument. The case \(h\equiv0\) is trivial, so assume \(h\not\equiv0\). By
Lemma~\ref{lem:homogeneous-flow-growth}, for every \(\delta>0\),
\[
    \nabla f\bigl(\log(|h|^2+\delta)\bigr)
    =\frac{2\lambda|h|^2}{|h|^2+\delta}
    \le2\lambda.
\]
Combined with the fact that
\(\Delta_\omega\log(|h|^2+\delta)\ge0\), since \(h\) is holomorphic, this
gives
\[
    \Delta_f \log(|h|^2+\delta)\ge-\lambda.
\]
On \(\{f\ge 1\}\),
\[
    \Delta_f \log f
    =\frac{\Delta_ff}{f}-\frac{|\nabla^{1,0}f|^2}{f^2}
    \le-1.
\]
For \(\ep>0\) and \(0<\delta\le1\), set
\[
    G_{\delta,\ep}
    :=\log(|h|^2+\delta)-\lambda\log f-\ep f.
\]
Then, on \(\{f\ge 1\}\),
\[
    \Delta_f G_{\delta,\ep}\ge\ep f>0.
\]
On the outer boundary \(\{f=R\}\), the subexponential estimate above, with
\(0<\tau<\min\{\ep,1\}\), gives
\[
    G_{\delta,\ep}
    \le
    C+2n\log\bigl(1+\sqrt{R+n}\bigr)-\lambda\log R
    -(\ep-\tau)R,
\]
and hence \(G_{\delta,\ep}\to-\infty\) as \(R\to\infty\). The maximum
principle on \(\{1\le f\le R\}\), followed by \(R\to\infty\) through regular
values, therefore gives
\[
    G_{\delta,\ep}
    \le
    \sup_{\{f=1\}}G_{\delta,\ep}
    \le C_1,
\]
where \(C_1\) is independent of \(\delta\) and \(\ep\). Letting
\(\delta\to0\) and then \(\ep\to0\), we get
\[
    |h|^2\le Cf^\lambda
\]
on \(\{f\ge 1\}\). Since \(f=\rho^2-n\), enlarging \(C\) on the compact
set \(\{f\le 1\}\) proves the claim.
\end{proof}

\subsection{Schwarz estimate on the Fano fibration}
\label{subsec:schwarz-lower-bound}

\begin{proposition}\label{prop:schwarz-lower-bound-general}
For every compact set \(K\Subset Y\), there exists a constant \(C_K>0\) such
that
\[
    \omega_t\ge C_K^{-1}\pi^*\omega_Y
\]
on \(\pi^{-1}(K)\) for all \(t\in[-1,0)\).
\end{proposition}

\begin{proof}
Set
\[
    u:=\tr_{\omega_t}\pi^*\omega_Y
    =\sum_{\alpha=1}^N|\partial h_\alpha|^2_{\omega_t}.
\]
We prove that \(u\) is uniformly bounded on \(\pi^{-1}(K)\).

The parabolic Schwarz calculation gives
\[
    (\partial_t-\Delta_{\omega_t})\log(u+1)\le0.
\]
Indeed, since the \(h_\alpha\) are holomorphic, the Bochner formula along the
K\"ahler--Ricci flow gives
\[
    (\partial_t-\Delta_{\omega_t})u
    =-\sum_{\alpha=1}^N|\nabla\partial h_\alpha|_{\omega_t}^2,
\]
and
\[
    |\nabla^{1,0}u|_{\omega_t}^2
    \le
    u\sum_{\alpha=1}^N|\nabla\partial h_\alpha|_{\omega_t}^2.
\]
Hence \((\partial_t-\Delta_{\omega_t})\log(u+1)\le0\).

Since
\[
    (\partial_t-\Delta_{\omega_t})F=-n,
\]
the quantity
\[
    Q:=\log(u+1)-F-(n+1)t
\]
satisfies \((\partial_t-\Delta_{\omega_t})Q\le-1\).

We next record the polynomial growth needed to apply the maximum principle on
the noncompact space. By the gradient estimate in
Lemma~\ref{lem:local-mean-value-drift},
Corollary~\ref{cor:homogeneous-polynomial-growth}, and the standard
polynomial volume upper bound for Ricci shrinkers, after increasing the
exponent if necessary, there exist constants \(C,m>0\) such that
\[
    |\partial h_\alpha|_\omega^2(x)
    \le C(1+\rho(x))^m
\]
for every \(\alpha\). As \(\psi^t\) is generated by
\((2|t|)^{-1}\nabla f\),
\[
    h_\alpha(\psi^t(x))
    =|t|^{-\lambda_\alpha/2}h_\alpha(x).
\]
Differentiating this identity and using
\(\omega_t=|t|(\psi^t)^*\omega\), we get
\[
    |\partial h_\alpha|_{\omega_t}^2(x)
    =|t|^{\lambda_\alpha-1}
      |\partial h_\alpha|_\omega^2(\psi^t(x)).
\]
Since
\[
    F(x,t)=|t|\rho^2(\psi^t(x)),
\]
it follows that, for each fixed \(t_0<0\),
\[
    u(x,t)\le C_{t_0}(1+F(x,t))^{m_{t_0}}
    \qquad\text{on }X\times[-1,t_0].
\]
Consequently, \(Q\to-\infty\) as \(F\to\infty\), uniformly for
\(t\in[-1,t_0]\), and
\[
    \sup_XQ(\cdot,-1)<\infty.
\]
Since \(\rho\) is proper and each \(\psi^t\) is a diffeomorphism, the
sublevel sets of \(F\) in \(X\times[-1,t_0]\) are compact. Applying the
maximum principle on the compact spacetime sublevel sets \(\{F\le R\}\) and
then letting \(R\to\infty\) through regular values gives
\[
    Q(x,t)\le\sup_XQ(\cdot,-1)
    \qquad\text{for }t\in[-1,t_0].
\]
Since \(t_0<0\) is arbitrary,
\[
    \log(u(x,t)+1)\le C+F(x,t)
    \qquad\text{on }X\times[-1,0).
\]

Now let \(K\Subset Y\). Since \(\pi\) is proper, \(\pi^{-1}(K)\) is compact.
Moreover, \(\partial_tF=-|t|\scal_{\omega_t}\le0\), so
\[
    F(x,t)\le F(x,-1)=\rho^2(x)
    \qquad\text{for }x\in\pi^{-1}(K),\quad t\in[-1,0).
\]
The right-hand side is bounded on \(\pi^{-1}(K)\). Thus
\[
    u=\tr_{\omega_t}\pi^*\omega_Y\le C_K
\]
on \(\pi^{-1}(K)\), which is equivalent to
\(\pi^*\omega_Y\le C_K\omega_t\).
\end{proof}

\begin{corollary}[Restatement of Theorem~\ref{thm:local-smooth-convergence}]\label{cor:smooth-convergence-generic-schwarz}
With \(X^\circ=X\setminus E\) as above, the metrics \(\omega_t\) converge to
a K\"ahler cone metric \(\omega_0\) in
\(C^\infty_{\mathrm{loc}}(X^\circ)\) as \(t\nearrow0\).
\end{corollary}

\begin{proof}
Let \(K\Subset X^\circ\). Choose \(U\Subset X^\circ\) with \(K\Subset U\).
Since \(\pi|_{X^\circ}\) is biholomorphic onto its image, \(\pi^*\omega_Y\)
is a smooth K\"ahler form on \(U\). After shrinking \(U\), there is \(c_U>0\)
such that
\[
    \pi^*\omega_Y\ge c_U\omega_{-1}
\]
on \(U\). Proposition~\ref{prop:schwarz-lower-bound-general}, applied to the
compact set \(\pi(\overline U)\subset Y\), gives
\[
    \omega_t\ge C_U^{-1}\omega_{-1}
\]
on \(U\). Since \(\scal_{\omega_t}\ge0\), the K\"ahler--Ricci flow volume
evolution gives
\[
    \partial_t\log\frac{\omega_t^n}{\omega_{-1}^n}
    =-\scal_{\omega_t}\le0,
\]
and hence \(\omega_t^n\le\omega_{-1}^n\). Combining the metric lower bound
with this volume-form upper bound gives
\[
    C_U^{-1}\omega_{-1}\le\omega_t\le C_U\omega_{-1}
\]
on \(U\). Sherman--Weinkove interior estimates
\cite[Theorem~1.1 and Corollary~1.2]{SWein} give uniform bounds for all
covariant derivatives of \(\Rm_{\omega_t}\) on \(K\times[-1/2,0)\). Since
\(\partial_t\omega_t=-\Ric(\omega_t)\), the metrics are Cauchy in
\(C^\infty(K)\) as \(t\nearrow0\).

It remains to show that the limit \(\omega_0\) is a cone metric. On every
\(K\Subset X^\circ\), the preceding curvature estimates and
\eqref{eq:parabolic-F} give
\[
    |\partial_tF|=|t|\scal_{\omega_t}\le C_K|t|.
\]
Thus \(F(\cdot,t)\) is Cauchy in \(C^0(K)\) and converges to a function
\(F_0\). The elliptic equation
\[
    \Delta_{\omega_t}F=n-|t|\scal_{\omega_t},
\]
together with the local smooth convergence of \(\omega_t\), implies by
interior elliptic estimates that
\[
    F(\cdot,t)\longrightarrow F_0
    \qquad\text{in }C^\infty_{\mathrm{loc}}(X^\circ).
\]
From the K\"ahler--Ricci shrinker equation and
\eqref{eq:identitiesric}, we have
\[
    |t|\Ric(g_t)+\nabla^2F(\cdot,t)=g_t,
    \qquad
    |\nabla F(\cdot,t)|_{g_t}^2
    =2F(\cdot,t)-2|t|^2\scal_{\omega_t}.
\]
Passing to the limit gives
\[
    \nabla^2F_0=g_0,
    \qquad
    |\nabla F_0|_{g_0}^2=2F_0.
\]
Setting \(r=\sqrt{2F_0}\), we obtain on \(\{r>0\}\)
\[
    |\nabla r|_{g_0}=1,
    \qquad
    \nabla^2r=\frac1r\bigl(g_0-\d r\otimes\d r\bigr).
\]
Thus, the flow generated by \(r\partial_r=\nabla F_0\) acts by homotheties,
and \(g_0\) has the form
\[
    g_0=\d r^2+r^2g_\Sigma.
\]
Therefore, the limit K\"ahler metric \((g_0,\omega_0)\) is a K\"ahler cone
metric.
\end{proof}

\subsection{A complex-analytic criterion for asymptotic conicality}

\begin{theorem}[Restatement of Theorem~\ref{thm:ac-criterion}]\label{thm:conical-criterion}
Let \((X,g,f)\) be a noncompact K\"ahler--Ricci shrinker. Then
\((X,g)\) is asymptotically conical if and only if the associated Fano
fibration \(\pi:X\to Y\) is biholomorphic outside a compact set.
\end{theorem}

\begin{proof}
If \((X,g)\) is asymptotically conical, then \(\pi\) is biholomorphic outside
a compact set by \cite{CSunD}. We prove the converse.

Assume that \(\pi\) is biholomorphic outside a compact set. Let \(E\) be the
exceptional set of \(\pi\); then \(E\) is compact. Choose a regular value
\(C_0\gg1\) of \(\rho\) with
\[
    E\subset\{\rho<C_0\},
\]
and set
\[
    \Sigma:=\{\rho=C_0\}\Subset X\setminus E.
\]
By Corollary~\ref{cor:smooth-convergence-generic-schwarz}, there is
\(A<\infty\) such that
\[
    \sup_{\Sigma\times[-1,0)}|\Rm_{\omega_t}|_{\omega_t}\le A.
\]

Let \(x\in X\) with \(\rho(x)>C_0\). Since \(f\), equivalently \(\rho\), has
no critical point on \(\{\rho\ge C_0\}\) \cite{SunZhang}, the flow generated
by \(-\frac12\nabla f\) from \(x\) reaches \(\Sigma\) in finite time. Thus,
there are \(T>0\) and \(y\in\Sigma\) such that, with \(t=-e^{-T}\),
\(\psi^t(y)=x\). Indeed,
\[
    \frac{\d}{\d s}\psi^{-e^{-s}}(y)
    =\frac12\nabla f(\psi^{-e^{-s}}(y)).
\]
Along the curve \(\psi^{-e^{-s}}(y)\), using \(\scal_\omega\ge0\),
\[
    \frac{\d}{\d s}\log\rho^2(\psi^{-e^{-s}}(y))
    =\frac{|\nabla^{1,0}f|^2}{\rho^2}(\psi^{-e^{-s}}(y))
    =1-\frac{\scal_\omega}{\rho^2}(\psi^{-e^{-s}}(y))
    \le1.
\]
Since \(\rho(y)=C_0\), this gives
\[
    T\ge\log\frac{\rho^2(x)}{C_0^2},
    \qquad
    e^{-T}\le\frac{C_0^2}{\rho^2(x)}.
\]
The scaling relation \(\omega_t=|t|(\psi^t)^*\omega\) gives
\[
    |\Rm_g|_g(x)
    =|t|\,|\Rm_{\omega_t}|_{\omega_t}(y)
    \le\frac{AC_0^2}{\rho^2(x)}
    \le Cd_g(p,x)^{-2}.
\]
The asymptotic-cone theorems in
\cite{KW} and \cite{CSunD} imply that \((X,g)\) is asymptotically conical.
\end{proof}

As an application of this complex-analytic criterion and building on the
results of Esparza \cite{esparza2025}, we obtain the following.

\begin{corollary}\label{cor:uniqueness-shrinker-structures}
Let \((X,\omega,\xi)\) be an asymptotically conical K\"ahler--Ricci shrinker.
Let \((\widetilde\omega,\widetilde\xi)\) be another K\"ahler--Ricci shrinker
on the same complex manifold. Then there is a biholomorphism \(\psi:X\to X\)
such that
\[
    \psi_*\xi=\widetilde\xi,
    \qquad
    \psi^*\widetilde\omega=\omega.
\]
\end{corollary}

\begin{proof}
By \cite[Corollary~5.18]{esparza2025}, the Fano fibration structure associated
to \((\widetilde\omega,\widetilde\xi)\) is the same as the one induced by
\((\omega,\xi)\) and hence is biholomorphic outside a compact set. Therefore
Theorem~\ref{thm:conical-criterion} implies that
\((X,\widetilde\omega,\widetilde\xi)\) is asymptotically conical. The
uniqueness theorem for asymptotically conical K\"ahler--Ricci shrinkers
\cite[Theorem~1.1]{esparza2025} gives the desired biholomorphism.
\end{proof}

\begin{proof}[Proof of Theorem~\ref{thm:special-uniqueness}]
The compact case was proved in \cite{TianZ1}; it therefore remains to consider the case where $X$ is noncompact. Let \(X\) be a noncompact complex manifold containing no compact positive-dimensional
analytic subvarieties at infinity, and suppose that it admits a K\"ahler--Ricci shrinker. By \cite{SunZhang}, it admits a polarized Fano
fibration structure
\[
    \pi:X\longrightarrow Y,
\]
where \(\pi\) is proper, \(Y\) is an affine variety of positive dimension, and
\(\pi_*\mathcal O_X=\mathcal O_Y\).

Choose a compact set \(K\Subset X\) such that \(X\setminus K\) contains no
compact positive-dimensional analytic subset, and set \(B:=\pi(K)\). Since $\pi$ is proper, \(\pi_*\mathcal O_X=\mathcal O_Y\), and $\pi^{-1}(Y\setminus B)$ contains no positive-dimensional compact analytic subset, $\pi$ is injective on $\pi^{-1}(Y\setminus B)$. Therefore, by the surjectivity and the properness of $\pi$, the restriction of $\pi$ to $\pi^{-1}(Y\setminus B)$ is a biholomorphism; that is, \(\pi\) is biholomorphic outside a compact set.
By Theorem~\ref{thm:conical-criterion}, every such shrinker is
asymptotically conical. Therefore, by
Corollary~\ref{cor:uniqueness-shrinker-structures}, \(X\) has at most
one K\"ahler--Ricci shrinker structure.
\end{proof}

\section{Ricci Flow Completions of Ricci Shrinkers}\label{sec:generalshrinker}
In this section, we consider a Ricci shrinker $(M^n,g,f)$ of real dimension $n$. Throughout, we assume that the scalar curvature $\scal$ is uniformly bounded above by $C_0$ and that the entropy is bounded below by $-Y$. We normalize the Ricci shrinker equation by
\begin{equation}\label{eq:normal1}
    \Ric(g)+\nabla^2 f=\frac{1}{2}g,
    \qquad
    \scal+|\nabla f|^2=f.
\end{equation}

\subsection{Spacetime completion and the time-zero potential}
We consider the associated self-similar Ricci flow $(M,(g_t)_{t \in (-\infty, 0)},(f_t)_{t \in (-\infty, 0)})$ satisfying
\[
    (g_{-1},f_{-1})=(g,f).
\]
Recall that if $\psi^t:M\to M$ denotes the family of diffeomorphisms generated by
\(
    |t|^{-1}\nabla f,
\)
normalized by $\psi^{-1}=\mathrm{id}$, then
\begin{align*}
    g_t=|t|(\psi^t)^*g,
    \qquad
    f_t=(\psi^t)^*f.
\end{align*}
Choose a base point $p$ at a minimum of $f$ on $M$, and set
\(
    F(x,t)=|t|f_t(x).
\)
By \cite[Lemma 1]{liwang2020heat1}, this function satisfies the distance estimate
\begin{equation}\label{eq:distanceestimate}
    \frac{1}{4}\left(d_{g_t}(p,x)-C(n)|t|^{1/2}\right)_+^2
    \le F(x,t)
    \le
    \frac{1}{4}\left(d_{g_t}(p,x)+C(n)|t|^{1/2}\right)^2.
\end{equation}
We also have the uniform volume upper bound from \cite[Lemma 2]{liwang2020heat1}:
for every \(t\in[-1,0)\) and \(L\ge 1\),
\begin{equation}\label{eq:volume-upper-bound}
      \left|B_{g_t}(p,L)\right|_{g_t}\le C(n)L^{n}.
\end{equation}

We next recall that the theory of noncollapsed Ricci flow limit spaces extends to the Ricci flow
\[
    (M,(g_t)_{t\in(-\infty,0)}).
\]
Indeed, this Ricci flow has entropy bounded below by $-Y$; see \cite[Section 5]{liwang2020heat1}. Moreover, it was proved in \cite{liwang2020heat1} and \cite{liwang2024heat} that all results in \cite[Section 2]{FL1} hold for this Ricci flow, except that in \cite[Theorem 2.19]{FL1} one needs to impose a polynomial growth condition on \(u\); see \cite[Theorem 4.20]{liwang2024heat}. On the other hand, the theory of \(\mathbb F\)-convergence has also been generalized to Ricci flows induced by Ricci shrinkers; see \cite{liwang2024heat}.

Thus, following the variational definition in \cite[Definition 3.2]{FL1}, we define the spacetime distance
\begin{equation}\label{eq:spacetime-distance-shrinker}
    d^*(x^*,y^*)
    :=
    \inf_{\tau\le s}
    \max\left\{
        \sqrt{t-\tau},
        d_{W_1}^{\tau}(\nu_{x^*;\tau},\nu_{y^*;\tau})
    \right\}
\end{equation}
for any \(x^*=(x,t)\), \(y^*=(y,s)\in M\times(-\infty,0)\) with \(s\le t\).

For every fixed \(t_0<0\), the \(d^*\)-distance on \(M \times(-\infty,t_0]\) is complete. This follows by the same argument as in \cite[Theorem 7.19]{FL1}. Note that the boundedness of the scalar curvature is crucial here.

As in Section 2, we then define
\(
    (Z,d_Z,\mathfrak t)
\)
to be the metric completion of \(M\times(-\infty,0)\) with respect to \(d^*\). The completion adds points only on the time slice
\[
    Z_0:=\mathfrak t^{-1}(0).
\]

The parabolic metric space \((Z,d_Z,\mathfrak t)\), which is defined over \((-\infty,0]\), shares the same properties as a noncollapsed Ricci flow limit space in the sense of \cite{FL1}. Roughly speaking, \((Z,d_Z,\mathfrak t)\) can be regarded as the Gromov--Hausdorff limit of
\(
    (M\times(-\infty,t_j],d^*)
\)
as \(t_j\nearrow 0\). Note that the Ricci flow
\(
    (M,(g_t)_{t\in(-\infty,t_j]})
\)
may not have uniformly bounded curvature. However, the heat kernel estimates in \cite{liwang2020heat1} and \cite{liwang2024heat} guarantee that the structure theory developed in \cite{FL1,fang-li-unique,fang-li2} still holds in this setting.
We can now define the intrinsic distance \(d^Z_t\) on \(Z_t\) for each \(t\in(-\infty,0]\), as in \eqref{eq:time-slice-distance}. We can also assign a base point \(\bar p\in Z_0\), defined as the limit of
\[
    p_t^*:=(p,t)
\]
as \(t\nearrow 0\). This limit always exists; see \cite[Section 7]{FL1}. Moreover, each point \((p,t)\) is an \(H\)-center of \(\bar p\), where $H$ depends on $n$ and $Y$. Thus, we obtain the pointed metric space
\[
    (Z_0,d^Z_0,\bar p).
\]
Unlike in the compact setting, $Z_0$ is noncompact.

We extend $F$ to the time-zero slice $Z_0$ by setting
\begin{align} \label{eq:Fon0}
 F(z)=\frac{1}{4}\bigl(d^Z_0(z,\bar p)\bigr)^2,
    \qquad z\in Z_0.
\end{align}

Arguing as in \cite[Proposition 6.6]{FL1}, $(Z_0, d^Z_0)$ is a complete extended metric space. In fact, we prove the following continuity statement, which, in particular, implies that $(Z_0, d^Z_0)$ is a complete metric space.

\begin{lemma}\label{lem:continuityF}
The function $F$ is continuous on $Z$ with respect to $d^*$. In particular, $d^Z_0$ is finite-valued on $Z_0$.
\end{lemma}

\begin{proof}
It suffices to prove continuity at points in the time-zero slice $Z_0$. Let $z\in Z_0$, and suppose that
\[
    x_i^*=(x_i,t_i)\to z
\]
with respect to $d^*$. By the definition of $d^*$, for every fixed $s<0$ we have
\begin{align*}
    \lim_{i\to\infty}
    d_{W_1}^{s}\bigl(\nu_{z;s},\nu_{x_i^*;s}\bigr)=0.
\end{align*}
Let $z_i^*=(z_i,s)$ be an $H_n$-center of $z$. Since $(x_i,s)$ is an $H$-center of $x_i^*$ by the scalar curvature bound, it follows that
\begin{align} \label{eq:continu001}
    d_{g_s}(z_i,x_i)\le C\sqrt{|s|}+\ep_i,
\end{align}
where $\ep_i\to0$ as $i\to\infty$. On the other hand, since
\[
    \partial_t F=-|t|\scal\in[-C_0,0],
\]
we have, for $i$ sufficiently large,
\begin{align} \label{eq:continu002}
    F(x_i,t_i)\le F(x_i,s)\le F(x_i,t_i)+C_0|s|.
\end{align}
Furthermore, by \eqref{eq:distanceestimate} and \eqref{eq:continu001}, we obtain
\begin{align}\label{eq:continu003}
    \left|2\sqrt{F(x_i,s)}-d_{g_s}(p,z_i)\right|
    \le C\sqrt{|s|}+\ep_i,
\end{align}
where $\ep_i\to0$ as $i\to\infty$.

Combining \eqref{eq:continu001}, \eqref{eq:continu002}, \eqref{eq:continu003}, and the definition of the time-zero distance $d^Z_0$, we obtain
\begin{align*}
    \limsup_{i\to\infty}
    \left|2\sqrt{F(z)}-2\sqrt{F(x_i,t_i)}\right|
    \le \Psi(|s|),
\end{align*}
where $\Psi(r)\to0$ as $r\to0$. Letting $s\nearrow 0$ gives
\[
    F(x_i,t_i)\to F(z).
\]
Therefore, $F$ is continuous at $z$. In particular, every point \(z\in Z_0\) satisfies
\[
    d^Z_0(z,\bar p)<\infty.
\]
For arbitrary \(z_1,z_2\in Z_0\), the triangle inequality for the extended
distance $d^Z_0$ gives
\[
    d^Z_0(z_1,z_2)
    \le d^Z_0(z_1,\bar p)+d^Z_0(\bar p,z_2)<\infty.
\]
Thus \(d^Z_0\) is finite-valued, completing the proof.
\end{proof}

\begin{definition}
Let $(M,g)$ be a Ricci shrinker. Any sequential limit of
\[
(M,\lambda_j^{-2}g,p),
    \qquad \lambda_j\to\infty,
\]
in the pointed Gromov--Hausdorff sense
is called a \textbf{tangent space at infinity} of \((M,g)\).
\end{definition}

A priori, such a tangent space may not exist; even if it exists, it need not be unique.
Using the self-similarity of \((M,(g_t)_{t\in(-\infty,0)})\), a tangent space at infinity of $M$ is the same as a sequential limit of $(M, g_{t_i},p)$ in the pointed Gromov--Hausdorff sense for some sequence $t_i\nearrow 0$. Indeed, for every sequence \(\lambda_j\to+\infty\), choose \(t_j<0\) such that
\[
    |t_j|^{-1/2}=\lambda_j.
\]
Then the map \(\psi^{t_j}\) induces an isometry from \((M,g_{t_j},p)\) to \((M,\lambda_j^{-2}g,p)\).

\subsection{Metric fibration and conical structure of the time-zero slice}\label{sec:singular-time-slice}
We investigate the self-similar structure of the tangent space \((Z_0,d^Z_0)\). As in \cite[Proposition 7.30]{FL1}, there exists a flow $\{\boldsymbol{\psi}^s\}_{s\in\mathbb R}$ on \(Z_0\) such that:
\begin{enumerate}[label=\textnormal{(\alph*)}]
\item For every \(s\in\mathbb R\),
\[
    \boldsymbol{\psi}^s(\bar p)=\bar p.
\]

\item For every \(x,y\in Z_0\),
\[
    d^Z_0(\boldsymbol{\psi}^s(x),\boldsymbol{\psi}^s(y))
    =
    e^{-s/2}d^Z_0(x,y).
\]

\item For every \(x\in Z_0\) and \(\tau>0\), the Nash entropy satisfies
\[
    \mathcal N_{\boldsymbol{\psi}^s(x)}(e^{-s}\tau)
    =
    \mathcal N_x(\tau).
\]
\end{enumerate}

More precisely, on \(M\times(-\infty,0)\), the flow \(\boldsymbol{\psi}^s\) is the spacetime flow generated by $|t|(\partial_t-\nabla f_t)$. We then extend \(\boldsymbol{\psi}^s\) to \(Z\) by taking limits with respect to \(d_Z\).
By \cite[Proposition~7.30(i)--(ii)]{FL1}, for every \(s\in\mathbb R\), we obtain a homothety
\[
    \boldsymbol{\psi}^s:
    \Sigma:=\{x\in Z_0\mid d^Z_0(x,\bar p)=1\}
    \longrightarrow
    \{x\in Z_0\mid d^Z_0(x,\bar p)=e^{-s/2}\},
\]
with homothety factor \(e^{-s/2}\), where \(\Sigma\) is equipped with the restricted distance. In particular, \(Z_0\) is homeomorphic to a cone over \(\Sigma\).

Next, we construct a natural map \(\Phi:M \to Z_0\) and establish some basic properties of $\Phi$.

\begin{lemma}\label{lem:worldline}
For every \(x\in M\), the limit
\(
    \lim_{t\nearrow 0}(x,t)
\)
exists in \((Z,d_Z)\). We therefore define a map \(\Phi:M\to Z_0\) by
\begin{align}
\Phi(x):=\lim_{t\nearrow 0}(x,t).\label{eq:analytic-projection}
\end{align}
\end{lemma}

\begin{proof}
By our assumption on the scalar curvature, we have
\[
    0\le \scal\le C_0|t|^{-1}
\]
on \(M\times(-\infty,0)\). Thus, by \cite[Proposition 4.4]{LW26}, for every \(x\in M\) and \(s<t<0\), the point \((x,s)\) is an \(H\)-center of \((x,t)\), where \(H\) depends only on \(C_0\), \(n\), and $Y$. Hence, by \cite[Lemma 3.13]{FL1},
\[
    d_Z\bigl((x,t),(x,s)\bigr)\le C\sqrt{t-s}.
\]
It follows that \((x,t)\) converges in \((Z,d_Z)\) as \(t\nearrow 0\).
\end{proof}

\begin{lemma}\label{lem:dis}
The map \(\Phi:M\to Z_0\) has the following properties.
\begin{enumerate}[label=(\arabic*), font=\normalfont]
    \item For every \(x,y\in M\),
    \[
        d^Z_0(\Phi(x),\Phi(y))
        =
        \lim_{t\nearrow 0} d_{g_t}(x,y).
    \]
    \item For every \(x,y\in M\) and \(t<0\),
    \[
        d_{g_t}(x,y)\le d^Z_0(\Phi(x),\Phi(y))+C|t|^{1/2}.
    \]
      \item The map \(\Phi:M\to Z_0\) is continuous and proper with respect to the
    topology induced by the ambient spacetime distance \(d_Z\).
    \item The map \(\Phi:M\to Z_0\) is surjective.
\end{enumerate}
\end{lemma}

\begin{proof}
As shown above, \((x,s)\) is an \(H\)-center of \((x,t)\) for every \(x\in M\) and \(s<t<0\). Passing to the limit as \(t\nearrow 0\), we see that \((x,s)\) is an \(H\)-center of \(\Phi(x)\). Therefore, by the definition of \(d^Z_0\), the first statement follows.

On the other hand, by the monotonicity of \(d_{W_1}^{t}\), we have
\[
    d_{g_s}(x,y)\le d_{g_t}(x,y)+C\sqrt{t-s}
\]
for every \(x,y\in M\) and \(s<t<0\). Letting \(t\nearrow 0\), we obtain the second statement.

We next prove item (3). Let \(x_i\to x\) in \((M,g_{-1})\). Fix
\(s<0\). As above, \((x_i,s)\) and \((x,s)\) are \(H\)-centers of
\(\Phi(x_i)\) and \(\Phi(x)\), respectively. Hence \cite[Lemma~6.14]{FL1}
gives
\[
    d_Z\bigl(\Phi(x_i),(x_i,s)\bigr)\le C|s|^{1/2},
    \qquad
    d_Z\bigl(\Phi(x),(x,s)\bigr)\le C|s|^{1/2}.
\]
Moreover, since \(d^Z_s=d_{g_s}\) on the smooth time slice \(Z_s\), by
\cite[Lemma 6.4]{FL1} we have
\[
    d_Z\bigl((x_i,s),(x,s)\bigr)
    \le d_{g_s}(x_i,x).
\]
For fixed \(s<0\), the metric \(g_s\) induces the same topology as \(g\),
so \(d_{g_s}(x_i,x)\to0\). Therefore
\[
    \limsup_{i\to\infty}d_Z\bigl(\Phi(x_i),\Phi(x)\bigr)
    \le C|s|^{1/2}.
\]
Letting \(s\nearrow0\), we obtain \(\Phi(x_i)\to\Phi(x)\) with respect to
\(d_Z\).

Fix \(z\in Z_0\), and let \(t_j\nearrow 0\). For each \(j\), choose an
\(H_n\)-center \(x_j^*=(x_j,t_j)\) of \(z\) at time \(t_j\). On the other hand, \(x_j^*\) is an
\(H\)-center of \(\Phi(x_j)\), where \(H\) is independent of \(j\). Thus
\(z\) and \(\Phi(x_j)\) admit a common center at time \(t_j\), with
uniform \(H\)-constant. The \(H\)-center estimate \cite[Lemma 6.14]{FL1} gives
\[
    d_Z\bigl(z,\Phi(x_j)\bigr)\le C |t_j|^{1/2}.
\]
Letting \(j\to\infty\) shows that \(\Phi(M)\) is dense in $Z_0$ with respect to $d_Z$.

We show that \(\Phi:M\to (Z_0,d_Z)\)
is proper. Let \(K\subset Z_0\) be \(d_Z\)-compact. By the continuity of
\(F\) and the identity \(F=\frac14 d^Z_0(\cdot,\bar p)^2\) on \(Z_0\), the
quantity \(d^Z_0(\cdot,\bar p)\) is bounded on \(K\). Hence, for
\(x\in\Phi^{-1}(K)\), item~(2), applied with \(t=-1\) and
\(y=p\), gives
\[
    d_{g_{-1}}(p,x)\le d^Z_0\bigl(\bar p,\Phi(x)\bigr)+C\le C_K.
\]
Since $\Phi^{-1}(K)$ is closed and bounded in the complete Riemannian manifold $(M,g)$, it is compact by the Hopf--Rinow theorem. Thus \(\Phi\) is proper, and hence
\(\Phi(M)\) is closed in the metric space \((Z_0,d_Z)\). Since \(\Phi(M)\)
is also dense, we have \(\Phi(M)=Z_0\).
\end{proof}

\begin{remark}
The density argument in Lemma~\ref{lem:dis} uses the extrinsic \(d_Z\)-distance,
not the intrinsic \(d^Z_0\)-distance. Indeed, a common \(H\)-center
\(x_j^*=(x_j,t_j)\) gives the same-time estimate
\[
    d_{W_1}^{t_j}\bigl(\nu_{z;t_j},
        \nu_{\Phi(x_j);t_j}\bigr)
    \le C |t_j|^{1/2},
\]
by the definition of an \(H\)-center \cite[Definition 2.13]{FL1}. Since
the \(W_1\)-distance is nondecreasing in time,
this estimate controls earlier times, but it does not control the limit
as \(s\nearrow0\) in the definition of \(d^Z_0\).
\end{remark}

\begin{lemma}\label{lem:disa}
If $(Z_0,d^Z_0)$ is locally compact, then the topologies induced by $d_Z$ and $d^Z_0$ coincide. In particular, Lemma~\ref{lem:dis} shows that if $(Z_0, d^Z_0)$ is locally compact, then $\Phi$ is surjective, continuous, and proper from $(M, g_{-1})$ to $(Z_0, d^Z_0)$.
\end{lemma}

\begin{proof}
By \cite[Lemma~6.4]{FL1},
\[
    d_Z(z_1,z_2)\le d^Z_0(z_1,z_2),
    \qquad z_1,z_2\in Z_0.
\]
Therefore, to show that $d^Z_0$ and $d_Z$ induce the same topology, it is enough to prove that $d_Z(y_i,y)\to0$ implies $d^Z_0(y_i,y)\to0$. Lemma~\ref{lem:continuityF} implies that
\(
    F=\frac14\bigl(d^Z_0(\cdot,\bar p)\bigr)^2
\)
is \(d_Z\)-continuous on \(Z_0\). Hence \(F(y_i)\to F(y)\), and is therefore bounded. Therefore
the sequence \(\{y_i\}\) is contained in a bounded \(d^Z_0\)-ball centered
at \(\bar p\).

We claim that, under the local compactness assumption, every closed
\(d^Z_0\)-ball centered at \(\bar p\) is compact. Indeed, local compactness at
\(\bar p\) gives a compact neighborhood \(K\) of \(\bar p\). Choose \(\rho>0\)
so that \(B_{d^Z_0}(\bar p,2\rho)\subset K\). Then
\(\overline B_{d^Z_0}(\bar p,\rho)\) is a closed subset of \(K\), hence is
compact. The self-similar flow
\(\boldsymbol{\psi}^s\) constructed above satisfies
\[
    d^Z_0\bigl(\boldsymbol{\psi}^s z,\bar p\bigr)
    =
    e^{-s/2}d^Z_0(z,\bar p),
\]
and hence maps \(\overline B_{d^Z_0}(\bar p,\rho)\) homeomorphically onto
\(\overline B_{d^Z_0}(\bar p,e^{-s/2}\rho)\). Varying \(s\) proves the claim.

It follows that every subsequence of \(\{y_i\}\) has a further
subsequence converging with respect to \(d^Z_0\), say to \(z\in Z_0\).
 Since the whole
sequence \(y_i\) converges to \(y\) with respect to \(d_Z\), we must
have \(z=y\). Thus, we have proved that
\(d^Z_0\bigl(y_i,y\bigr)\to0.
\)
\end{proof}

\begin{proposition}[=Theorem~\ref{thm--locally compact criterion}]\label{prop:local-compactness-pgh}
Let $(M^n,g,f)$ be a Ricci shrinker with bounded scalar curvature. Suppose that $(Z_0,d^Z_0)$ is locally
compact. Then
\[
    (M,d_{g_t},p)
    \xrightarrow[t\nearrow0]{\pGHconvtext}
    (Z_0,d^Z_0,\bar p).
\]
In particular, $(Z_0,d^Z_0)$ is the unique tangent space at infinity of $(M,g)$ and $(Z_0,d^Z_0)$ is also a geodesic space.
\end{proposition}

\begin{proof}
By the proof of Lemma~\ref{lem:disa}, every closed
$d^Z_0$-ball centered at $\bar p$ is compact. Moreover,
Lemma~\ref{lem:disa} implies that
\[
    \Phi:(M,g_{-1})\longrightarrow (Z_0,d^Z_0)
\]
is continuous, while Lemma~\ref{lem:dis}(4) implies that $\Phi$ is surjective.

For $R>0$, set
\[
    K_R
    :=
    \Phi^{-1}\left(
        \overline B_{d^Z_0}(\bar p,R)
    \right).
\]
The set $K_R$ is compact in $(M,g_{-1})$.

We first claim that
\begin{equation}\label{eq:uniform-distance-convergence}
    \lim_{t\nearrow0}
    \sup_{x,y\in K_R}
    \left|
        d_{g_t}(x,y)
        -
        d^Z_0(\Phi(x),\Phi(y))
    \right|
    =0.
\end{equation}
Suppose otherwise. Then there exist $\ep>0$, $t_i\nearrow0$, and
$x_i,y_i\in K_R$ such that
\[
    \left|
        d_{g_{t_i}}(x_i,y_i)
        -
        d^Z_0(\Phi(x_i),\Phi(y_i))
    \right|
    \ge\ep.
\]
After passing to a subsequence, we may assume that
\[
    x_i\to x,
    \qquad
    y_i\to y
\]
in $(M,g_{-1})$. By Lemma~\ref{lem:disa},
\[
    d^Z_0(\Phi(x_i),\Phi(x))\to0,
    \qquad
    d^Z_0(\Phi(y_i),\Phi(y))\to0.
\]
Lemma~\ref{lem:dis}(2) therefore gives
\[
    d_{g_{t_i}}(x_i,x)
    \le
    d^Z_0(\Phi(x_i),\Phi(x))+C|t_i|^{1/2}
    \to0,
\]
and similarly
\[
    d_{g_{t_i}}(y_i,y)\to0.
\]
Hence
\[
    \left|
        d_{g_{t_i}}(x_i,y_i)-d_{g_{t_i}}(x,y)
    \right|
    \to0.
\]
On the other hand, Lemma~\ref{lem:dis}(1) gives
\[
    d_{g_{t_i}}(x,y)
    \to
    d^Z_0(\Phi(x),\Phi(y)),
\]
while the continuity of $\Phi$ gives
\[
    d^Z_0(\Phi(x_i),\Phi(y_i))
    \to
    d^Z_0(\Phi(x),\Phi(y)).
\]
This is a contradiction, proving
\eqref{eq:uniform-distance-convergence}.

We next establish uniform radial control. Since
\[
    \partial_tF=-|t|\scal_{g_t}\in[-C_0,0],
\]
and $F(x,t)\to F(\Phi(x))$ as $t\nearrow0$, we have
\[
    0\le F(x,t)-F(\Phi(x))\le C_0|t|.
\]
Furthermore, \eqref{eq:distanceestimate} implies
\[
    \left|
        d_{g_t}(p,x)-2\sqrt{F(x,t)}
    \right|
    \le C(n)|t|^{1/2}.
\]
Since
\[
    2\sqrt{F(\Phi(x))}
    =
    d^Z_0(\bar p,\Phi(x)),
\]
we obtain
\begin{equation}\label{eq:uniform-radial-control}
    \left|
        d_{g_t}(p,x)
        -
        d^Z_0(\bar p,\Phi(x))
    \right|
    \le C_1|t|^{1/2}
\end{equation}
for every $x\in M$.

Set $\eta(t)=C_1|t|^{1/2}$. For every fixed $R>0$,
\eqref{eq:uniform-radial-control} gives
\[
    \Phi\bigl(B_{g_t}(p,R)\bigr)
    \subset
    B_{d^Z_0}(\bar p,R+\eta(t)).
\]
Conversely, using the surjectivity of $\Phi$, we obtain
\[
    B_{d^Z_0}(\bar p,R-\eta(t))
    \subset
    \Phi\bigl(B_{g_t}(p,R)\bigr).
\]
For $t$ sufficiently close to zero,
\[
    B_{g_t}(p,R)\subset K_{R+1}.
\]
Thus, by \eqref{eq:uniform-distance-convergence}, the restriction of $\Phi$
to $B_{g_t}(p,R)$ has distortion converging to zero. Together with the two
preceding inclusions, this proves that $\Phi$ is a pointed
$(R,\ep(t))$-approximation, where $\ep(t)\to0$ as $t\nearrow0$.
Therefore
\[
    (M,d_{g_t},p)
    \xrightarrow[t\nearrow0]{\pGHconvtext}
    (Z_0,d^Z_0,\bar p).
\]
\end{proof}

\begin{remark}
One can also consider a compact Ricci flow $(M, (g_t)_{t \in [-1,0)})$ such that $0$ is the first singular time and the scalar curvature has a type-I bound, that is,
\begin{align*}
    \scal \le \frac{C_0}{|t|},
\end{align*}
on $M \times [-1,0)$, for a constant $C_0$. The assumption on the scalar curvature guarantees that the map $\Phi:M \to Z_0$ as in \eqref{eq:analytic-projection} is well defined. Moreover, $\Phi: M \to Z_0$ is continuous, where $Z_0$ is equipped with the topology induced by $d_Z$, and satisfies the properties in Lemma~\ref{lem:dis}(1)--(2), by the same argument as in that lemma.

If we further assume that $(Z_0, d^Z_0)$ is locally compact, then it is in fact compact. Indeed, for any sequence $\{x_i\} \subset Z_0$, it follows from the compactness of $(Z, d_Z)$ that by passing to a subsequence, $x_i$ converges to $x \in Z_0$ under $d_Z$. Then the local compactness of $(Z_0, d^Z_0)$ at $x$ implies that $x_i$ converges to $x$ with respect to $d^Z_0$. Thus, by the same arguments as in Lemmas~\ref{lem:dis} and~\ref{lem:disa}, we conclude that $\Phi:(M, g_{-1}) \to (Z_0, d^Z_0)$ is continuous and surjective. In addition, the analogue of Proposition~\ref{prop:local-compactness-pgh} holds in this case.
In particular, if $(Z_0,d^Z_0)$ is locally
compact, then
\[
    (M,d_{g_t})
    \xrightarrow[t\nearrow0]{\GHconvtext}
    (Z_0,d^Z_0).
\]
\end{remark}

\subsection{Ricci shrinkers with maximal volume growth}
\begin{definition}
For a Ricci shrinker $(M^n,g,f)$, the asymptotic volume ratio is defined as
\begin{align*}
\operatorname{AVR}\left(M,g\right):=\lim _{r \to+\infty} \frac{\left|B_g(p,r)\right|_g}{\omega_n r^{n}},
\end{align*}
where $\omega_n$ is the volume of the unit ball in $\R^n$. The limit always exists and belongs to $[0, 1]$ (see \cite{CLY} and \cite{wang2025rigidity}). We say that a Ricci shrinker has maximal volume growth if its asymptotic volume ratio is positive.
\end{definition}

Recall that \(Z_0\) has the regular--singular decomposition in \eqref{eq:regular-singular-decomposition}.
The same argument as in \cite[Theorem 7.26]{FL1} implies that
\begin{equation}\label{eq:regular-part-nonempty}
    Z_0^{\reg}\neq\emptyset
    \quad\Longleftrightarrow\quad
    \operatorname{AVR}(M,g)>0.
\end{equation}
Moreover, by \cite[Proposition~7.30(iii)]{FL1}, the flow \(\boldsymbol{\psi}^s\) preserves the regular and singular sets:
\[
    x\in Z_0^{\reg}
    \quad
    (\text{resp. }x\in Z_0^{\sing})
    \quad\Longleftrightarrow\quad
    \boldsymbol{\psi}^s(x)\in Z_0^{\reg}
    \quad
    (\text{resp. }\boldsymbol{\psi}^s(x)\in Z_0^{\sing}).
\]

By the same argument as in Corollary~\ref{cor:smooth-convergence-generic-schwarz}, \((Z_0^{\reg},g_0^Z)\) has a metric cone structure. Here, \(g_0^Z\) denotes the smooth metric on \(Z_0^{\reg}\).
\begin{proposition}\label{prop:cone}
The function $F$ (see \eqref{eq:Fon0}) is smooth on $Z_0^{\reg}$ and satisfies
\[
    \nabla_{g_0^Z}^2F=\frac{g_0^Z}{2}.
\]
Set
\[
    \Sigma^{\mathrm{reg}}
    :=
    \{x\in Z_0^{\reg}\mid d^Z_0(\bar p,x)=1\},
\]
and let \(g_\Sigma\) be the metric induced by \(g_0^Z\). Then
\begin{equation}
(Z_0^{\reg}\setminus\{\bar p\},g_0^Z)
    =
(\Sigma^{\mathrm{reg}}\times(0,\infty),\,\d r^2+r^2g_\Sigma).
\end{equation}
\end{proposition}

Next, we give a characterization of $\operatorname{AVR}$.

\begin{proposition}\label{prop:AVRchar}
Let $(M,g)$ be a Ricci shrinker with maximal volume growth and bounded scalar curvature. Then
\begin{align*}
    \operatorname{AVR}(M, g)=\omega_n^{-1}\abs{\{z \in Z_0 \mid d^Z_0(z, \bar p) < 1\}}_{d^Z_0}.
\end{align*}
\end{proposition}

\begin{proof}
By \eqref{eq:distanceestimate}, it suffices to prove that
\begin{align}\label{eq:volratio1}
    \lim_{r \to +\infty} \frac{|\{x \in M \mid f(x) < r^2/4\}|_g}{r^n}=\abs{\{z \in Z_0 \mid d^Z_0(z, \bar p) < 1\}}_{d^Z_0}.
\end{align}
For \(t\in[-1,0)\), set $\Omega_t:=\{x\in M\mid F(x,t)<1/4\}$ and $\Omega_0:=\{z\in Z_0\mid d^Z_0(z,\bar p)<1\}$.
We claim that
\begin{align}\label{eq:volratio2}
\lim_{t \nearrow 0} |\Omega_t|_{g_t}=|\Omega_0|_{d^Z_0}.
\end{align}

We first show that the sets \(\Omega_t\) remain in a uniformly bounded spacetime region. More precisely, there exists $C_1=C_1(n, Y)>0$ such that for every $t\in[-1,0)$ and $x\in\Omega_t$,
\begin{align}\label{eq:volratio3}
d^*(\bar p, (x, t)) \le C_1.
\end{align}
Indeed, $(p,t)$ is an $H$-center of $\bar p$, where $H$ depends only on $n$ and $Y$; see \cite[Corollary~5.6]{liwang2024heat}. Hence \cite[Lemma~6.14]{FL1} implies that
\[
    d^*(\bar p,(p,t))\le A_1
\]
for all $t\in[-1,0)$, where $A_1=A_1(n, Y)$. On the other hand, by \eqref{eq:distanceestimate} and \cite[Lemma~6.4]{FL1}, for every $(x,t)\in \Omega_t \times \{t\}$ we have
\[
    d^*((p,t),(x,t))\le A_2,
\]
where $A_2=A_2(n, Y)$. Therefore
\[
d^*(\bar p, (x, t)) \le A_1+A_2.
\]
Similarly, after enlarging \(C_1\) if necessary, \cite[Lemma~6.4]{FL1} gives
\begin{align}\label{eq:volratio4}
d^*(\bar p, z) \le C_1
\end{align}
for every $z\in\Omega_0$.

Combining \eqref{eq:volratio3} and \eqref{eq:volratio4} with \cite[Theorem~1.12(b)]{FL1}, we obtain, for every sufficiently small $\delta>0$,
\begin{align}\label{eq:volratio5}
\sup_{ t \in [-1, 0)}\abs{\{r_{\Rm}<\delta\} \cap \Omega_t}_{g_t}+\abs{\{r_{\Rm}<\delta \}\cap \Omega_0}_{d^Z_0} \le C(n, Y) \delta^{1.99}.
\end{align}
Thus, \eqref{eq:volratio5} and the smooth convergence on the regular part from $g_t$ to $g^Z_0$ and from $F(\cdot,t)$ to $F(\cdot,0)$ as $t\nearrow0$ imply \eqref{eq:volratio2}.

By self-similarity, we can rewrite the claim as
\begin{align}\label{eq:volratio6}
\lim_{t \nearrow 0} |t|^{n/2}\abs{\{x \in M \mid f(x) < (4|t|)^{-1}\}}_g=|\Omega_0|_{d^Z_0},
\end{align}
which is precisely \eqref{eq:volratio1}. This completes the proof.
\end{proof}

Combining Propositions~\ref{prop:cone} and~\ref{prop:AVRchar} with \cite{wang2025rigidity}, we conclude that, unless $(M,g)$ is the Gaussian soliton, the $(n-1)$-dimensional Hausdorff measure of $\Sigma^{\reg}$ is at most $(1-\ep(n))$ times that of the standard sphere $S^{n-1}$.

\section{Comparison of the Fano and Metric Fibrations for K\"ahler--Ricci Shrinkers}\label{sec--6}

Let $(X,g,\omega,J,f)$ be a K\"ahler--Ricci shrinker, and adopt the conventions of Section~\ref{sec:schwarz-generic}. We have the Fano fibration $\pi:X\to Y$ from \cite{SunZhang}; under the bounded scalar curvature assumption, we also have the metric fibration $\Phi:X\to Z_0$ constructed in Lemma~\ref{lem:worldline}. In this section, we compare these two fibrations. In particular, we show that they coincide when the shrinker metric has subquadratic Riemannian curvature growth.

\subsection{K\"ahler--Ricci shrinkers with bounded scalar curvature}

\begin{lemma}\label{lem--factor through}
    Let $X$ be a K\"ahler--Ricci shrinker with bounded scalar curvature. Then there is a commutative diagram
    \begin{equation}
    \begin{tikzcd}
        & X \arrow[dl, "\Phi"'] \arrow[dr, "\pi"] \\
        Z_0 \arrow[rr, "\Phi_0"'] && Y,
    \end{tikzcd}
\end{equation}
Moreover, $\Phi_0$ is continuous with respect to $d_Z$.

\end{lemma}
\begin{proof}

To show that the map $\pi$ factors through $\Phi$, it suffices to prove that, whenever \(y_1\ne y_2\) in \(Y\) and \(x_i\in\pi^{-1}(y_i)\), one has
\begin{equation}\label{eq--to prove}
    \Phi(x_1)\neq \Phi(x_2).
\end{equation}
We want to use Proposition~\ref{prop:schwarz-lower-bound-general} to show that for all \(t\in [-1,0)\),
\[
   d_{\omega_t}(x_1,x_2)\ge \delta>0.
\]
For this to work, we need to show that there is a fixed compact subset $K$ of $X$ such that each $d_{\omega_t}$-geodesic connecting $x_1$ and $x_2$ is contained in $K$. Fix $A$ large such that
\begin{equation}
    A\geq d_{\omega_{-1}}(p, x_1)+d_{\omega_{-1}}(p, x_2)+100 n.
\end{equation}
Then we claim that the set $K=\{f\leq (2A+100n)^2+10\|\scal_{\omega_{-1}}\|_{L^\infty}\}$ satisfies the property we need. Since $F(x,t)$ is decreasing in $t$, \eqref{eq:distanceestimate} and the triangle inequality imply that for each $t\in [-1,0)$, we have
\begin{equation}
  d_{\omega_t}(p,x_1)\leq A, \quad d_{\omega_t}(x_1,x_2)\leq A.
\end{equation}
Moreover, the evolution equation for $F(x,t)$ and the bounded scalar curvature assumption imply that on the boundary $\partial K$, for every $t\in [-1,0)$
\begin{equation}
    F(x,t)\geq (2A+100n)^2.
\end{equation}
By \eqref{eq:distanceestimate} and the triangle inequality, any curve starting from $x_1$ and leaving $K$ has $d_{\omega_t}$-length at least $A+1$. Therefore, a $d_{\omega_t}$-geodesic connecting $x_1$ and $x_2$ must remain inside $K$.

Letting \(t\nearrow0\) and using Lemma~\ref{lem:dis}(1), we get
\[
d^Z_0\bigl(\Phi(x_1),\Phi(x_2)\bigr)
=\lim_{t\nearrow0}d_{\omega_t}(x_1,x_2)
    \ge \delta>0.
\]
Thus we have proved \eqref{eq--to prove}.

The map $\Phi_{0}$ is continuous with respect to the topology induced by $d_Z$. If \(z_i\to z\), then
\(\{z\}\cup\{z_i\}_{i\ge1}\) is compact in the metric space \((Z_0,d_Z)\), so
properness of $\Phi$ gives compactness of its preimage under $\Phi$. Then continuity of $\pi$ implies that $\Phi_0$ is continuous.
\end{proof}

\begin{theorem}\label{thm:regular-part-identification}
Let $X$ be a K\"ahler--Ricci shrinker with bounded scalar curvature. Then
\begin{equation} \label{eq:regularmap}
     \Phi:X^\circ\to Z_0^{\reg}
\end{equation}
is a biholomorphism, where \(Z_0^{\reg}\) is equipped with the
complex structure induced by the smooth convergence of the nearby time slices, and
\begin{equation}
   \Phi(E)=Z_0^{\sing}.
\end{equation}
\end{theorem}
\begin{proof}
Recall that we let $E$ denote the exceptional set of $\pi$ and $X^\circ=X\setminus E$.
If $(X,g)$ does not have maximal volume growth, then both sides of \eqref{eq:regularmap} are empty by
\eqref{eq:regular-part-nonempty} and
Proposition~\ref{prop:maximal-volume-growth-characterization}, so $E=X$ and $\Phi(E)=Z_0=Z_0^{\sing}$. In what follows, we may assume that $(X,g)$ has maximal
volume growth; equivalently, both sides are nonempty.

By Corollary~\ref{cor:smooth-convergence-generic-schwarz}, $\omega_t$ converges locally smoothly on $X\setminus E$ and hence yields an embedding $X\setminus E\hookrightarrow Z_0^{\reg}$. The same argument as in Proposition~\ref{prop:regular-part-identification} implies that $\Phi(X\setminus E)=Z_0^{\reg}$. Indeed, a point $z\in Z_0^{\reg}$ is the limit of $(x,t)$ with $r_{\Rm}(x,t)>c>0$ for all $t\in [-\ep,0)$. Hence, by the Ricci flow equation, there is a neighborhood $U_x$ of $x$ and a constant $C_x>0$ such that, for every $t\in [-1,0)$, we have on $U_x$
\begin{equation}\label{eq:two-sided-metric-bound}
C_x^{-1}\omega_{-1}\leq \omega_t=|t|(\psi^t)^*\omega_{-1}\leq C_x\omega_{-1}.
\end{equation}
If $x\in E$, then let $V_x$ denote a $k$-dimensional irreducible component of $\pi^{-1}(\pi(x))$ through $x$, where $k\geq1$. By \eqref{eq:two-sided-metric-bound}, we would have
\begin{equation}
    \vol(V_x,\omega_t^k)\geq C_x^{-k} \vol(V_x\cap U_x, \omega_{-1}^k).
\end{equation}
On the other hand, we always have
\begin{equation}
    \vol(V_x,\omega_t^k)=|t|^k\vol(V_x,\omega_{-1}^k)\to 0,
\end{equation}
a contradiction. Hence $x\in X\setminus E$.

Moreover, by definition, if $\Phi(x)\in Z_0^{\reg}$, then $r_{\Rm}(x,t)>c>0$ for all $t\in [-\ep,0)$ and hence $x\in X^\circ$ by the previous argument. Therefore, $\Phi(E)\subset Z_0^{\sing}$. Since $\Phi$ is surjective by Lemma~\ref{lem:dis}, we have $\Phi(E)=Z_0^{\sing}$.

Fix \(z\in Z_0^{\reg}\). By the regularity theorem for Ricci-flow limit spaces
\cite[Theorem~1.5]{FL1}, \(z\) has a regular spacetime flow box
\(\mathcal U\). If \(U_0=\mathcal U\cap Z_0\), then the flow of
\(\partial_{\mathfrak t}\) identifies \(\mathcal U\) with
\(U_0\times(-\ep,0]\).
The flow of $\partial_{\mathfrak t}$ gives a biholomorphism between different time slices $\mathcal{U}\cap (X\times \{t\})$ for $t<0$, and the complex structure on \(Z_0^{\reg}\) is the
smooth limit of these structures. Hence the above local diffeomorphism is
biholomorphic.
\end{proof}

Let $J_0$ denote the complex structure on $Z_0^{\reg}$. Then $\xi_0:=J_0\nabla_{g_0^Z}F$ is a holomorphic Killing vector field on $(Z_0^{\reg}, g_0^Z, J_0)$.

\begin{proposition}\label{prop-equivariant}
The vector field $\xi_0$ is complete and generates a one-parameter group
$\{\widehat h_s\}_{s\in\mathbb R}$ of holomorphic isometries of
$(Z_0^{\reg},g_0^Z,J_0)$. Moreover, if
$\{h_s\}_{s\in\mathbb R}$ denotes the flow generated by
$\xi=J\nabla f$ on $X$, then
\begin{equation}\label{eq:Phi-equivariant}
    \Phi\circ h_s=\widehat h_s\circ\Phi
    \quad\text{on }\quad X^\circ,
\end{equation}
and
\begin{equation}\label{eq:Phi-pushforward-xi}
\d\Phi_x(\xi_x)=\xi_{0,\Phi(x)}
\qquad\text{for every }x\in X^\circ.
\end{equation}
\end{proposition}

\begin{proof}
Since the Fano fibration $\pi:X\to Y$ is $\mathbb T$-equivariant, both
$X^\circ$ and its complement $E$ are invariant under $h_s$. We may
therefore define
\[
    \widehat h_s
    :=
    \Phi\circ h_s\circ\Phi^{-1}
    \colon Z_0^{\reg}\longrightarrow Z_0^{\reg},
\]
where we use
Theorem~\ref{thm:regular-part-identification}.

The vector fields $\xi$ and $\nabla f$ commute. Consequently, $h_s$
commutes with the self-similar diffeomorphisms defining $g_t$, and for
every $t\in[-1,0)$ we have
\[
    h_s^*g_t=g_t,
    \qquad
    h_s^*J=J,
    \qquad
    \xi=J\nabla_{g_t}F(\cdot,t).
\]
Passing to the limit as $t\nearrow0$ on $X^\circ$, using the smooth
convergence of $g_t$, $J$, and $F(\cdot,t)$ and the biholomorphic
identification by $\Phi$, gives
\[
    \widehat h_s^*g_0^Z=g_0^Z,
    \qquad
    \widehat h_s^*J_0=J_0,
\]
and
\[
    \Phi_*\xi
    =J_0\nabla_{g_0^Z}F
    =\xi_0.
\]
Thus $\{\widehat h_s\}_{s\in\mathbb R}$ is precisely the flow generated
by $\xi_0$. In particular, $\xi_0$ is complete, every
$\widehat h_s$ is a holomorphic isometry, and
\eqref{eq:Phi-equivariant} holds by definition.
\end{proof}

\begin{proposition}\label{prop:extended-torus-action}
Let $X$ be a K\"ahler--Ricci shrinker with bounded scalar curvature and
maximal volume growth. Suppose that $(Z_0,d^Z_0)$ is locally compact.
Then $Z_0^{\reg}$ is dense in $(Z_0,d^Z_0)$, and the following
statements hold.
\begin{enumerate}[label=\textnormal{(\roman*)}]
    \item The flow generated by $\xi_0$ extends to an
    isometric $\mathbb R$-action on $(Z_0,d^Z_0)$. The closure of its
    image in $\operatorname{Isom}(Z_0,d^Z_0)$ coincides with the compact torus $\mathbb T$ generated by $\xi$ and the map $\Phi$ is $\mathbb T$-equivariant.

    \item The link
    \[
        \Sigma
        :=
        \left\{
            z\in Z_0:d^Z_0(z,\bar p)=1
        \right\}
    \]
    is connected.
\end{enumerate}
\end{proposition}

\begin{proof}
By Lemma~\ref{lem:disa}, the map
\[
    \Phi:(X,g_{-1})\longrightarrow(Z_0,d^Z_0)
\]
is continuous and surjective. Since $E$ is a proper complex analytic
subset of $X$, the set $X^\circ=X\setminus E$ is dense in $X$.
Therefore, for every $z\in Z_0$, we may choose $x\in X$ with
$\Phi(x)=z$ and a sequence $x_j\in X^\circ$ converging to $x$.
Theorem~\ref{thm:regular-part-identification} then gives
\[
    \Phi(x_j)\in Z_0^{\reg},
    \qquad
    \Phi(x_j)\xrightarrow[]{d^Z_0}z.
\]
Hence $Z_0^{\reg}$ is dense in $Z_0$.

The proof of Lemma~\ref{lem:disa} also shows that every closed
$d^Z_0$-ball centered at $\bar p$ is compact. It follows that
$(Z_0,d^Z_0)$ is proper, and hence complete. Let
$\{\widehat h_s\}_{s\in\mathbb R}$ be the action on $Z_0^{\reg}$
constructed in Proposition~\ref{prop-equivariant}. Since every
$\widehat h_s$ is an isometry and $Z_0^{\reg}$ is dense, it extends
uniquely to an isometry of $(Z_0,d^Z_0)$. The group law passes to these
extensions. Thus, they define
an isometric $\mathbb R$-action on $Z_0$.

On $Z_0^{\reg}$ we have
\[
    \xi_0(F)
    =
    \left\langle
        J_0\nabla F,\nabla F
    \right\rangle_{g_0^Z}
    =0.
\]
Thus, the action preserves $F$ on $Z_0^{\reg}$, and hence on all of
$Z_0$ by continuity. Since
\[
    F(z)=\frac12d^Z_0(z,\bar p)^2,
\]
the vertex $\bar p$ is the unique zero of $F$, and the extended action
fixes $\bar p$. The factor \(1/2\), rather than \(1/4\), in the definition of \(F\) reflects the K\"ahler convention adopted here.

For a proper metric space, the stabilizer of a point in its isometry
group is compact. Therefore,
\[
    T
    :=
    \overline{
        \{\widehat h_s:s\in\mathbb R\}
    }
    \subset
    \operatorname{Isom}(Z_0,d^Z_0)_{\bar p}
\]
is compact, where the closure is taken in the compact-open topology.
The curvature radius is invariant under $\widehat h_s$ and continuous
on $Z_0$ after being extended by zero on $Z_0^{\sing}$. Consequently,
every element of $T$ preserves $Z_0^{\reg}$. Since the restriction of $d^Z_0$ to $Z_0^{\reg}$ locally coincides with the Riemannian distance induced by $g_0^Z$ (see \cite[Proposition 6.8]{FL1}), each element of $T$ preserves the Riemannian metric $g_0^Z$ \cite[Theorem 5.6.15]{petersen2006}. Restriction therefore
defines a continuous injective homomorphism
\[
    T
    \longrightarrow
    \operatorname{Isom}(Z_0^{\reg},g_0^Z).
\]
Since $T$ is compact, this map is a homeomorphism onto a compact, hence
closed, subgroup of
$\operatorname{Isom}(Z_0^{\reg},g_0^Z)$. By the Myers--Steenrod
theorem and the closed subgroup theorem, $T$ is a compact Lie group.
Moreover, $T$ is abelian and connected, because it is the closure of
the image of the connected abelian group $\mathbb R$. Consequently,
$T$ is a compact torus. It remains to show that the torus $T$ coincides with the torus $\mathbb T$ generated by $\xi$. Since $\Phi$ is continuous, there is a surjective group homomorphism from $\mathbb T$ to $T$. It remains to show that this map is injective. If $h\in\mathbb T$ acts as the identity on $Z_0$, then in particular it fixes $Z_0^{\reg}=X^\circ$ and hence it is the identity map on the whole $X$.

It remains to prove \textnormal{(ii)}. Set
\[
    \Sigma^{\reg}:=\Sigma\cap Z_0^{\reg}.
\]
We first show that $\Sigma^{\reg}$ is dense in $\Sigma$. Let
$z\in\Sigma$ and choose $z_j\in Z_0^{\reg}$ with
$z_j\to z$ in $d^Z_0$. Set
\[
    r_j:=d^Z_0(z_j,\bar p),
    \qquad
    s_j:=\log r_j,
    \qquad
    \widetilde z_j:=\boldsymbol{\psi}^{s_j}(z_j).
\]
Then $r_j\to1$ and $s_j\to0$. Since the homotheties
$\boldsymbol{\psi}^s$ preserve $Z_0^{\reg}$ and satisfy
\[
    d^Z_0\bigl(
        \boldsymbol{\psi}^s(w),\bar p
    \bigr)
    =
    e^{-s}d^Z_0(w,\bar p),
\]
we have
\[
    \widetilde z_j\in\Sigma^{\reg},
    \qquad
    \widetilde z_j\longrightarrow z.
\]
Thus
\[
    \overline{\Sigma^{\reg}}=\Sigma.
\]

Finally, $X^\circ$ is connected, because it is the complement of a
proper complex analytic subset of the connected complex manifold $X$.
By Theorem~\ref{thm:regular-part-identification},
$Z_0^{\reg}$ is connected. The cone decomposition in
Proposition~\ref{prop:cone} gives
\[
    Z_0^{\reg}\setminus\{\bar p\}
    \cong
    \Sigma^{\reg}\times(0,\infty),
\]
so $\Sigma^{\reg}$ is connected. Its closure $\Sigma$ is therefore
connected.
\end{proof}

\subsection{K\"ahler--Ricci shrinkers with subquadratic Riemannian curvature growth}
\begin{theorem}\label{thm:bounded-curvature-homeomorphism}
    Suppose $X$ has subquadratic Riemannian curvature growth, that is, $\lim_{x\to \infty}f^{-1}(x) |\Rm|(x)=0$. Then the
map \(\Phi_0:Z_0\to Y\) constructed in Lemma~\ref{lem--factor through} is a homeomorphism, where \(Y\) is endowed with its
analytic topology and \(Z_0\) is endowed with the topology induced by the
ambient spacetime distance \(d_Z\).
\end{theorem}

We use the subquadratic Riemannian curvature condition to show that
the compact fibers of the Fano fibration collapse with respect to $d_{\omega_t}$. It follows that \(\Phi\) is constant along
each fiber of \(\pi\), and hence that \(\Phi_0:Z_0\to Y\) is injective. The homeomorphism property then follows from the properness of $\Phi_0$.

\begin{lemma}
\label{lem:fano-fiber-diameter-collapse}
Let \((X,g,J,f)\) be a K\"ahler--Ricci shrinker with subquadratic Riemannian curvature growth. If \(P\subset X\) is a connected compact analytic curve, then its extrinsic diameter satisfies
\[
    \lim_{t\nearrow0}\diam_{d_{g_t}}(P)=0.
\]
\end{lemma}

\begin{proof}
Set $P_t:=\psi^t(P)$. Since $\psi^t$ is holomorphic and isotopic to
the identity,
\[
\operatorname{Area}_{\omega}(P_t)
=\operatorname{Area}_{\omega}(P)=:A.
\]
Moreover, compactness of $P$ and the evolution equation for the
potential imply that, for some $A_0<\infty$,
\[
P_t\subset\{f\le A_0|t|^{-1}\}
\]
for all $t$ sufficiently close to zero. Indeed, this can be seen from
\begin{align*}
    \frac{\d}{\d t} f(\psi^t(x))=\frac{|\na f|^2(\psi^t(x))}{2|t|} \le \frac{f(\psi^t(x))+n}{|t|}.
    \end{align*}
    
Define
\[
K(t):=
\sup_{\{f\le 2A_0|t|^{-1}\}}|\operatorname{Rm}_g|.
\]
The subquadratic-curvature assumption gives
\[
|t|K(t)\longrightarrow0.
\]
Set
\[
\rho_t:=c\min\{1,K(t)^{-1/2}\},
\]
where $c>0$ is sufficiently small. Since
$|\nabla\sqrt {f+n}|\le 1$, after decreasing $c$ if necessary, every ball
$B_g(y,2\rho_t)$ with $y\in P_t$ is contained in
$\{f\le2A_0|t|^{-1}\}$. Curvature control and the $\kappa$-noncollapsing of the
shrinker (see \cite[Theorem 22]{liwang2020heat1}) give bounded geometry at scale comparable to $\rho_t$.
The analytic curve $P_t$ defines a stationary integral $2$-varifold
with density at least one. The monotonicity formula (see \cite[Section~2.3]{ColdingDeLellis2003} and \cite[Section~3]{Scharrer2022}) therefore
gives
\[
\operatorname{Area}_{\omega}
\bigl(P_t\cap B_g(y,c\rho_t)\bigr)
\ge c_1\rho_t^2
\]
for every $y\in P_t$, with uniform constants.
Choose a maximal $2c\rho_t$-separated set
$\{y_1,\ldots,y_N\}\subset P_t$. The balls
$B_g(y_i,c\rho_t)$ are pairwise disjoint, so
\[
Nc_1\rho_t^2\le A,
\qquad
N\le C\rho_t^{-2}.
\]
The balls $B_g(y_i,2c\rho_t)$ cover the connected set $P_t$, and the
intersection graph of this relative cover is connected. Hence
\[
\operatorname{diam}_g(P_t)
\le C N\rho_t
\le C\rho_t^{-1}.
\]
Finally, self-similarity gives
\[
\begin{aligned}
\operatorname{diam}_{g_t}(P)
&=\sqrt{|t|}\operatorname{diam}_g(P_t)\\
&\le C\sqrt{|t|}\rho_t^{-1}\\
&\le C\left(\sqrt{|t|}+\sqrt{|t|K(t)}\right)
\longrightarrow0.
\end{aligned}
\]
\end{proof}

\begin{proof}[Proof of Theorem~\ref{thm:bounded-curvature-homeomorphism}]
Given $x_1,x_2\in X$ with $\pi(x_1)=\pi(x_2)$, by \cite[Corollary 1.4]{HM07}, we can connect $x_1$ and $x_2$ by a chain of rational curves and then Lemma~\ref{lem:fano-fiber-diameter-collapse}
 gives,
\[
    d^Z_0\bigl(\Phi(x_1),\Phi(x_2)\bigr)
    =
    \lim_{t\nearrow0}d_{g_t}(x_1,x_2)
    =
    0.
\]
Hence \(\Phi\) is constant on the fibers of \(\pi\), and therefore $\Phi_0:Z_0\to Y$ in Lemma~\ref{lem--factor through} is injective. Clearly $\Phi_0$ is surjective since $\pi$ is surjective. $\Phi_0$ is also proper since $\pi$ is proper and $\Phi$ is surjective.
Therefore, \(\Phi_0\) is a
homeomorphism.
\end{proof}

\begin{proof}[Proof of Theorem~\ref{bounded curvature-homome}]
    This follows from Lemma~\ref{lem--factor through}, Theorem~\ref{thm:regular-part-identification}, and Theorem~\ref{thm:bounded-curvature-homeomorphism}.
\end{proof}

\subsection{More conjectures on K\"ahler--Ricci shrinkers}
Several conjectures on K\"ahler--Ricci shrinkers were proposed in \cite[Section~6]{SunZhang}. Building on the results established in the preceding sections, we formulate further conjectures on K\"ahler--Ricci shrinkers and prove partial results toward them in the subsequent sections.

Let $X$ be a K\"ahler--Ricci shrinker with bounded scalar curvature. Then we have a surjective map $\Phi:X\to Z_0$ as in Lemma~\ref{lem:dis}. It is natural to conjecture the following.
\begin{conjecture}\label{conjecture1}
 $Z_0$ with the topology induced by $d^Z_0$ admits a polarized affine cone structure and the map $\Phi$ realizes the polarized Fano fibration structure on $X$.
\end{conjecture}

\begin{conjecture}
 $(Z_0,d^Z_0)$ is a metric cone and coincides with the metric completion of $(Z_0^{\reg},g_0^Z)$ if the K\"ahler--Ricci shrinker has maximal volume growth.
\end{conjecture}

\section{Volume Estimates for K\"ahler--Ricci Shrinkers}\label{sec--7}

It is conjectured in \cite[Conjecture 6.2]{SunZhang} that every K\"ahler--Ricci shrinker admits a unique tangent space at infinity and that this space is a metric cone. As a noncompact analogue of Theorem~\ref{thm--all}, we prove some results in this direction under additional assumptions.

Let $(X^n,g,\omega,\xi=J\nabla f)$ be a K\"ahler--Ricci shrinker. We will use the same notation as in \cite[Section~3]{SunZhang}. Let $\mathbb T$ denote the compact torus containing the holomorphic isometric action generated by the soliton vector field $J\nabla f$. Using the convexity of the momentum map, \cite[Lemma~3.3]{SunZhang} shows that one can perturb $J\nabla f$ to obtain an element of $\lie(\mathbb T)$ which generates an $S^1$-action and still admits a proper Hamiltonian potential, that is,
\begin{equation}
    \Lambda_{\mathbb Q}:=\{\eta\in\lie(\mathbb T)\mid u_\eta\text{ is proper and bounded below, and }\eta\text{ generates an }S^1\text{-action}\}\neq\emptyset.
\end{equation}
Moreover, we can choose $\eta$ sufficiently close to $J\nabla f$ such that
\begin{equation}\label{eq:hamiltonian-comparison}
   \frac12 f-C\leq u_{\eta}\leq 2f+C.
\end{equation}
By \cite[Proposition~3.5]{SunZhang}, we can fix $z_0\in \mathbb Q$ sufficiently large such that all critical values of $u_\eta$ are less than $z_0-1$. Then, for every regular value $z\geq z_0$ of $u_\eta$, the K\"ahler quotient
 \begin{equation}
     X_z:=u_{\eta}^{-1}(z)/S^1
 \end{equation}
is biholomorphic to $X_{z_0}$. Henceforth, we identify these K\"ahler quotients with the K\"ahler orbifold $X_{z_0}$. The K\"ahler reduction construction in \cite{SunZhang} yields a family of smooth orbifold K\"ahler forms $\omega_z$ parametrized by $z$, all defined on $X_{z_0}$. There exist an effective $\mathbb Q$-divisor $D$ and a $\mathbb Q$-line bundle $L$ on $X_{z_0}$ such that, for every regular value $z\geq z_0$ of $u_\eta$,
 \begin{equation}
     -(K_{X_{z_0}}+D)+zL \text{ is K\"ahler}.
 \end{equation}
In particular, $L$ is nef. By the Kawamata base-point-freeness theorem, $L$ is semiample; that is, there exists a fibration $\pi_{z_0}$ and a positive integer $\ell$ such that
\begin{equation}\label{eq-semiample-fibration}
  \pi_{z_0}:X_{z_0}\to X'_{z_0}\subset\mathbb P^N,
  \qquad
  \ell L=\pi_{z_0}^*\left(\left.\mathcal O_{\mathbb P^N}(1)\right|_{X'_{z_0}}\right).
\end{equation}
By \cite[Section~3]{SunZhang}, we obtain a polarized Fano fibration
 \begin{equation}\label{eq-fano-fibration}
     \pi:X\to Y,
 \end{equation}
where $Y$ is a normal affine cone. By construction,
 \begin{equation}
     \dim Y=\dim X_{z_0}'+1=\mathrm{nd}(L)+1,
 \end{equation}
where $\mathrm{nd}(L)$ denotes the numerical dimension of $L$, that is, the maximal integer $k$ such that $c_1(L)^k\neq 0$ in $H^{2k}(X_{z_0},\mathbb R)$.


\subsection{General volume growth estimates}
Using the K\"ahler reduction method, we obtain the following noncompact analogue of \eqref{eq:numerical-noncollapsing}, which gives a complex-analytic criterion for a K\"ahler--Ricci shrinker to have Euclidean volume growth.

\begin{proposition}\label{prop:maximal-volume-growth-characterization}
    Let $(X,g)$ be a K\"ahler--Ricci shrinker. The following are equivalent:
\begin{enumerate}[label=(\arabic*), font=\normalfont]
    \item $\operatorname{AVR}(X,g)>0$;
    \item $\dim Y=\dim X$;
    \item $L$ is a big line bundle on $X_{z_0}$.
\end{enumerate}
\end{proposition}
\begin{proof}
This is a direct consequence of the coarea formula. Since $u_\eta$ has only finitely many critical values and its critical set has zero top-dimensional Hausdorff measure, we have
\begin{equation}
\begin{aligned}
\left|\{\min u_\eta \le u_\eta \le r\}\right|_g
&=\int_{\min u_\eta}^{r}\int_{u_\eta^{-1}(s)}
\frac{1}{|\nabla u_\eta|}\,\d\mathscr H_g^{2n-1}\,\d s\\
&=T\int_{\min u_\eta}^{r}
\left|u_\eta^{-1}(s)/S^1\right|_{\omega_s}\,\d s\\
&=C(z_0)+T\int_{z_0}^{r}
\vol\bigl(-(K_{X_{z_0}}+D)+sL\bigr)\,\d s,
\end{aligned}
\end{equation}
where $T>0$ is the period of the $S^1$-action. Since $L$ is nef, $\vol(-(K_{X_{z_0}}+D)+sL)$ is a polynomial in $s$ whose degree is given by the numerical dimension of $L$. Since $u_\eta$ is comparable to $f$ by \eqref{eq:hamiltonian-comparison}, and $f$ is comparable to $d(p,\cdot)^2$ by \cite{CaoZhou}, the desired result follows.
\end{proof}
The same argument yields the following.

\begin{proposition}
    Let $X$ be a noncompact K\"ahler--Ricci shrinker. Set $k:=\dim Y=\mathrm{nd}(L)+1$. Then there exists a constant \(C>0\) satisfying
\[
C^{-1}\leq \frac{|B_g(p,r)|_g}{\omega_{2k} r^{2k}} \leq C
\]
for all \(r\geq 1\). Moreover, if $J\nabla f$ generates an $S^1$-action, then
\[
\lim_{r\to+\infty}\frac{|B_g(p,r)|_g}{\omega_{2k} r^{2k}}\in(0,\infty).
\]
\end{proposition}

\subsection{Volume noncollapsing estimates}\label{sec:volume-noncollapsing-assumptions}
In this subsection, we consider a K\"ahler--Ricci shrinker $(X, g, f)$ satisfying
\begin{enumerate}
    \item $\operatorname{AVR}(X,g)>0$;
    \item the scalar curvature satisfies
    \[
        \limsup_{x\to\infty}\frac{\scal_g(x)}{f(x)}<1;
    \]
    \item $J\nabla f$ generates an $S^1$-action.
\end{enumerate}
Using the K\"ahler reduction method again, we reduce the problem to a volume noncollapsing estimate for compact K\"ahler metrics and obtain the following estimate for certain classes of K\"ahler--Ricci shrinkers.
\begin{theorem}\label{thm-shrinker-volume-noncollapsing}
Under the above assumptions, for every $\ep>0$ and $\sigma \in (0,1)$, there exists a constant $C_\ep$, depending on $\ep$ and the K\"ahler--Ricci shrinker, such that
    \begin{equation}\label{eq:shrinker-volume-noncollapsing}
        \left|B_g(x,\sigma r)\right|_g\geq C_\ep \sigma^{2n+\ep} r^{2n}
    \end{equation}
for all $x \in X$ with $r=d(p, x) \ge 1$.
\end{theorem}

Let
\[
    \mathcal F=f-\log |\nabla u_{\eta}|^2.
\]
The function $\mathcal F$ is $S^1$-invariant and hence descends to a function on $X_{z_0}$. Under the identifications above, let $D_z$ denote the divisor on $X_z$ corresponding to $D$. By \cite[Lemma~3.4]{SunZhang}, we have a family of smooth orbifold K\"ahler forms on $X_{z_0}$ satisfying
\begin{equation}\label{eq:omega-z-equation}
    \Ric(\omega_z)=2\pi [D_z]+\omega_z-z\partial_z\omega_z-\ii\partial \pp \mathcal{F}.
\end{equation}
Let
\begin{equation}\label{eq:quotient-time-variable}
    t=-\tau=-\frac{1}{z}\in[t_0,0),
    \qquad
    t_0=-\tau_0=-\frac{1}{z_0},
\end{equation}
and set
\begin{equation}\label{eq:theta-definition}
    \theta_t=\frac{\omega_z}{z}.
\end{equation}
Then \eqref{eq:omega-z-equation} can be rewritten as a twisted K\"ahler--Ricci flow equation
\begin{equation}
    \partial_t \theta_t=-\Ric(\theta_t)+2\pi [D]-\ii\partial\pp \mathcal{F}.
\end{equation}
The cohomology class satisfies
\begin{equation}
\begin{aligned}
     [\theta_t]&=\frac{1}{z}(-(K_{X_{z_0}}+D)+zL)\\
     &=\frac{1}{z}(-(K_{X_{z_0}}+D)+z_0L+(z-z_0)L)\\
     &=\tau \tau_0^{-1}[\theta_{t_0}]+(1+\tau_0^{-1} t)c_1(L).
\end{aligned}
\end{equation}

Via the semiample fibration \eqref{eq-semiample-fibration}, we can choose a smooth semipositive form $\alpha\in c_1(L)$ on $X_{z_0}$. Then we can find a smooth orbifold volume form $\Omega$ such that on $X_{z_0}$,
\begin{equation}\label{eq:reference-volume-curvature}
    \ii\partial\pp \log (\Omega)=-\tau_0^{-1}\theta_{t_0}+\tau_0^{-1} \alpha-2\pi [D].
\end{equation}
Here, we emphasize that our orbifold structure on $X_{z_0}$ allows a complex codimension-one branching locus, which is given by the support of the divisor $D$. Thus, a smooth orbifold volume form $\Omega$ has conical singularities along $D$, which accounts for the divisorial term on the right-hand side of \eqref{eq:reference-volume-curvature}. We write
\begin{equation}
\theta_t=\tau\tau_0^{-1}\theta_{t_0}+(1+\tau_0^{-1}t)\alpha+\ii\partial\pp \varphi_t.
\end{equation}
The parabolic equation for $\theta_t$ can be written as the parabolic equation for the scalar function
\begin{equation}\label{eq:reduced-potential-flow}
\partial_t\varphi_t=\log\left(\frac{\theta_t^{n-1}}{\Omega}\right)-\mathcal{F}+c_t.
\end{equation}
Equivalently,
\begin{equation}
    \theta_t^{n-1}=e^{\partial_t\varphi_t+\mathcal{F}-c_t}\Omega.
\end{equation}
where $c_t$ is a constant.

The preceding arguments apply to arbitrary K\"ahler--Ricci shrinkers. From now on, assume that $J\nabla f$ generates an $S^1$-action.
Under this assumption, we do not need to perturb the vector field; therefore,
\[
    \eta=J\nabla f,
    \qquad
    u_\eta=f,
\]
with the normalization $\scal+|\nabla^{1,0}f|^2=f+n$, where $\scal$ is the scalar curvature of $X$.

By $S^1$-invariance, the relevant quantities descend to functions on the quotient \(f^{-1}(z)/S^1\), which is identified with \(X_{z_0}\). On this space we have
\begin{equation}\label{eq:gradient-f-level}
  f\equiv z,
  \qquad
  |\nabla^{1,0} f|^2=z+n-\scal.
\end{equation}
In \eqref{eq:reduced-potential-flow}, choose
\[
    c_t=z-\log z=\tau^{-1}+\log\tau.
\]
Then
\begin{equation}\label{eq:twisted-potential-flow}
\partial_t\varphi_t=\log\left(\frac{\theta_t^{n-1}}{\Omega}\right)+\log(1+n\tau-\scal\tau).
\end{equation}
Equivalently,
\begin{equation}\label{eq:twisted-measure}
    \theta_t^{n-1}=e^{\partial_t\varphi_t}(1+n\tau-\scal\tau)^{-1} \Omega.
\end{equation}


We now compute \(\partial_t^2\varphi_t\) under the normalization given by \eqref{eq:twisted-potential-flow}. Locally on an
orbifold uniformizing chart of the quotient, write
\[
    \omega_z^{n-1}
    =V(\sqrt{-1})^{(n-1)^2}\Omega_R\wedge\overline{\Omega_R},
    \qquad
    \Omega=e^\psi(\sqrt{-1})^{(n-1)^2}
    \Omega_R\wedge\overline{\Omega_R},
\]
where the second identity expresses the fixed reference volume form
\(\Omega\) in the fixed chart on \(X_{z_0}\). Thus, the local density
\(\psi\) depends only on the spatial variable and is independent of \(t\);
only \(V\), the density of \(\omega_z^{n-1}\), varies with \(z\) (equivalently
with \(t\)). Since \(\theta_t=z^{-1}\omega_z\),
\eqref{eq:gradient-f-level} and \eqref{eq:twisted-potential-flow} imply
\[
    \partial_t\varphi_t
    =\log(z^{-n}|\nabla^{1,0} f|^2V)-\psi.
\]
We now use the computation in \cite[proof of Lemma~3.4, equations
(3.11)--(3.13)]{SunZhang}. In the notation of that proof, the local functions
\(h\) and \(F_{\mathrm{SZ}}\) satisfy
\[
    h^{-1}=|\nabla f|^2=2|\nabla^{1,0} f|^2,
    \qquad
    F_{\mathrm{SZ}}=f-\log(h^{-1}V).
\]
Moreover, \cite[(3.13)]{SunZhang} gives
\[
    h^{-1}\partial_zF_{\mathrm{SZ}}=2z.
\]
Since \(f=z\) on the quotient, this is equivalent to
\[
\partial_z\log(|\nabla^{1,0} f|^2V)=1-z|\nabla^{1,0} f|^{-2}.
\]
Using \(\partial_t=z^2\partial_z\), we obtain
\begin{equation}\label{eq:potential-second-derivative}
    \begin{aligned}
    \partial_t^2\varphi_t
    &=z^2\partial_z\log(z^{-n}|\nabla^{1,0}f|^2V)  \\
    &=z^2\left(-\frac{n}{z}+1-z|\nabla^{1,0} f|^{-2}\right) \\
    &=-\tau^{-3}|\nabla^{1,0} f|^{-2}+\tau^{-2}-n\tau^{-1}.
\end{aligned}
\end{equation}

\begin{lemma}\label{lem:potential-bound}
    Suppose \(J\nabla f\) generates an \(S^1\)-action and
    \(
        \limsup_{x\to\infty}\scal_g(x)/f(x)<1.
    \)
    For the normalization fixed in
    \eqref{eq:twisted-potential-flow}, there exists a constant \(C\) such that
    \[
        \partial_t\varphi_t\leq C
    \]
    for all \(t\in [t_0,0)\).
\end{lemma}
\begin{proof}
By the computation above,
\[
    \partial_t^2\varphi_t
    =
    -\tau^{-3}|\nabla^{1,0} f|^{-2}+\tau^{-2}-n\tau^{-1}.
\]
Using \eqref{eq:gradient-f-level}, we have
\[
    |\nabla^{1,0} f|^{-2}=\frac{\tau}{1+n\tau-\tau\scal},
\]
and hence
\[
    \begin{aligned}
        \partial_t^2\varphi_t
    &=
    \tau^{-2}-n\tau^{-1}
    -\tau^{-2}(1+n\tau-\tau\scal)^{-1}\\
    &=
    \frac{\scal(n\tau-1)-n^2\tau}
    {\tau(1+n\tau-\tau\scal)}
    \le 0,
    \end{aligned}
\]
provided that $n\tau\le1$, since the scalar curvature is nonnegative on a Ricci shrinker \cite{Chen2009Strong}.
Hence, we obtain a uniform upper bound for $\partial_t\varphi_t$.
\end{proof}

\begin{proof}[Proof of Theorem~\ref{thm-shrinker-volume-noncollapsing}]
Using the scalar-curvature assumption, after decreasing $\tau_0$ if necessary, there exists $\ep_0>0$ such that \(1+n\tau-\tau\scal\ge \ep_0\). The measure identity~\eqref{eq:twisted-measure} and Lemma~\ref{lem:potential-bound} then give, for all $t\in[t_0,0)$,
\[
    \frac{\theta_t^{n-1}}{\Omega}
    \le C(\ep_0).
\]
Since $\Omega$ is smooth in the orbifold sense, we obtain the uniform $L^p$-integrability of the
volume form $\theta_t^{n-1}$ with respect to a fixed smooth volume form, for some $p>1$. Moreover, since the shrinker has maximal volume growth, Proposition~\ref{prop:maximal-volume-growth-characterization} implies that the global volume $\int_{X_{z_0}}\theta_t^{n-1}$ has a uniform lower bound.
Using \cite[Theorem~8.1 and Proposition~9.1]{GPSS2}, we obtain
a uniform volume noncollapsing estimate for the quotient metrics $\theta_t$. Namely, there exists
a constant $C_\ep>0$ such that, for all relevant quotient metrics
$\theta_t$, all $x\in X_{z_0}$, and all $\sigma\in (0,1)$,
\begin{equation}\label{eq:quotient-volume-noncollapsing}
    \left|B_{\theta_t}(x,\sigma)\right|_{\theta_t}
    \geq C_\ep \sigma^{2n-2+\ep}.
\end{equation}

We now derive the corresponding volume noncollapsing estimate for the shrinker
metric $\omega$. Let $x\in X$ and set $r=d(p,x)\ge1$.
For each
\[
    z\in \big((1-\sigma/2)f(x),(1+\sigma/2)f(x)\big),
\]
let $x_z$ be the point on the integral curve of $\nabla f$ through $x$ satisfying
\(
    f(x_z)=z.
\)
Then
\[
    d(x,x_z)\asymp \left|\sqrt{f(x)}-\sqrt{z}\right|.
\]
Since we only need to consider the case where $r$ is sufficiently large, the
asymptotics of the soliton potential \cite{CaoZhou} imply that, on the region
under consideration,
\[
    f(x)\asymp r^2, \qquad |\nabla f|\asymp r.
\]
In particular, after decreasing the numerical constants if necessary, we have
\[
    d(x,x_z)\leq \frac{\sigma r}{10}
\]
for all
\[
    z\in \big((1-\sigma/100)f(x),(1+\sigma/100)f(x)\big).
\]

Let $\hat g_z$ denote the restriction of $g$ to the level set $f^{-1}(z)$, and
let $B_{\hat g_z}(x_z,\rho)$ be the intrinsic ball in $f^{-1}(z)$ with respect to
$\hat g_z$. The triangle inequality then gives
\[
    B_{\hat g_z}\left(x_z,\frac{\sigma r}{100}\right)
    \subset B_g(x,\sigma r).
\]
Since $|\nabla f|\asymp r$ on this region, the coarea formula gives
\begin{equation}\label{eq-coarea1}
\begin{aligned}
    \left|B_g(x,\sigma r)\right|_g
    &\geq
    \int_{(1-\sigma/100)f(x)}^{(1+\sigma/100)f(x)}
    \int_{B_{\hat g_z}(x_z,\sigma r/100)}
    \frac{1}{|\nabla f|}\,\d V_{\hat g_z}\,\d z  \\
    &\geq
    \frac{1}{C r}
    \int_{(1-\sigma/100)f(x)}^{(1+\sigma/100)f(x)}
    \left|B_{\hat g_z}\left(x_z,\frac{\sigma r}{100}\right)\right|_{\hat g_z}\,\d z.
\end{aligned}
\end{equation}

We next estimate the level-set volume from below using the $S^1$-fibration
structure. Let
\[
    \Phi_z:f^{-1}(z)\to X_z:=f^{-1}(z)/S^1
\]
be the quotient map, and let $g_z$ be the quotient metric. Denote by
$\bar x_z=\Phi_z(x_z)$ the projection of $x_z$.
Since $\Phi_z$ is a Riemannian submersion away from a measure-zero singular set,
we can horizontally lift minimizing geodesics in the quotient. Thus, for every
\[
    \bar y\in B_{g_z}\left(\bar x_z,\frac{\sigma r}{200}\right),
\]
there is a point $y\in f^{-1}(z)$ above $\bar y$ such that
\[
    d_{\hat g_z}(x_z,y)\leq \frac{\sigma r}{200}.
\]
Moreover, since the $S^1$-orbits have speed
\(
    |J\nabla f|=|\nabla f|\asymp r,
\)
an orbit segment of length comparable to $\sigma r$ through such a point remains
inside
\[
    B_{\hat g_z}\left(x_z,\frac{\sigma r}{100}\right).
\]
Applying the coarea formula to the Riemannian submersion
$\Phi_z:f^{-1}(z)\to X_z$, we obtain
\begin{equation}\label{eq-coarea2}
\begin{aligned}
    \left|B_{\hat g_z}\left(x_z,\frac{\sigma r}{100}\right)\right|_{\hat g_z}
    &\geq
    c\sigma r\,
    \left|B_{g_z}\left(\bar x_z,\frac{\sigma r}{200}\right)\right|_{g_z}\\
    &\geq
    c\sigma r^{2n-1}
    \left|B_{g_z/z}\left(\bar x_z,\frac{\sigma}{200}\right)\right|_{g_z/z}.
\end{aligned}
\end{equation}
Therefore, the volume noncollapsing estimate for the shrinker metric follows from \eqref{eq-coarea1}, \eqref{eq-coarea2}, and \eqref{eq:quotient-volume-noncollapsing}.
\end{proof}

\subsection{Uniqueness of the tangent space at infinity}

As in the preceding section, we consider the associated K\"ahler--Ricci flow $(X,(g_t)_{t\in(-\infty,0)})$ for a shrinker satisfying the three conditions in Section~\ref{sec:volume-noncollapsing-assumptions}. We first prove the following volume estimate.

\begin{lemma}\label{lem:noncollapsed1}
For every \(t\in[-1,0)\), \(\ep>0\), and \(\sigma\in(0,1)\), there exists a constant \(c_\ep>0\), depending on \((X,g,f)\) and \(\ep\), such that
\[
    \left|B_{g_t}(x,\sigma)\right|_{g_t}\ge c_\ep \ell^{-\ep}\sigma^{2n+\ep}>0
\]
for every \(x\in X\) with
\[
    d_{g_t}(p,x)=\ell\ge 2\sigma.
\]
\end{lemma}

\begin{proof}
By the self-similarity of \((X,g_t)\), we have
\[
    \left|B_{g_t}(x,\sigma)\right|_{g_t}
    =
    |t|^n\left|B_{g_{-1}}(x',\sigma |t|^{-1/2})\right|_{g_{-1}},
\]
where \(x'=\psi^t(x)\). Since
\[
    d_{g_{-1}}(p,x')
    =
    |t|^{-1/2}d_{g_t}(p,x)
    =
    \ell |t|^{-1/2},
\]
the conclusion follows from Theorem~\ref{thm-shrinker-volume-noncollapsing}.
\end{proof}

We next prove that sequential pointed Gromov--Hausdorff limits of \((X,d_{g_t},p)\) exist as \(t\nearrow 0\).

\begin{lemma}\label{lem:sequence}
For every sequence \(t_j\nearrow 0\), after passing to a subsequence, we have
\[
   (X,d_{g_{t_j}},p)
   \xrightarrow[t_j\nearrow 0]{\pGHconvtext}
   (X_{\GH},d_{\GH},p_{\GH}),
\]
and $(X_{\GH},d_{\GH})$ is a locally compact metric space.
\end{lemma}

\begin{proof}
It suffices to show that, for every \(L>1\) and \(\sigma\in(0,1)\), the maximal number of mutually disjoint balls of radius \(\sigma\), with centers contained in \(B_{g_{t_j}}(p,L)\), is bounded by a constant independent of \(j\).

Fix \(j\), and suppose that
\[
    \{B_{g_{t_j}}(y_i,\sigma)\}_{1\le i\le N}
\]
are mutually disjoint balls with \(y_i\in B_{g_{t_j}}(p,L)\). At most one of the points \(y_i\) can lie in \(B_{g_{t_j}}(p,\sigma)\); denote this point, if it exists, by \(y_N\). For the remaining points, we consider the mutually disjoint balls
\[
    \{B_{g_{t_j}}(y_i,\sigma/2)\}_{1\le i\le N-1}.
\]
By \eqref{eq:volume-upper-bound}, Lemma~\ref{lem:noncollapsed1}, and a standard ball-packing argument, \(N-1\) is bounded by a constant independent of \(j\).

Moreover, this ball-packing estimate
gives uniform total boundedness on every bounded ball for the family
\((X,d_{g_t},p)\), \(t\nearrow0\). Therefore, every closed bounded ball in the
pointed Gromov--Hausdorff limit obtained is
compact.
\end{proof}

To use the spacetime completion for a Ricci shrinker discussed in Section~\ref{sec:generalshrinker}, we now assume in addition that the shrinker has bounded scalar curvature. Thus, we consider a K\"ahler--Ricci shrinker $(X, g, f)$ satisfying
\begin{enumerate}
    \item $\operatorname{AVR}(X,g)>0$;
    \item the scalar curvature is bounded;
    \item $J\nabla f$ generates an $S^1$-action.
\end{enumerate}
We set
\[
    (Z,d_Z,\mathfrak t, \bar p)
\]
to be the metric completion of \(X\times(-\infty,0)\) with respect to \(d^*\) defined in \eqref{eq:spacetime-distance-shrinker}. As in Proposition~\ref{prop:embh}, we obtain an isometric embedding $h$ from $(Z_0,d^Z_0)$ into $(X_{\GH},d_{\GH})$. Because $(Z_0,d^Z_0)$ is complete, its image $h(Z_0)$ is closed in $X_{\GH}$; in particular, $(Z_0,d^Z_0)$ is locally compact. Proposition~\ref{prop:local-compactness-pgh} therefore yields the uniqueness of the pointed Gromov--Hausdorff limit in the following theorem. By self-similarity, this also implies Theorem~\ref{thm:tangentspace}. Moreover, Lemma~\ref{lem:disa} gives a continuous map from $X$ to $(Z_0,d^Z_0)$.

\begin{theorem}\label{thm:tangent1}
The pointed Gromov--Hausdorff limit of
\(
    (X,d_{g_t},p)
\)
as \(t\nearrow 0\) exists and is isometric to \((Z_0,d^Z_0,\bar p)\), which is a locally compact metric space. Moreover, we have a continuous proper map
\begin{equation}
    \Phi:X\to(Z_0,d^Z_0).
\end{equation}
\end{theorem}


We also have the following analogue of Theorem~\ref{thm:exclude}.
\begin{proposition}\label{prop:excludecylinder}
For $(Z,d_Z,\t)$, no tangent flow is isometric to $\overline{\mathcal C}^{2n-2}$.
\end{proposition}
\begin{proof}
Suppose, to the contrary, that a tangent flow of \((Z,d_Z,\t)\) at some point
\(z_\ast\in Z\) is isometric to \(\overline{\mathcal C}^{2n-2}\). Since the
flow is smooth on \(X\times (-\infty,0)\), we must have \(z_\ast\in Z_0\).

As in the proof of Theorem~\ref{thm:exclude}, the smooth convergence to the
tangent flow gives a limiting parallel complex structure on
\(S^2\times \mathbb R^{2n-2}\). This complex structure preserves the splitting,
and hence the cylindrical model is biholomorphic to
\(
	\mathbb P^1\times \mathbb C^{n-1}.
\)
Therefore, by Lemma~\ref{lem--rigidity}, after choosing a sufficiently small
scale in the tangent-flow convergence and pulling back to the original flow,
there exists an open set \(U\subset X\) and a biholomorphism
\[
	U\simeq \mathbb P^1\times B^{n-1}.
\]
Here we use that the associated K\"ahler--Ricci flow has the fixed complex
structure \(J\). For \(a\in B^{n-1}\), denote the corresponding holomorphic
sphere by
\[
	\mathbb P^1(a):=\mathbb P^1\times\{a\}\subset U.
\]
Let
\[
	\pi:X\to Y
\]
be the Fano fibration, and let \(E\subset X\) be the exceptional set, so that
\[
	\pi|_{X\setminus E}:X\setminus E\longrightarrow Y\setminus \pi(E)
\]
is biholomorphic. Since
\(Y\) is affine, the restriction of $\pi$ to each $\mathbb P^1(a)$ is constant, and hence $U\subset E$. This is impossible because \(E\), being a proper analytic subset of \(X\), has empty interior. Thus, no tangent flow of \((Z,d_Z,\t)\) is isometric to \(\overline{\mathcal C}^{2n-2}\).
\end{proof}

As in Theorem~\ref{thm:exclude1}, it follows from Proposition~\ref{prop:excludecylinder} that $\MS=\MS^{2n-4}$. Combining this with the proof of Theorem~\ref{thm:codimension-four} and using Theorem~\ref{thm-shrinker-volume-noncollapsing}, we obtain the following result.
\begin{theorem}
The Hausdorff dimension of the singular set \(Z_0^{\sing}\) of \(Z_0\) satisfies
\begin{equation}\label{eq:dimension-shrinker-singular}
    \dim_{\mathscr H}(Z_0^{\sing})\le 2n-4.
\end{equation}
\end{theorem}

By Theorem~\ref{thm:tangent1}, $(Z_0,d^Z_0)$ is locally compact, so Lemma~\ref{lem:disa} applies. Combining Theorem~\ref{thm:regular-part-identification} and Theorem~\ref{thm:bounded-curvature-homeomorphism}, we obtain the following, which verifies Conjecture~\ref{conjecture1} under extra assumptions.

\begin{corollary}
    Let $X$ be a K\"ahler--Ricci shrinker with maximal volume growth and bounded Riemannian curvature such that $J\nabla f$ generates an $S^1$-action. Then $\Phi_0:(Z_0,d^Z_0) \to Y$ is a homeomorphism and induces a biholomorphism from $Z_0^{\reg}$ to $\pi(X^\circ)$.
\end{corollary}

\section{End Structures and Four-Dimensional Ricci Shrinkers}\label{sec--8}
We use the same notation and conventions as in Section~\ref{sec:generalshrinker}. Here $n$ stands for the real dimension of a Ricci shrinker. Throughout this section, we assume that the scalar curvature is bounded above by $C_0$ and that the entropy is bounded below by $-Y$.

\subsection{Splitting of tangent flows and decomposition by ends}
We prove the following splitting result. The definition of $k$-splitting is given in \cite[Definition 8.1]{FL1}.

\begin{proposition}\label{prop:1splitting}
    For every $z\in Z_0\setminus \{\bar p\}$, every tangent flow at $z$ is $1$-splitting. In particular, when $n=4$, the tangent flow at $z$ is unique.
\end{proposition}

\begin{proof}
Let $(Z',d_{Z'},z',\t')$ be a tangent flow at $z$. Thus $(Z',d_{Z'},z',\t')$ arises as a pointed Gromov--Hausdorff limit of
\[
    \bigl(Z,r_i^{-1}d_Z,z,r_i^{-2}(\t-\t(z))\bigr)
\]
for some sequence $r_i\searrow0$. Moreover, the regular part of $Z'$ is given by a Ricci flow spacetime
\[
    (\RR',\t',\partial_{\t'},g^{Z'}_t),
\]
on which the convergence is smooth.

Choose an $H_n$-center of $z$:
\[
    (z_i,-r_i^2)\in M\times\{-r_i^2\}.
\]
Define
\begin{align*}
s_i&=F(z_i,-r_i^2),\\
g_i(t)&=r_i^{-2}g_{r_i^2t}, \\
f_i(\cdot,t)&=s_i^{-\frac12}r_i^{-1}
    \bigl(F(\cdot,r_i^2t)-s_i\bigr).
\end{align*}
The following identities are straightforward to verify:
\begin{align}
    \partial_t f_i
    &=
    s_i^{-\frac12} r_i t\, \scal_{g_i(t)},  \label{eq:iden01}\\
    \nabla^2_{g_i(t)} f_i(\cdot,t)
    &=
    \frac12 s_i^{-\frac12}r_i g_i(t)
    +
    s_i^{-\frac12}r_i t\,\Ric_{g_i(t)},\label{eq:iden02}\\
    |\nabla_{g_i(t)} f_i(\cdot,t)|_{g_i(t)}^2
    &=
    1+s_i^{-\frac12}r_i f_i(\cdot,t)
    -s_i^{-1}r_i^2t^2\scal_{g_i(t)}. \label{eq:iden03}
\end{align}
By \eqref{eq:distanceestimate} and the definition of $d^Z_0$, we have
\begin{align*}
    \lim_{i\to\infty}s_i
    =
    \frac14\bigl(d^Z_0(\bar p,z)\bigr)^2
    >0.
\end{align*}
Passing to the limit in \eqref{eq:iden01}--\eqref{eq:iden03}, and using the smooth convergence on the regular part, we conclude that $f_i$ converges smoothly on $\RR'$ to a function $f_\infty$ satisfying
\begin{align*}
    \partial_{\t'}f_\infty=0,
    \qquad
    \nabla^2_{g^{Z'}} f_\infty=0,
    \qquad
    |\nabla_{g^{Z'}} f_\infty|_{g^{Z'}}^2=1.
\end{align*}
Therefore, $\nabla_{g^{Z'}} f_\infty$ gives a parallel unit vector field on $\RR'$, and hence induces a splitting direction. Thus, every tangent flow at $z$ is $1$-splitting.

Now suppose $n=4$. By the classification of three-dimensional Ricci shrinkers \cite{Hamilton1995Formation,Naber2010,NiWallach2008}, the time $-1$ slice
\[
    \bigl(\RR'_{-1},g^{Z'}_{-1}\bigr)
\]
is isometric to one of the following standard models:
\begin{align*}
    \R^4,\quad
    S^3/\Gamma\times \R,\quad
    S^2\times \R^2,\quad
    \mathbb{RP}^2\times \R^2,\quad
    \bigl(S^2\times_{\mathbb Z_2}\R\bigr)\times \R.
\end{align*}
Hence, any such tangent flow is unique. This proves the proposition.
\end{proof}


Next, we consider a special case in which an end of a Ricci shrinker is asymptotic to a cylinder. We begin by recalling the relevant definition.

\begin{definition}[$\ep$-closeness]\label{def:close}
Let $(M_i,g_i,x_i)$, $i=1,2$, be two pointed Riemannian manifolds. We say that $(M_1,g_1,x_1)$ is \textbf{$\ep$-close} to $(M_2,g_2,x_2)$ if there exists a map
\[
    \varphi:B_{g_2}(x_2,\ep^{-1})\to M_1
\]
such that $\varphi(x_2)=x_1$ and
\begin{equation*}
    \bigl\|r^{-2}\varphi^*g_1-g_2\bigr\|_
    {C^{[\ep^{-1}]}\left(B_{g_2}(x_2,\ep^{-1})\right)}
    <\ep,
\end{equation*}
where
\[
    r^{-2}=\scal_{g_1}(x_1),
\]
and the $C^{[\ep^{-1}]}$-norm is taken with respect to $g_2$. In particular, we implicitly assume that $\scal_{g_1}(x_1)>0$. Moreover, we say that a vector field $V_1$ on $M_1$ is $\ep$-close to a vector field $V_2$ on $M_2$ if 
\begin{equation*}
    \bigl\|r(\varphi^{-1})_*(V_1)- V_2\bigr\|_
    {C^{[\ep^{-1}]}\left(B_{g_2}(x_2,\ep^{-1})\right)}
    <\ep.
\end{equation*}
\end{definition}

Choose \(L\ge 2C_0\), to be enlarged later if necessary. By \eqref{eq:normal1}, the level set \(\{f=s\}\) is regular for every \(s\ge L\). In particular, \(\{f=L\}\) has finitely many components
\[
    \Sigma_1,\ldots,\Sigma_k.
\]
Each component \(\Sigma_i\) bounds an end \(E_i\) of \(M\). For each end \(E_i\), we define the corresponding subset \(Z_0(E_i)\subset Z_0\) to consist of all points \(z\in Z_0\setminus\{\bar p\}\) which arise as limits, with respect to \(d^*\), of points \((x,t)\in E_i^*\), where
\begin{align}
    E_i^*
    :=
    \left\{(x,t)\in M\times[-1,0)\;:\; \psi^t(x)\in E_i\right\}.
\end{align}

\begin{lemma}\label{lem:unionend}
After increasing \(L\) if necessary, we have
\begin{align*}
    Z_0\setminus\{\bar p\}
    =
    \bigsqcup_{i=1}^k Z_0(E_i).
\end{align*}
\end{lemma}

\begin{proof}
We first prove that
\[
    Z_0\setminus\{\bar p\}
    \subset
    \bigcup_{i=1}^k Z_0(E_i).
\]
Let \(z\in Z_0\) with \(z\ne \bar p\), and choose an \(H_n\)-center
\(
    (x_i,-2^{-i})
\)
of \(z\). By Lemma~\ref{lem:continuityF} and \eqref{eq:distanceestimate}, we have
\begin{align*}
    \lim_{i\to\infty}
    2\sqrt{F(x_i,-2^{-i})}
    =
    d^Z_0(\bar p,z)
    >
    0.
\end{align*}
Writing
\(
    x_i'=\psi^{-2^{-i}}(x_i),
\)
we obtain
\(\lim_{i\to\infty} f(x_i')=+\infty\).
Thus, for all sufficiently large \(i\),
\[
    x_i'\in \bigcup_{j=1}^k E_j.
\]
After passing to a subsequence, we may assume that \(x_i'\in E_j\) for some fixed \(j\). Hence \(z\in Z_0(E_j)\), proving the desired inclusion.

It remains to show that the union is disjoint. Suppose, toward a contradiction, that for some \(i\ne j\),
\[
    z\in Z_0(E_i)\cap Z_0(E_j).
\]
Then there exist two sequences
\[
    x_l^*=(x_l,t_l),
    \qquad
    y_l^*=(y_l,s_l),
\]
both converging to \(z\) with respect to \(d^*\), such that
\[
    x_l':=\psi^{t_l}(x_l)\in E_i,
    \qquad
    y_l':=\psi^{s_l}(y_l)\in E_j.
\]
After passing to a subsequence, we may assume that $t_l \ge s_l$ for all $l$. Since $d^*((y_l, s_l), (y_l, t_l)) \le C(Y, C_0) \sqrt{t_l-s_l}$ and $\psi^{t_l}(y_l)\in E_j$, we may further assume that $s_l=t_l$ by considering $(y_l, t_l)$ instead of $(y_l,s_l)$.

Arguing as above, there exists a constant \(c_0>0\) such that, for all sufficiently large \(l\),
\begin{align} \label{eq:union1}
    \min\{f(x_l'),f(y_l')\}
    \ge
    \frac{c_0}{|t_l|}.
\end{align}
Since \(x_l'\) and \(y_l'\) lie in different ends, \eqref{eq:union1} implies that
\begin{align} \label{eq:union2}
    d_{g_{-1}}(x_l',y_l')
    \ge
    \frac{c_1}{\sqrt{|t_l|}}
\end{align}
for some constant \(c_1>0\).

On the other hand, since \(d^*(x_l^*,y_l^*)\to0\), we have
\begin{align*}
    d_{W_1}^{s}\bigl(\nu_{x_l^*;s},\nu_{y_l^*;s}\bigr)
    \xrightarrow{l\to\infty}0
\end{align*}
for a fixed \(s\in[-1,0)\), to be chosen below. By the scalar curvature bound, it follows that
\begin{align*}
    d_{g_s}(x_l,y_l)\le C\sqrt{|s|}.
\end{align*}
Equivalently,
\begin{align} \label{eq:union3}
    d_{g_{-1}}\bigl(\psi^s(x_l),\psi^s(y_l)\bigr)\le C
\end{align}
for all sufficiently large \(l\).

By the definition of \(\psi^t\), we have
\[
    \psi^{-|t_l|/|s|}\circ \psi^s=\psi^{t_l}.
\]
Set
\[
    a_l^s:=\psi^s(x_l),
    \qquad
    b_l^s:=\psi^s(y_l).
\]
After increasing \(L\) if necessary, we may assume that
\begin{align} \label{eq:union4}
    \max\{f(a_l^s),f(b_l^s)\}\le L.
\end{align}
Indeed, otherwise \eqref{eq:union3} would force \(x_l'\) and \(y_l'\) to lie in the same end.

Using
\[
    \partial_t f(\psi^t(x))
    =
    |t|^{-1}|\nabla f|^2
    \le
    |t|^{-1}f(\psi^t(x)),
\]
we obtain
\[
    f(\psi^t(x))\le |t|^{-1}f(x).
\]
Applying this estimate to \(a_l^s\) and \(b_l^s\), we get
\begin{align} \label{eq:union5}
    \max\{f(x_l'),f(y_l')\}
    \le
    \frac{|s|}{|t_l|}
    \max\{f(a_l^s),f(b_l^s)\}
    \le
    L\frac{|s|}{|t_l|}.
\end{align}
Together with the standard distance estimate for the potential function, \eqref{eq:union5} implies
\[
    d_{g_{-1}}(x_l',y_l')
    \le
    C\frac{\sqrt{L|s|}}{\sqrt{|t_l|}}.
\]
Choosing \(|s|>0\) sufficiently small gives a contradiction to \eqref{eq:union2}. Therefore
\[
    Z_0(E_i)\cap Z_0(E_j)=\emptyset
\]
for \(i\ne j\), completing the proof.
\end{proof}

\begin{definition}[Cylindrical end]\label{def:cylindricalend}
An end \(E\) of a Ricci shrinker \((M^n,g,f)\) is called \textbf{cylindrical} with respect to \(S^{n-1}/\Gamma\times\R\) if, for every \(\ep>0\), there exists a compact set \(K\subset M\) such that every \(x\in E\setminus K\) satisfies
\[
    \left|\scal(x)-\frac{n-1}{2}\right|\le \ep
\]
and \((M,g,x)\) is \(\ep\)-close to \(S^{n-1}/\Gamma\times\R\). Here \(\Gamma\le \mathrm{O}(n)\) is a finite group acting freely on \(S^{n-1}\), and \(S^{n-1}/\Gamma\times\R\) is equipped with the rescaled product metric whose scalar curvature is normalized to be \(1\).
\end{definition}

\subsection{Cylindrical ends and tangent-flow characterizations}

We next prove the following result, which is analogous to \cite[Theorem 1.6]{munteanu2019structure}.

\begin{proposition}\label{prop:cylindricalend}
Suppose that a tangent flow $(Z',d_{Z'},z',\t')$ at some point
$z\in Z_0(E_j)$ has a time $-1$ slice
\[
    (\RR'_{-1},g^{Z'}_{-1})
\]
that is isometric to $S^{n-1}/\Gamma\times\R$. Then $E_j$ is a cylindrical end with respect to
$S^{n-1}/\Gamma\times\R$.
\end{proposition}

\begin{proof}
Let $t_i=-2^{-i}$, and let $(x_i,t_i)$ be an $H_n$-center of $z$. By the strong uniqueness theorem \cite[Theorem 8.15]{fang-li-unique}, we have
\begin{align}\label{eq:cylin001}
    d_{g_{t_{i+1}}}(x_i,x_{i+1})\le C\sqrt{|t_i|} \quad \text{and} \quad \lim_{i \to \infty}\scal(x_i, t_i)|t_i|=\frac{n-1}{2}.
\end{align}

Equivalently, after pulling back by the shrinker diffeomorphisms, we have
\begin{align}\label{eq:distancecontrol}
    d_g\bigl(\psi^{t_{i+1}}(x_i),\psi^{t_{i+1}}(x_{i+1})\bigr)\le C \quad \text{and} \quad \lim_{i \to \infty}\scal_g(\psi^{t_{i}}(x_i))=\frac{n-1}{2}.
\end{align}
By the assumption on the tangent flow, if $i$ is sufficiently large, then for every
$t\in[t_i,t_{i+1}]$ the pointed manifold
\[
    \bigl(M,g,\psi^t(x_i)\bigr)
\]
is $\ep_i$-close to $S^{n-1}/\Gamma\times\R$, where $\ep_i \to 0$. By Definition~\ref{def:close}, there exists an open set $U_i^t$ around $\psi^t(x_i)$, whose geometry is modeled on $S^{n-1}/\Gamma\times\R$. Equation~\eqref{eq:distancecontrol} implies that $U_i^{t_{i+1}} \cap U_{i+1}^{t_{i+1}} \ne \emptyset$ and the axes of $U_i^{t_{i+1}}$ and $U_{i+1}^{t_{i+1}}$ are almost parallel on their intersection. Thus, $\bigcup_i\bigcup_{t \in [t_i,t_{i+1}]}U_i^t$ forms an end of $(M, g)$. Since $z\in Z_0(E_j)$, Lemma~\ref{lem:unionend} implies that this end must be $E_j$. Therefore, along the end $E_j$, the geometry is asymptotic to
$S^{n-1}/\Gamma\times\R$. Hence $E_j$ is cylindrical with respect to
$S^{n-1}/\Gamma\times\R$.
\end{proof}

Conversely, the following result follows immediately from Definition~\ref{def:cylindricalend}.

\begin{proposition}\label{prop:cylindricalend1}
Suppose that $E_j$ is a cylindrical end with respect to
$S^{n-1}/\Gamma\times\R$. Then every tangent flow
$(Z',d_{Z'},z',\t')$ at every point $z\in Z_0(E_j)$ has a time $-1$ slice
\[
    (\RR'_{-1},g^{Z'}_{-1})
\]
that is isometric to $S^{n-1}/\Gamma\times\R$.
\end{proposition}

We have the following rough classification of the ends.

\begin{itemize}
    \item $E_i$ is \textbf{conical}. That is, the curvature on $E_i$ decays quadratically. Equivalently, the Ricci curvature on $E_i$ tends to zero at infinity; see \cite[Theorem 1.1]{MW}.

    \item $E_i$ is \textbf{cylindrical} with respect to
    $S^{n-1}/\Gamma\times\R$. Since the entropy is assumed to be bounded below by $-Y$, there are only finitely many possible groups $\Gamma$ up to conjugacy.

    \item $E_i$ is \textbf{generic} if it is neither conical nor cylindrical.
\end{itemize}
By Propositions~\ref{prop:1splitting},~\ref{prop:cylindricalend}, and~\ref{prop:cylindricalend1}, the above end types can be characterized by tangent flows at the time-zero slice.

\begin{itemize}
    \item $E_i$ is conical if and only if $Z_0(E_i)$ is smooth.

    \item $E_i$ is cylindrical if and only if there exists a tangent flow
    $(Z',d_{Z'},z',\t')$ at some singular point $z\in Z_0(E_i)$ such that
    \[
        (\RR'_{-1},g^{Z'}_{-1})
    \]
    is isometric to $S^{n-1}/\Gamma\times\R$.

    \item $E_i$ is generic if it is neither conical nor cylindrical. If $n=4$, then $E_i$ is generic if and only if every tangent flow
    $(Z',d_{Z'},z',\t')$ at every singular point $z\in Z_0(E_i)$ has a time $-1$ slice
    \[
        (\RR'_{-1},g^{Z'}_{-1})
    \]
    that is isometric to one of the following models:
    \begin{align*}
        S^2\times\R^2,\qquad
        \mathbb{RP}^2\times\R^2,\qquad
        \bigl(S^2\times_{\mathbb Z_2}\R\bigr)\times\R.
    \end{align*}
\end{itemize}

The following characterization of conical and cylindrical ends follows from
\cite[Theorems~1.2 and~6.1]{liwang2024rigidity}. A similar argument appears in \cite[Theorem 5.4]{LW26}.

\begin{proposition}\label{prop:alternative}
There exists a constant
\[
    \bar\ep=\bar\ep(n,C_0,Y)>0
\]
such that the following holds. If there exists a point $x\in E_i$ such that
$(M,g,x)$ is $\bar\ep$-close to $S^{n-1}/\Gamma\times\R$, then $E_i$ is either conical or cylindrical with respect to
$S^{n-1}/\Gamma\times\R$.
\end{proposition}

Roughly speaking, if $(M,g,x)$ is $\bar\ep$-close to
$S^{n-1}/\Gamma\times\R$, then, for every $y\in\Sigma=\{f=f(x)\}$ and every
$t\in[-1,0)$, the pointed manifold
\[
    (M,g,\psi^t(y))
\]
is $\ep$-close to $S^{n-1}/\Gamma\times\R$, where $\ep>0$ is sufficiently small. If $\scal(y)$ is larger than $\frac{n-1}{2}$ by a detectable amount for some $y\in\Sigma$, then $\scal(\psi^t(y))$ would exceed $C_0$ for some $t>-1$, which is impossible. Thus, either $\scal(\psi^t(y))$ remains close to $\frac{n-1}{2}$ for all $y\in\Sigma$ and all $t\in[-1,0)$, which implies that $E_i$ is cylindrical, or else $\scal(\psi^t(y))$ decays like $|t|$, which implies that $E_i$ is conical.

\subsection{Canonical neighborhoods}

In this subsection, we consider a four-dimensional Ricci shrinker $(M^4,g,f)$ whose scalar curvature is bounded above by $C_0$ and whose entropy is bounded below by $-Y$. By \cite{munteanu2015geometry}, these assumptions imply that $(M,g)$ has uniformly bounded curvature.

As in Section~\ref{sec:generalshrinker}, let $E_1,\ldots,E_k$ be the ends of $M$, each having as boundary a connected component of $\{f=L\}$ for some sufficiently large constant $L>0$.
We consider the following spacetime region:
\begin{align*}
    \mathcal U(\sigma)
    =
    \{(x,t)\in M\times[-1,0)\mid F(x,t)\ge \sigma\}.
\end{align*}

\begin{theorem}[Canonical neighborhood theorem]\label{thm:canonical}
For every $\sigma>0$ and $\ep>0$, there exists a constant
\[
    \delta=\delta(\sigma,\ep,M,g,f)>0
\]
such that the following holds. If $(x,t)\in \mathcal U(\sigma)$ and
\[
    \scal(x,t)>\delta^{-2},
\]
then $(M,g_t,x)$ is $\ep$-close to a standard model $N^3\times \R$, where $N$ is a three-dimensional $\kappa$-solution. Moreover, $\na_{g_t} F/|\na_{g_t} F|_{g_t}$ is $\ep$-close to the unit splitting vector field on $\R$; see Definition~\ref{def:close}.
\end{theorem}

\begin{proof}
Suppose otherwise. Then there exist $\sigma>0$, $\ep>0$, and a sequence
\[
    x_i^*=(x_i,t_i)\in \mathcal U(\sigma)
\]
such that
\[
    r_i^{-2}:=\scal(x_i,t_i)\to+\infty,
\]
but either $(M,g_{t_i},x_i)$ is not $\ep$-close to $N^3\times\R$ for any three-dimensional $\kappa$-solution $N$, or $\na_{g_{t_i}}F/|\na_{g_{t_i}}F|_{g_{t_i}}$ is not $\ep$-close to the unit splitting vector field.

We perform the parabolic rescaling at $x_i^*$ by defining
\begin{align*}
    g_i(t)&=r_i^{-2}g_{t_i+r_i^2t},\\
    d_i^*&=r_i^{-1}d^*,\\
    \t_i&=r_i^{-2}(\t-t_i).
\end{align*}
By the weak compactness theorem \cite[Theorem 1.3]{FL1}, after passing to a subsequence, we have pointed Gromov--Hausdorff convergence
\begin{align*}
    \bigl(M\times(-\infty,0],d_i^*,x_i^*,\t_i\bigr)
    \xrightarrow[i\to\infty]{\pGHconvtext}
    (Z',d_{Z'},x',\t').
\end{align*}
Moreover, by \cite[Theorem 1.5]{FL1}, the limit $Z'$ admits a Ricci flow spacetime structure
\[
    (\RR',\t',\partial_{\t'},g^{Z'}_t)
\]
on its regular part, and the convergence is smooth on $\RR'$.

Set
\[
    s_i=F(x_i,t_i).
\]
By the assumption $x_i^*\in\mathcal U(\sigma)$, we have $s_i\ge \sigma$. Define
\[
    f_i(\cdot,t)
    =
    s_i^{-\frac12}r_i^{-1}
    \bigl(F(\cdot,t_i+r_i^2t)-s_i\bigr).
\]
As in the proof of Proposition~\ref{prop:1splitting}, using the identities analogous to
\eqref{eq:iden01}--\eqref{eq:iden03}, together with the type-I bound for the shrinker Ricci flow, we conclude that $f_i$ subconverges smoothly on $\RR'$ to a function $f_\infty$ satisfying
\begin{align*}
    \partial_{\t'} f_\infty=0,
    \qquad
    \nabla^2_{g^{Z'}} f_\infty=0,
    \qquad
    |\nabla_{g^{Z'}} f_\infty|_{g^{Z'}}^2=1.
\end{align*}
Thus $\nabla_{g^{Z'}}f_\infty$ is a parallel unit vector field on $\RR'$, and hence induces a splitting direction.

Since the singular set of $Z'$ has codimension at least four, see \cite[Theorem 1.10]{FL1}, and since the limit is four-dimensional and splits off a line, the singular set must be empty. Hence $Z'=\RR'$, and the limiting Ricci flow spacetime splits as
\[
    (\RR',g^{Z'}_t)
    =
    (\RR''\times\R,g''_t+g_E).
\]
The flow $(\RR'',g''_t)_{t\le 0}$ is a three-dimensional $\kappa$-solution. By the smooth convergence, this implies that, for all sufficiently large $i$, the pointed manifolds $(M,g_{t_i},x_i)$ are $\ep$-close to $N^3\times\R$ for some three-dimensional $\kappa$-solution $N$, and the normalization of $\na_{g_{t_i}} F$ is $\ep$-close to the unit vector field on the final splitting factor, contradicting our choice of $x_i^*$.

Therefore, the desired constant $\delta$ exists.
\end{proof}

\subsection{Sequential compactness and identification of the tangent space at infinity}

Next, we employ the same argument as in \cite{Li26} to prove sequential compactness of the time slices.

\begin{proposition}\label{prop:sequential}
For every sequence $\{t_j\}$ with $t_j\nearrow0$, after passing to a subsequence if necessary, we have
\begin{align} \label{eq:grsequence-4d}
    (M,d_{g_{t_j}},p)
    \xrightarrow[j\to\infty]{\pGHconvtext}
    (X_{\GH},d_{\GH},p_{\GH}),
\end{align}
where $d_{g_t}$ denotes the distance induced by $g_t$.
\end{proposition}

\begin{proof}
By \eqref{eq:distanceestimate}, it suffices to prove the following uniform covering property: for every $\sigma\in(0,1)$ and $\ep>0$, there exists a constant $C_\ep$ such that, for all sufficiently small $|t|$, the set
\[
    V_t(\sigma)
    :=
    \{x\in M\mid \sigma\le F(x,t)\le \sigma^{-1}\}
\]
admits an $\ep$-net with at most $C_\ep$ elements.

For each end $E_j$, set
\[
    E_{j,t}:=\{x\in M\mid \psi^t(x)\in E_j\}.
\]
For $t$ sufficiently close to $0$, we have
\[
    V_t(\sigma)\subset \bigcup_{j=1}^k E_{j,t}.
\]
We write
\[
    V_t^j(\sigma):=V_t(\sigma)\cap E_{j,t}.
\]
It suffices to prove the desired covering estimate for each $V_t^j(\sigma)$. We consider the following three cases.

\medskip

\noindent\textbf{Case 1: \(E_j\) is conical.}
In this case, there exists a constant $C$ independent of $t$ such that
\[
    |\Rm|(x,t)\le C
\]
for all $x\in V_t^j(\sigma)$. Since every Ricci shrinker is $\kappa$-noncollapsed, see \cite[Theorem 22]{liwang2020heat1}, a standard ball-packing argument shows that $V_t^j(\sigma)$ admits an $\ep$-net with at most
\[
    C\sigma^{-4}\ep^{-4}
\]
elements, where $C$ is independent of $t$.

\medskip

\noindent\textbf{Case 2: \(E_j\) is cylindrical with respect to \(S^3/\Gamma\times\R\).}
In this case, for every $\ep_1>0$, if $|t|$ is sufficiently small, then
\[
    (M,g_t,x)
\]
is $\ep_1$-close to $S^3/\Gamma\times\R$ for every $x\in V_t^j(\sigma)$. As in the proof of Theorem~\ref{thm:canonical}, the vector field $\nabla F$ induces an almost splitting direction. Hence $\nabla F$ is almost parallel to the axial direction of each cylindrical model $S^3/\Gamma\times\R$. It follows that, for all sufficiently small $|t|$, the set $V_t^j(\sigma)$ admits an $\ep$-net with at most
\[
    C\sigma^{-1}\ep^{-1}
\]
elements.

\medskip

\noindent\textbf{Case 3: \(E_j\) is generic.}
In this case, fix $\ep_1>0$ to be a small universal constant. By Theorem~\ref{thm:canonical}, if $x\in V_t^j(\sigma)$ and $\scal(x,t)$ is sufficiently large, then
\[
    (M,g_t,x)
\]
is $\ep_1$-close to $N^3\times\R$, where $N$ is a three-dimensional $\kappa$-solution.

We first claim that, provided $\ep_1$ is sufficiently small, $N$ cannot be a standard spherical space form $S^3/\Gamma$. Indeed, otherwise, by the self-similarity of the shrinker, there would exist a point $x'\in E_j$ such that
\[
    (M,g,x')
\]
is $\ep_1$-close to $S^3/\Gamma\times\R$. Proposition~\ref{prop:alternative} would then imply that $E_j$ is either conical or cylindrical, contradicting the assumption that $E_j$ is generic.

Therefore, for every $x\in V_t^j(\sigma)$ with sufficiently large $\scal(x,t)$, the pointed manifold $(M,g_t,x)$ is $\ep_1$-close to $N^3\times\R$, where, up to finite quotient, $N$ is one of the following:
\[
    S^2\times\R,\qquad
    \text{the Bryant steady soliton } \mathrm{Br}^3,\qquad
    \text{or Perelman's ancient oval on } S^3.
\]
Here, we may assume that the last ancient oval, with respect to the based point, is not close to the standard $S^3$, $\RP^3$ or $\mathrm{Br}^3$.

The ancient oval case can also be excluded (after adjusting $\ep_1$ if necessary). Indeed, if this case occurred, then there would exist some $t'>t$ such that
\[
    (M,g_{t'},x)
\]
is $\ep_2$-close to the standard $S^3\times\R$ or to the standard $\mathbb{RP}^3\times\R$, where $\ep_2 \to 0$ as $\ep_1 \to 0$. This follows because $\ep$-closeness at a given time slice implies closeness at later times, by the uniqueness of the Ricci flow \cite{chenzhu2006uniqueness} and a standard limiting argument. Since $\psi^{t'}(x) \in E_j$, arguing as above, this would again contradict Proposition~\ref{prop:alternative} and the genericity of $E_j$.

We now consider the spacetime subset
\[
    V(\sigma)
    :=
    \{x^*=(x,t)\mid t\in[-1,0),\ x\in V_t(\sigma)\}.
\]
We first show that
\[
    V(\sigma)\subset B^*(\bar p,C_3)
\]
for some constant $C_3>0$. Indeed, $(p,t)$ is an $H$-center of $\bar p$, where $H$ depends only on $Y$; see \cite[Corollary 5.6]{liwang2024heat}. Hence \cite[Lemma~6.14]{FL1} implies that
\[
    d^*(\bar p,(p,t))\le C_1
\]
for all $t\in[-1,0)$, where $C_1$ depends only on the entropy. On the other hand, by \eqref{eq:distanceestimate} and \cite[Lemma~6.4]{FL1}, for every $x^*=(x,t)\in V(\sigma)$ we have
\[
    d^*((p,t),x^*)\le C_2,
\]
where $C_2$ depends only on $\sigma$ and the entropy. Therefore
\[
    V(\sigma)\subset B^*(\bar p,C_1+C_2).
\]
We fix $C_3=C_1+C_2$.

It follows from \cite[Theorem 6.31]{fang-li2} that
\begin{align} \label{eq:scalarl1}
    \int_{V_t^j(\sigma)} |\scal|\,\d V_{g_t}
    \le
    \int_{B^*(\bar p,C_3)\cap (M\times\{t\})}
    |\scal|\,\d V_{g_t}
    \le C_4,
\end{align}
where $C_4$ is independent of $t$. Although \cite[Theorem 6.31]{fang-li2} is stated for four-dimensional closed Ricci flows, the proof generalizes verbatim to four-dimensional Ricci flows with bounded curvature on each compact time interval; see also \cite[Remark 6.32]{fang-li2}.

Next, let
\[
    \gamma(s)\subset V_t^j(\sigma),\qquad s\in[\sigma,\sigma^{-1}],
\]
be an integral curve of $\nabla F/|\na F|^2$. Here, we parametrize $\gamma$ so that $F(\gamma(s), t)=s$. We cover $\gamma$ by a family of balls
\[
    \{B_{g_t}(y^a,\eta\ep)\}_{1\le a\le N},
\]
where $y^a\in\gamma$ are consecutive and
\[
    N\le C\sigma^{-1}\eta^{-1}\ep^{-1}.
\]
Here $\eta>0$ is a small universal constant to be chosen later. Define the measure
\[
    \mu_t(A)
    :=
    \int_A (|\scal|+1)\,\d V_{g_t}.
\]
For $s\in[\sigma,\sigma^{-1}]$, define
\[
l_t(s):=\max_{y \in \{F(\cdot, t)=s\} \cap V_t^j(\sigma)} d_{g_{t, s}}(y, \gamma(s)),
\]
where $g_{t, s}$ is the metric on $F(\cdot, t)=s$ induced by $g_t$. Note that $l_t(s)$ is finite by the connectedness of $\{F(\cdot, t)=s\} \cap V_t^j(\sigma)$.

\medskip

\noindent\textbf{Claim.}
For every $x\in V_t^j(\sigma)$ with $s=F(x,t)$, if $l_t(s)\ge \ep$, then
\[
    \mu_t(B_{g_t}(x,2\ep))\ge c_\ep,
\]
where $c_\ep>0$ is independent of $t$ and $x$.
\medskip

We follow the proof of \cite[Proposition 3.2]{Li26}. There are two key inputs in that proof. The first is the canonical neighborhood theorem, which is supplied here by Theorem~\ref{thm:canonical}. Moreover, in the present generic case we have already excluded the spherical space-form models $S^3/\Gamma\times\R$. Hence, every sufficiently high-curvature point is modeled on one of
\[
    S^2\times\R^2,\qquad
    \mathbb{RP}^2\times\R^2,\qquad
    \bigl(S^2\times_{\mathbb Z_2}\R\bigr)\times\R,\qquad \text{or} \qquad
    \mathrm{Br}^3 \times \R.
\]
The second input is the scalar curvature estimate \eqref{eq:scalarl1}.

For the reader's convenience, we sketch the argument. There are two subcases. If the scalar curvature $\scal(y, t) \le \bar r^{-2}$ for some $y \in B_{g_t}(x,\ep)$, where $\bar r$ is a small universal constant, then by the canonical neighborhood theorem (Theorem~\ref{thm:canonical}), $\scal(\cdot, t) \le \eta^{-2} \bar r^{-2}$ on $B_{g_t}(y,\eta \bar r)$. Using the $\kappa$-noncollapsing property, we obtain a uniform volume lower bound for $B_{g_t}(y,\ep)$ and hence for $B_{g_t}(x,2\ep)$. This yields a uniform lower bound for
\[
    \mu_t(B_{g_t}(x,2\ep)).
\]
This corresponds to Case 1 of the proof of \cite[Proposition 3.2]{Li26}.

Otherwise, every point in $B_{g_t}(x,\ep)$ has large scalar curvature greater than $\bar r^{-2}$. By Theorem~\ref{thm:canonical} and the discussion above, $B_{g_t}(x,\ep)$ is then covered by pieces modeled on
\[
    S^2\times\R^2,\qquad
    \mathbb{RP}^2\times\R^2,\qquad
    \bigl(S^2\times_{\mathbb Z_2}\R\bigr)\times\R,\qquad \text{or} \qquad
    \mathrm{Br}^3 \times \R.
\]
Note that Theorem~\ref{thm:canonical} indicates that the last $\R$-factor is induced by a rescaling of $\nabla F$. Therefore, the level set
\[
   \{F(\cdot, t)=s\} \cap B_{g_t}(x,\ep)
\]
is covered by pieces modeled on
\[
    S^2\times\R,\qquad
    \mathbb{RP}^2\times\R,\qquad
    S^2\times_{\mathbb Z_2}\R,\qquad \text{or} \qquad
    \mathrm{Br}^3.
\]

Since $l_t(s) \ge \ep$, the same argument as in Case 2 of the proof of \cite[Proposition 3.2]{Li26} then gives
\begin{align} \label{eq:typec01a}
\int_{\{F(\cdot, t)=s\}\cap B_{g_t}(x,\ep)} |\scal| \,\d\mathscr H^{3}_{g_t} \ge c_\ep,
\end{align}
where $\mathscr H^{3}_{g_t}$ is the $3$-dimensional Hausdorff measure induced by $g_t$. More precisely, by the definition of $l_t(s)$ and our assumption, there exists $x' \in \{F(\cdot, t)=s\} \cap V_t^j(\sigma)$ such that $d_{g_{t,s}}(x', \gamma(s)) \ge \ep$. From
\[
d_{g_{t,s}}(x, x')+d_{g_{t,s}}(x, \gamma(s)) \ge d_{g_{t,s}}(x', \gamma(s)) \ge \ep,
\]
we conclude that either $d_{g_{t,s}}(x, x')$ or $d_{g_{t,s}}(x, \gamma(s))$ is greater than $\ep/2$.

Without loss of generality, assume that $d_{g_{t,s}}(x,\gamma(s))\ge \ep/2$; the other case is analogous. Choose a unit-speed curve $\xi(\tau)$ on $\{F(\cdot,t)=s\}\cap V_t^j(\sigma)$ from $x$ to $\gamma(s)$ whose length is at least $\ep/2$. In particular, $\xi(\tau)$, for $\tau \in [0, \ep/2]$, is contained in $\{F(\cdot, t)=s\}\cap B_{g_t}(x,\ep)$. Thus, by analyzing the canonical models along $\xi$, the integral
\[
\int_{U_{\xi}} |\scal| \,\d\mathscr H^{3}_{g_t}
\]
is bounded below by the length of $\xi$, where $U_{\xi}$ is a tubular neighborhood of $\xi(\tau),\,\tau \in [0, \ep/2]$ in $\{F(\cdot, t)=s\}\cap B_{g_t}(x,\ep)$.

On the other hand, since $t^2 \scal+|\na F|^2=F$ and $|t| \scal \le C_0$, we have
\begin{align} \label{eq:typec01ab}
C_1^{-1} \le |\na F| \le C_1
\end{align}
on $V_t^j(\sigma)$. In addition, by $|t| \Ric+\na^2 F=g/2$ and the fact that $|t||\Ric| \le C_2$, we conclude that
\begin{align} \label{eq:typec01ac}
|\na^2 F| \le C_3 \quad \text{on} \quad V_t^j(\sigma).
\end{align}

For every $s'\in[s-\eta\ep,s+\eta\ep]$, let $\xi^{s'}$ to be the curve in $\{F(\cdot, t)=s'\}$, obtained by flowing $\xi$ along the vector field $\na F/|\na F|^2$. Then it follows from \eqref{eq:typec01ab} and \eqref{eq:typec01ac} that the length of $\xi^{s'}$ is at least $\ep/4$, provided that $\eta$ is sufficiently small.

Thus, for every such $s'$, a similar argument yields
\begin{align} \label{eq:typec01c}
\int_{\{F(\cdot, t)=s'\}\cap B_{g_t}(x,\ep)} |\scal| \,\d\mathscr H^{3}_{g_t} \ge c_\ep.
\end{align}
Using the coarea formula and \eqref{eq:typec01ab}, we conclude from \eqref{eq:typec01c} that
\begin{align} \label{eq:typec01}
    \mu_t(B_{g_t}(x,2\ep))
    \ge
    \int_{B_{g_t}(x,\ep)} |\scal|\,\d V_{g_t}
    \ge
    c_\ep.
\end{align}
This proves the claim.

Next, we show that $V_t^j(\sigma)$ admits a $4\ep$-net with at most $C_\ep$ elements. We choose a maximal $4\ep$-separated subset $\{x_i\}_{1 \le i \le N'}$ of $V_t^j(\sigma)$. We say that $x_i$ is of type (i) if
\[
l_t(F(x_i, t)) \ge \ep.
\]
Otherwise, we say that $x_i$ is of type (ii). Without loss of generality, we assume that $\{x_i\}_{1 \le i \le N''}$ are of type (i) and $\{x_i\}_{N''+1 \le i \le N'}$ are of type (ii).

The claim and \eqref{eq:scalarl1} imply that $N'' \le C_\ep$. Moreover, for each $x_i$ with $N''+1\le i\le N'$, we can choose $y_a$ so that $d_{g_t}(x_i, y_a)<2\ep$, provided that $\eta$ is small. Thus, $\{x_i\}_{1 \le i \le N''} \cup \{y_a\}_{1 \le a \le N}$ form a $4\ep$-net with at most $C_\ep$ elements. This completes the proof in Case 3.

Combining the three cases and summing over the finitely many ends $E_1,\ldots,E_k$, we obtain a uniform $\ep$-net for $V_t(\sigma)$ with at most $C_\ep$ elements. Hence the time slices are sequentially precompact in the pointed Gromov--Hausdorff topology, and the proposition follows.
\end{proof}

As before, we identify the possible pointed Gromov--Hausdorff limit. We assume that
\[
    (X_{\GH},d_{\GH},p_{\GH})
\]
is a sequential pointed Gromov--Hausdorff limit obtained in Proposition~\ref{prop:sequential} along a sequence $\{t_i\}$ with $t_i\nearrow0$. As in Proposition~\ref{prop:embh}, we obtain an isometric embedding from $(Z_0,d^Z_0)$ into $(X_{\GH},d_{\GH})$. Since $(Z_0,d^Z_0)$ is complete, its isometric image is closed in $(X_{\GH},d_{\GH})$; hence $(Z_0,d^Z_0)$ is locally compact. Proposition~\ref{prop:local-compactness-pgh} therefore gives the following identification of the pointed Gromov--Hausdorff limit, which implies Theorem~\ref{thm:4d}.

\begin{theorem}
\label{thm:tangentunique}
The pointed Gromov--Hausdorff limit of
\(
    (M,d_{g_t},p)
\)
as \(t\nearrow0\) exists and is isometric to \((Z_0,d^Z_0,\bar p)\). Therefore, the tangent space at infinity of $(M,g)$ exists, is unique, and is isometric to $(Z_0,d^Z_0)$.
\end{theorem}

\begingroup
\small
\setlength{\labelsep}{0.35em}
\bibliographystyle{alpha_nobreak}
\bibliography{ref_master_format_updated}

@book{McDuffSalamon2012,
  author    = {Dusa McDuff and Dietmar Salamon},
  title     = {J-Holomorphic Curves and Symplectic Topology},
  edition   = {2},
  series    = {American Mathematical Society Colloquium Publications},
  volume    = {52},
  publisher = {American Mathematical Society},
  address   = {Providence, RI},
  year      = {2012},
  isbn      = {978-0-8218-8746-2}
}

@article{bamler3,
  title = {Structure theory of non-collapsed limits of {Ricci} flows},
  author = {Bamler, R. H.},
  journal = {arXiv:2009.03243},
  year = {2020}
}

@incollection{Cao96,
  title = {Existence of gradient {K{\"a}hler--Ricci} solitons},
  author = {Cao, H.-D.},
  booktitle = {Elliptic and Parabolic Methods in Geometry},
  pages = {1--16},
  publisher = {A K Peters},
  address = {Wellesley, MA},
  year = {1996}
}

@article{CaoZhou,
  title = {On complete gradient shrinking {Ricci} solitons},
  author = {Cao, H.-D. and Zhou, D.},
  journal = {Journal of Differential Geometry},
  volume = {85},
  number = {2},
  pages = {175--186},
  year = {2010}
}

@article{Chen2009Strong,
  title = {Strong uniqueness of the {Ricci} flow},
  author = {Chen, B.-L.},
  journal = {Journal of Differential Geometry},
  volume = {82},
  number = {2},
  pages = {363--382},
  year = {2009},
  doi = {10.4310/jdg/1246888488}
}

@article{CLY,
  title = {A necessary and sufficient condition for {Ricci} shrinkers to have positive {AVR}},
  author = {Chow, B. and Lu, P. and Yang, B.},
  journal = {Proceedings of the American Mathematical Society},
  volume = {140},
  number = {6},
  pages = {2179--2181},
  year = {2012}
}

@article{Cifarelli25,
  title = {Explicit complete {Ricci}-flat metrics and {K{\"a}hler--Ricci} solitons on direct sum bundles},
  author = {Cifarelli, C.},
  journal = {arXiv:2410.23645},
  year = {2024}
}

@article{CCD,
  title = {On finite time {Type I} singularities of the {K{\"a}hler--Ricci} flow on compact {K{\"a}hler} surfaces},
  author = {Cifarelli, C. and Conlon, R. J. and Deruelle, A.},
  journal = {Journal of the European Mathematical Society},
  volume = {28},
  number = {2},
  pages = {463--504},
  year = {2026},
  doi = {10.4171/JEMS/1485}
}

@article{ColdingDeLellis2003,
  title = {The min-max construction of minimal surfaces},
  author = {Colding, T. H. and De Lellis, C.},
  journal = {Surveys in Differential Geometry},
  volume = {8},
  pages = {75--107},
  year = {2003},
  doi = {10.4310/SDG.2003.v8.n1.a3}
}

@article{CT,
  title = {{K{\"a}hler} currents and null loci},
  author = {Collins, T. C. and Tosatti, V.},
  journal = {Inventiones Mathematicae},
  volume = {202},
  number = {3},
  pages = {1167--1198},
  year = {2015}
}

@article{CSunD,
  title = {Classification results for expanding and shrinking gradient {K{\"a}hler--Ricci} solitons},
  author = {Conlon, R. J. and Deruelle, A. and Sun, S.},
  journal = {Geometry \& Topology},
  volume = {28},
  number = {1},
  pages = {267--351},
  year = {2024}
}

@article{DancerWang11,
  title = {On {Ricci} solitons of cohomogeneity one},
  author = {Dancer, A. S. and Wang, M. Y.},
  journal = {Annals of Global Analysis and Geometry},
  volume = {39},
  number = {3},
  pages = {259--292},
  year = {2011},
  doi = {10.1007/s10455-010-9233-1}
}

@article{EndersMullerTopping,
  title = {On {Type I} singularities in {Ricci} flow},
  author = {Enders, J. and M{\"u}ller, R. and Topping, P. M.},
  journal = {Communications in Analysis and Geometry},
  volume = {19},
  number = {5},
  pages = {905--922},
  year = {2011}
}

@article{esparza2025,
  title = {Uniqueness of asymptotically conical shrinking gradient {K{\"a}hler--Ricci} solitons},
  author = {Esparza, C.},
  journal = {arXiv:2502.13521},
  year = {2025}
}

@article{fang-li-unique,
  title = {Strong uniqueness of tangent flows at cylindrical singularities in {Ricci} flow},
  author = {Fang, H. and Li, Y.},
  journal = {arXiv:2510.20320},
  year = {2025}
}

@article{fang-li2,
  title = {Singular sets in noncollapsed {Ricci} flow limit spaces},
  author = {Fang, H. and Li, Y.},
  journal = {arXiv:2510.26317},
  year = {2025}
}

@article{FL1,
  title = {On the structure of noncollapsed {Ricci} flow limit spaces},
  author = {Fang, H. and Li, Y.},
  journal = {arXiv:2510.12398},
  year = {2025}
}

@article{FIK03,
  title = {Rotationally symmetric shrinking and expanding gradient {K{\"a}hler--Ricci} solitons},
  author = {Feldman, M. and Ilmanen, T. and Knopf, D.},
  journal = {Journal of Differential Geometry},
  volume = {65},
  number = {2},
  pages = {169--209},
  year = {2003},
  doi = {10.4310/jdg/1090511686}
}

@article{FutakiWang11,
  title = {Constructing {K{\"a}hler--Ricci} solitons from {Sasaki--Einstein} manifolds},
  author = {Futaki, A. and Wang, M.-T.},
  journal = {Asian Journal of Mathematics},
  volume = {15},
  number = {1},
  pages = {33--52},
  year = {2011},
  doi = {10.4310/AJM.2011.v15.n1.a3}
}

@article{Futaki21,
  title = {Irregular {Eguchi--Hanson} type metrics and their soliton analogues},
  author = {Futaki, A.},
  journal = {Pure and Applied Mathematics Quarterly},
  volume = {17},
  number = {1},
  pages = {27--53},
  year = {2021},
  doi = {10.4310/PAMQ.2021.v17.n1.a2}
}

@article{grauert,
  title = {{\"U}ber {Modifikationen} und exzeptionelle analytische {Mengen}},
  author = {Grauert, H.},
  journal = {Mathematische Annalen},
  volume = {146},
  number = {4},
  pages = {331--368},
  year = {1962}
}

@article{guedj-to,
  title = {{K{\"a}hler} families of {Green}'s functions},
  author = {Guedj, V. and T{\^o}, T. D.},
  journal = {Journal de l'\'Ecole polytechnique---Math\'ematiques},
  volume = {12},
  pages = {319--339},
  year = {2025},
  doi = {10.5802/jep.291}
}

@article{GPSS2,
  title = {Sobolev inequalities on {K{\"a}hler} spaces},
  author = {Guo, B. and Phong, D. H. and Song, J. and Sturm, J.},
  journal = {arXiv:2311.00221},
  year = {2023}
}

@article{GPSS-ii,
  title = {Diameter estimates in {K{\"a}hler} geometry {II}: removing the small degeneracy assumption},
  author = {Guo, B. and Phong, D. H. and Song, J. and Sturm, J.},
  journal = {Mathematische Zeitschrift},
  volume = {308},
  number = {3},
  pages = {43},
  year = {2024},
  doi = {10.1007/s00209-024-03600-x}
}

@article{GPSS1,
  title = {Diameter estimates in {K{\"a}hler} geometry},
  author = {Guo, B. and Phong, D. H. and Song, J. and Sturm, J.},
  journal = {Communications on Pure and Applied Mathematics},
  volume = {77},
  number = {8},
  pages = {3520--3556},
  year = {2024},
  doi = {10.1002/cpa.22196}
}

@article{HM07,
  title = {On {Shokurov}'s rational connectedness conjecture},
  author = {Hacon, C. D. and McKernan, J.},
  journal = {Duke Mathematical Journal},
  volume = {138},
  number = {1},
  pages = {119--136},
  year = {2007},
  doi = {10.1215/S0012-7094-07-13813-4}
}

@article{Hamilton1995Formation,
  title = {The formation of singularities in the {Ricci} flow},
  author = {Hamilton, R. S.},
  journal = {Surveys in Differential Geometry},
  volume = {2},
  pages = {7--136},
  year = {1995}
}

@article{HaslhoferMueller2011,
  title = {A compactness theorem for complete {Ricci} shrinkers},
  author = {Haslhofer, R. and M{\"u}ller, R.},
  journal = {Geometric and Functional Analysis},
  volume = {21},
  number = {5},
  pages = {1091--1116},
  year = {2011},
  doi = {10.1007/s00039-011-0137-4}
}

@article{hwang,
  title = {An application of {Cartan}'s equivalence method to {Hirschowitz}'s conjecture on the formal principle},
  author = {Hwang, J.-M.},
  journal = {Annals of Mathematics},
  volume = {189},
  number = {3},
  pages = {979--1000},
  year = {2019}
}

@incollection{koiso1990,
  title = {On rotationally symmetric {Hamilton}'s equation for {K{\"a}hler--Einstein} metrics},
  author = {Koiso, N.},
  booktitle = {Recent Topics in Differential and Analytic Geometry},
  series = {Advanced Studies in Pure Mathematics},
  volume = {18-I},
  pages = {327--337},
  publisher = {Academic Press},
  address = {Boston, MA},
  year = {1990},
  doi = {10.2969/aspm/01810327}
}

@article{KW,
  title = {Rigidity of asymptotically conical shrinking gradient {Ricci} solitons},
  author = {Kotschwar, B. and Wang, L.},
  journal = {Journal of Differential Geometry},
  volume = {100},
  number = {1},
  pages = {55--108},
  year = {2015},
  doi = {10.4310/jdg/1427202764}
}

@article{Li10,
  title = {On rotationally symmetric {K{\"a}hler--Ricci} solitons},
  author = {Li, C.},
  journal = {arXiv:1004.4049},
  year = {2010}
}

@article{liwang2020heat1,
  title = {Heat kernel on {Ricci} shrinkers},
  author = {Li, Y. and Wang, B.},
  journal = {Calculus of Variations and Partial Differential Equations},
  volume = {59},
  number = {6},
  pages = {194},
  year = {2020},
  doi = {10.1007/s00526-020-01861-y}
}

@article{liwang2024heat,
  title = {Heat kernel on {Ricci} shrinkers ({II})},
  author = {Li, Y. and Wang, B.},
  journal = {Acta Mathematica Scientia},
  volume = {44},
  number = {5},
  pages = {1639--1695},
  year = {2024},
  doi = {10.1007/s10473-024-0502-7}
}

@article{liwang2024rigidity,
  title = {Rigidity of the round cylinders in {Ricci} shrinkers},
  author = {Li, Y. and Wang, B.},
  journal = {Journal of Differential Geometry},
  volume = {127},
  number = {2},
  pages = {817--897},
  year = {2024},
  doi = {10.4310/jdg/1717772425}
}

@article{Li26,
  title = {{Gromov--Hausdorff} convergence of time-slices of singular {Ricci} flows in dimension three},
  author = {Li, Y.},
  journal = {arXiv:2606.13301},
  year = {2026}
}

@article{LW26,
  title = {On {K{\"a}hler} {Ricci} shrinker surfaces},
  author = {Li, Y. and Wang, B.},
  journal = {Acta Mathematica},
  volume = {236},
  number = {1},
  pages = {1--50},
  year = {2026},
  doi = {10.4310/ACTA.2026.v236.n1.a1}
}

@article{munteanu2012analysis,
  title = {Analysis of weighted {Laplacian} and applications to {Ricci} solitons},
  author = {Munteanu, O. and Wang, J.},
  journal = {Communications in Analysis and Geometry},
  volume = {20},
  number = {1},
  pages = {55--94},
  year = {2012},
  doi = {10.4310/CAG.2012.v20.n1.a3}
}

@article{munteanu2015geometry,
  title = {Geometry of shrinking {Ricci} solitons},
  author = {Munteanu, O. and Wang, J.},
  journal = {Compositio Mathematica},
  volume = {151},
  pages = {2273--2300},
  year = {2015}
}

@article{MW,
  title = {Conical structure for shrinking {Ricci} solitons},
  author = {Munteanu, O. and Wang, J.},
  journal = {Journal of the European Mathematical Society},
  volume = {19},
  number = {11},
  pages = {3377--3390},
  year = {2017}
}

@article{munteanu2019structure,
  title = {Structure at infinity for shrinking {Ricci} solitons},
  author = {Munteanu, O. and Wang, J.},
  journal = {Annales scientifiques de l'\'Ecole normale sup\'erieure},
  volume = {52},
  number = {4},
  pages = {891--925},
  year = {2019},
  doi = {10.24033/asens.2400}
}

@article{Naber2010,
  title = {Noncompact shrinking four solitons with nonnegative curvature},
  author = {Naber, A.},
  journal = {Journal f{\"u}r die reine und angewandte Mathematik (Crelle's Journal)},
  volume = {645},
  pages = {125--153},
  year = {2010}
}

@article{NiWallach2008,
  title = {On a classification of gradient shrinking solitons},
  author = {Ni, L. and Wallach, N.},
  journal = {Mathematical Research Letters},
  volume = {15},
  number = {5},
  pages = {941--955},
  year = {2008}
}

@book{petersen2006,
  title = {Riemannian Geometry},
  author = {Petersen, P.},
  series = {Graduate Texts in Mathematics},
  volume = {171},
  publisher = {Springer},
  edition = {2nd},
  year = {2006},
  doi = {10.1007/978-0-387-29403-2}
}

@article{Scharrer2022,
  title = {Some geometric inequalities for varifolds on {Riemannian} manifolds based on monotonicity identities},
  author = {Scharrer, C.},
  journal = {Annals of Global Analysis and Geometry},
  volume = {61},
  pages = {691--719},
  year = {2022},
  doi = {10.1007/s10455-021-09822-0}
}

@article{SWein,
  title = {Interior derivative estimates for the {K{\"a}hler--Ricci} flow},
  author = {Sherman, M. and Weinkove, B.},
  journal = {Pacific Journal of Mathematics},
  volume = {257},
  number = {2},
  pages = {491--501},
  year = {2012}
}

@article{SunZhang,
  title = {{K{\"a}hler--Ricci} shrinkers and {Fano} fibrations},
  author = {Sun, S. and Zhang, J.},
  journal = {arXiv:2410.09661},
  year = {2024}
}

@article{TianZ1,
  title = {A new holomorphic invariant and uniqueness of {K{\"a}hler--Ricci} solitons},
  author = {Tian, G. and Zhu, X.},
  journal = {Commentarii Mathematici Helvetici},
  volume = {77},
  number = {2},
  pages = {297--325},
  year = {2002},
  doi = {10.1007/s00014-002-8341-3}
}

@article{vu2024,
  title = {Uniform diameter and non-collapsing estimates for {K{\"a}hler} metrics},
  author = {Vu, D.-V.},
  journal = {The Journal of Geometric Analysis},
  volume = {36},
  pages = {75},
  year = {2026},
  doi = {10.1007/s12220-026-02322-2}
}

@article{wang2025rigidity,
  title = {Rigidity and $\epsilon$-regularity theorems of {R}icci shrinkers},
  author = {Wang, J. and Wang, Y.},
  journal = {Calculus of Variations and Partial Differential Equations},
  volume = {64},
  number = {2},
  pages = {42},
  year = {2025},
  doi = {10.1007/s00526-024-02903-5}
}

@article{Yang12,
  title = {A characterization of noncompact {Koiso}-type solitons},
  author = {Yang, B.},
  journal = {International Journal of Mathematics},
  volume = {23},
  number = {5},
  pages = {1250054},
  year = {2012},
  doi = {10.1142/S0129167X12500541}
}

@article{chenzhu2006uniqueness,
title={Uniqueness of the {R}icci flow on complete noncompact manifolds},
author={Chen, B.-L. and Zhu, X.-P.},
journal={Journal of differential geometry},
volume={74},
number={1},
pages={119--154},
year={2006},
publisher={Lehigh University}
}
\endgroup

\vspace{10pt}

Yu Li, Institute of Geometry and Physics, University of Science and Technology of China, No. 96 Jinzhai Road, Hefei, Anhui Province, 230026, China; Hefei National Laboratory, No. 5099 West Wangjiang Road, Hefei, Anhui Province, 230088, China; Email: yuli21@ustc.edu.cn.

\vspace{6pt}

Junsheng Zhang, Courant Institute of Mathematical Sciences, New York University, 251 Mercer St, New York, NY 10012; Email: jz7561@nyu.edu.

\end{document}